\documentclass{amsart}
\usepackage{amsbsy,amssymb,amscd,amsfonts,latexsym,amstext,delarray, amsmath,graphicx,color,caption,enumerate,amsaddr,lscape,xcolor,bm}
\usepackage{graphicx}
\usepackage{caption}
\usepackage{subcaption}
\usepackage{epsfig}
\usepackage{tikz}
\usetikzlibrary{decorations.markings,calc,cd}
\input xy

\xyoption{all}
\usepackage{euscript}
\usepackage{multirow}
\usepackage{etex, pictexwd,dcpic}
\usepackage[makeroom]{cancel}

\allowdisplaybreaks

\newtheorem{theorem}{Theorem}
\newtheorem{definition}[theorem]{Definition}
\newtheorem{proposition}[theorem]{Proposition}
\newtheorem{lemma}[theorem]{Lemma}
\newtheorem{corollary}[theorem]{Corollary}

\newtheorem{remark}[theorem]{Remark}

\definecolor{darkgreen}{rgb}{0, 0.5, 0}

\tikzset{->-/.style={decoration={
  markings,
  mark=at position .5 with {\arrow{>}}},postaction={decorate}}}
\tikzset{-<-/.style={decoration={
  markings,
  mark=at position .5 with {\arrow{<}}},postaction={decorate}}}
\tikzset{-->-/.style={decoration={
  markings,
  mark=at position .75 with {\arrow{>}}},postaction={decorate}}}
\tikzset{->--/.style={decoration={
  markings,
  mark=at position .25 with {\arrow{>}}},postaction={decorate}}}
\tikzset{->>-/.style={decoration={
  markings,
  mark=at position .57 with {\arrow{>>}}},postaction={decorate}
}}
\tikzset{->>>-/.style={decoration={
  markings,
  mark=at position .58 with {\arrow{>>>}}},postaction={decorate}
}}
\tikzset{->>>>-/.style={decoration={
  markings,
  mark=at position .62 with {\arrow{>>>>}}},postaction={decorate}
}}

\newcommand{\Spec}{\mathrm{Spec}\,}
\newcommand{\git}{\mathbin{
  \mathchoice{/\mkern-6mu/}
    {/\mkern-6mu/}
    {/\mkern-5mu/}
    {/\mkern-5mu/}}}

\newcommand{\bigslant}[2]{{\raisebox{.2em}{$#1$}\left/\raisebox{-.2em}{$#2$}\right.}}

\begin{document}
\title{Classical Limit of genus two DAHA}
\author{S.~Arthamonov}
\address{Department of Mathematics, University of Toronto}
\email{semeon.artamonov@utoronto.ca}
\begin{abstract}
We show that one-parameter deformation $\mathcal A_{q,t}$ of the skein algebra $Sk_q(\Sigma_2)$ of a genus two surface suggested in \cite{ArthamonovShakirov'2019} is flat. We solve the word problem in the algebra and describe monomial basis. In addition, we calculate the classical limit $\mathcal A_{q=1,t}$ of the algebra and prove that it is a one-parameter flat Poisson deformation of the coordinate ring $\mathcal A_{q=t=1}$ of an $SL(2,\mathbb C)$-charater variety of a genus two surface. As a byproduct, we obtain a remarkably simple presentation in terms of generators and relations for the coordinate ring $\mathcal A_{q=t=1}$ of a genus two character variety.
\end{abstract}
\maketitle

\section{Introduction}


With every surface $\Sigma$ and linear algebraic group $G$ we can associate an affine variety $\mathrm{Hom}(\pi_1(\Sigma),G)$ of representations of the fundamental group. The coordinate ring $\mathcal O(\mathrm{Hom}(\pi_1(\Sigma),G))$ of representation variety comes equipped with a $G$-action corresponding to simultaneous conjugation by an element $g\in G$. The spectrum of the $G$-invariant subring
\begin{align*}
\mathrm{Hom}(\pi_1(\Sigma),G)\git G:=\Spec(\mathcal O(\mathrm{Hom}(\pi_1(\Sigma),G))^G)
\end{align*}
is commonly referred to as \textit{$G$-character variety of $\Sigma$}.

When surface $\Sigma$ is oriented, the coordinate ring of $G$-character variety comes equipped with a natural Poisson bracket \cite{AtiyahBott'1983, Goldman'1986, GuruprasadHuebschmannJeffreyWeinstein'1997}. For $G=SL(2,\mathbb C)$ the quantization of this Poisson algebra is known as the skein algebra $Sk_q(\Sigma)$ of the surface \cite{Turaev'1991}. In other words, the skein algebra $Sk_q(\Sigma)$ and the commutative coordinate ring of $SL(2,\mathbb C)$-character variety share the same basis as vector spaces, however the structure constants of the skein algebra depend on the additional parameter $q$. At $q=1$ the structure constants of the skein algebra coincide with those of the commutative coordinate ring, while the main linear part of the decomposition of structure constants as a series in $(q-1)$ is controlled by the Poisson bracket.

For the case of a torus $\Sigma_1$, the skein algebra $Sk_q(\Sigma_1)$ can be further deformed into $A_1$ spherical Double Affine Hecke Algebra (DAHA) \cite{Cherednik'1995}. Now, spherical DAHA depends on two parameters $q,t$, where extra parameter $t$ corresponds to Macdonald deformation in the theory of orthogonal polynomials \cite{Macdonald'1979}. This brings up a natural question of study and classification of flat deformations of surface skein algebras up to isomorphism.

It is known that formal deformations are controlled by the Hochschild cohomology \cite{Gerstenhaber'1964}. This allows, in principle, to classify flat deformations of quantum character varieties up to isomorphism. The question is well-studied in the genus one case \cite{Oblomkov'2004,Sahi'1999, NoumiStockman'2000, Stockman'2003}, yet no examples of deformations with analytic dependence on parameter were constructed earlier beyond genus one.

In \cite{ArthamonovShakirov'2019} together with Sh.~Shakirov we have proposed an algebra $A_{q,t}$ of $q$-difference operators with two generic parameters $q,t$ and proved that the Mapping Class Group of a genus two surface acts by automorphisms of this algebra. It was subsequently proved by J.~Cooke and P.~Samuelson \cite{CookeSamuelson'2021} that $q=t$ specialization of this algebra is isomorphic to the genus two skein algebra $A_{q=t}\simeq Sk_q(\Sigma_2).$ Results of the current work will, in particular, imply that $A_{q,t}$ provides an example of nontrivial flat deformation of the genus two skein algebra.

\subsection{Outline}

In this paper we introduce the alternative definition of algebra $\mathcal A_{q,t}$ in terms of generators and relations and show that $\mathcal A_{q,t}\simeq A_{q,t}$ is isomorphic to the original algebra of $q$-difference operators from \cite{ArthamonovShakirov'2019}. Our presentation of $\mathcal A_{q,t}$ allows us to solve the word problem and prove that $\mathcal A_{q,t}$ is a flat two-parameter deformation of the commutative coordinate ring of the character variety
\begin{align*}
\mathcal A_{q=t=1}\simeq\mathcal O(\mathrm{Hom}(\pi_1(\Sigma_2),SL(2,\mathbb C)))^{SL(2,\mathbb C)}.
\end{align*}

We calculate the classical limit $q\rightarrow1$ of $\mathcal A_{q,t}$ and prove that resulting commutative algebra $\mathcal A_{q=1,t}$ is a flat Poisson deformation of $\mathcal A_{q=t=1}$. We show that this deformed commutative algebra $\mathcal A_{q=1,t}$ is also equipped with a $\mathrm{Mod}(\Sigma_2)$-action by Poisson automorphisms.

All four algebras involved:
\begin{align*}
\mathcal A_{q,t},\quad\mathcal A_{q=t}\simeq Sk_q(\Sigma_2),\quad\mathcal A_{q=1,t},\quad\mathcal A_{q=t=1}
\end{align*}
share the same monomial basis which we describe in the paper. In all four cases we construct an algorithm which brings the expression in generators to the canonical form\footnote{Realization of this algorithm in Mathematica can be found at \cite{Arthamonov-GitHub-Flat}}. For the two of commutative algebras, this simply amounts to calculating the Groebner basis of defining ideal for appropriate choice of monomial ordering. While, for the two noncommutative algebras we prove a PBW-like Theorem which allows us to relate monomial bases of $\mathcal A_{q,t}$ and $\mathcal A_{q=t}$ to their commutative counterparts. As a byproduct we find a remarkably simple presentation of the $SL(2,\mathbb C)$-classical character variety of a genus two surface which have not appeared in the literature.

In addition we prove that the Mapping Class Group $\mathrm{Mod}(\Sigma_2)$ of a closed genus two surface acts by automorphisms of $\mathcal A_{q,t}$ which induce Poisson automorphsims of $\mathcal A_{q=1,t}$ and $\mathcal A_{q=t=1}$. We show that the action of $\mathrm{Mod}(\Sigma_2)$ coincides with the natural action of the Mapping Class Group on the character variety $\mathcal A_{q=t=1}=\mathcal O(\mathrm{Hom}(\pi_1(\Sigma_2),SL(2,\mathbb C)))^{SL(2,\mathbb C)}$.

\subsection{Finite-dimensional representations and relation with TQFT}

Skein algebra $Sk_q(\Sigma_g)$ of a closed surface admits finite dimensional representations when parameter $q=\mathrm e^{\frac{2\pi\mathrm i}{k+2}},\;k\in\mathbb Z_{\geqslant0}$ is a root of unity. A family of representations for each $k\in\mathbb Z_{\geqslant0}$ can be constructed from $SU(2)$ Chern-Simons topological quantum field theory \cite{KirillovReshetikhin'1989, ReshetikhinTuraev'1990, Turaev'1994}.

The mapping class group $Mod(\Sigma_g)$ of the surface acts by automorphisms of the skein algebra $Sk_q(\Sigma_g)$. On the level of finite dimensional representations of $Sk_q(\Sigma_g)$ coming from TQFT, this action corresponds to simultaneous conjugation by matrices associated to mapping classes \cite{MassbaumRoberts'1993, MasbaumVogel'1994, Funar'1999} (See also \cite{Massbaum'2003} for a review). Thus, $SU(2)$ Chern-Simons theory gives rise to a family of projective representations of Mapping Class Group $\mathrm{Mod}(\Sigma_g)$ for all $g,k\in\mathbb Z_{\geqslant0}$.

On the other hand, $A_1$-spherical Double Affine Hecke Algebra, which realizes a one-parameter deformation of the genus one skein algebra, admits finite dimensional representations at a special values of parameters satisfying $q^kt^2=1$, where $k\in\mathbb Z_{\geqslant0}$ is a nonnegative integer \cite{Cherednik'2005}. In \cite{AganagicShakirov'2015} M.~Aganagic and Sh.~Shakirov have proposed a remarkable interpretation of these finite-dimensional representations from the point of view of Topological Quantum Field Theory. Namely, they have constructed a refinement of genus one algebra of knot operators of Chern-Simons topological quantum field theory and have computed the corresponding family of projective representations of the mapping class group of a torus $\mathrm{Mod}(\Sigma_2)\simeq SL(2,\mathbb Z)$.

This manuscript is a continuation of our joint project with Shamil Shakirov \cite{ArthamonovShakirov'2020, ArthamonovShakirov'2019} where we have studied Macdonald deformation of a TQFT representation of the genus two Mapping Class Group. In \cite{ArthamonovShakirov'2020} we have proposed a conjectural one-parameter deformation of the family of projective TQFT representations of $\mathrm{Mod}(\Sigma_2)$ labelled by Chern-Simons level $k\in\mathbb Z_{\geqslant0}$ and computed the corresponding algebras of knot operators. In \cite{ArthamonovShakirov'2019} we have proved that our formulas proposed in \cite{ArthamonovShakirov'2020} do define a family of projective representation of $\mathrm{Mod}(\Sigma_2)$ by introducing the algebra of $q$-difference operators $A_{q,t}$ equipped with an action of $\mathrm{Mod}(\Sigma_2)$ by automorphisms. The algebra of knot operators at level $k\in\mathbb Z_{\geqslant 0}$ provides a finite dimensional representation of $A_{q,t}$ at special value of parameters $q^kt^2=1$. Based on complete analogy with the genus one case we referred to $A_{q,t}$ as a \textit{genus two generalization of $A_1$-spherical Double Affine Hecke Algebra.}

\subsection{Representation in $q$-difference operators}

Algebra $A_{q,t}$ introduced in \cite{ArthamonovShakirov'2019} can be defined as an associative algebra, which is generated by six $q$-difference operators:
\begin{align}
A_{q,t}=\mathbf k\langle\hat O_{B_{12}}, \hat O_{B_{23}}, \hat O_{B_{13}}, \hat O_{A_1}, \hat O_{A_2}, \hat O_{A_3}\rangle
\label{eq:AqtDifferenceOperatorsAlgebra}
\end{align}
corresponding to six cycles on genus two surface shown on figure \ref{fig:qDiffGeneratingCycles}.
\begin{figure}
\includegraphics[width=250pt]{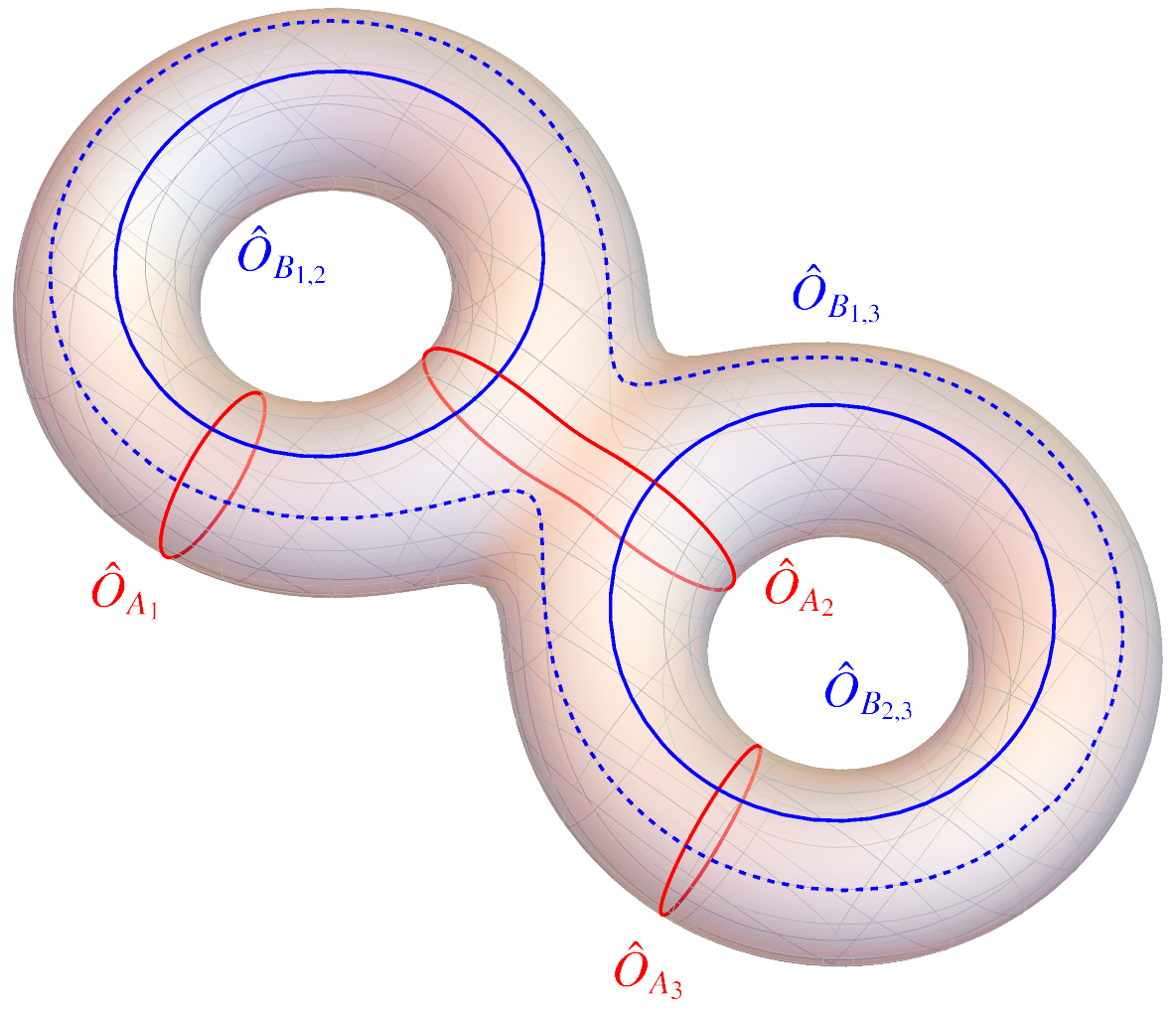}
\caption{Cycles corresponding to six $q$-difference operators}
\label{fig:qDiffGeneratingCycles}
\end{figure}
Here $\mathbf k=\mathbb C(q^{\frac14},t^{\frac14})$ is the ground field.\footnote{We use $q^{\frac14}$ and $t^{\frac14}$ instead of introducing new variables $Q=q^{\frac14}$ and $T=t^{\frac14}$ for consistency with earlier papers on Macdonald Theory and Double Affine Hecke Algebras. This simplifies comparison of familiar formulas and became standard in literature on the subject.}

The operators are acting on the subspace of Laurent polynomials in three variables
\begin{align*}
\mathcal H=\mathbf k[X_{12}+X_{12}^{-1},X_{23}+X_{23}^{-1},X_{13}+X_{13}^{-1}]=\mathbf k[X_{12}^{\pm1},X_{23}^{\pm1},X_{13}^{\pm1}]^{S_2\times S_2\times S_2}.
\end{align*}
The first three operators act by multiplication
\begin{subequations}
\begin{align}
\label{eq:Mult1}{\hat O}_{B_{12}} \ = \ X_{12} + X_{12}^{-1} \\[10pt]
\label{eq:Mult2}{\hat O}_{B_{13}} \ = \ X_{13} + X_{13}^{-1} \\[10pt]
\label{eq:Mult3}{\hat O}_{B_{23}} \ = \ X_{23} + X_{23}^{-1}
\end{align}
\label{eq:Mult}
\end{subequations}
while the remaining three operators have the following form
\begin{subequations}
\begin{align}
{\hat O}_{A_1} \ = \ \sum\limits_{a,b \in \{\pm 1\}} \ a b \ \dfrac{(1 - t^{\frac{1}{2}} X_{23} X_{12}^a X_{13}^b)(1 - t^{\frac{1}{2}} X_{23}^{-1} X_{12}^a X_{13}^b)}{t^{\frac{1}{2}} X_{12}^{a} X_{13}^b (X_{12} - X_{12}^{-1})(X_{13} - X_{13}^{-1})} \ {\hat \delta}_{12}^{a} {\hat \delta}_{13}^{b}
\label{eq:Hamiltonian1}\\[5pt]
{\hat O}_{A_2} \ = \ \sum\limits_{a,b \in \{\pm 1\}} \ a b \ \dfrac{(1 - t^{\frac{1}{2}} X_{13} X_{12}^a X_{23}^b)(1 - t^{\frac{1}{2}} X_{13}^{-1} X_{12}^a X_{23}^b)}{t^{\frac{1}{2}} X_{12}^{a} X_{23}^b (X_{12} - X_{12}^{-1})(X_{23} - X_{23}^{-1})} \ {\hat \delta}_{12}^{a} {\hat \delta}_{23}^{b}
\label{eq:Hamiltonian2}\\[5pt]
{\hat O}_{A_3} \ = \ \sum\limits_{a,b \in \{\pm 1\}} \ a b \ \dfrac{(1 - t^{\frac{1}{2}} X_{12} X_{13}^a X_{23}^b)(1 - t^{\frac{1}{2}} X_{12}^{-1} X_{13}^a X_{23}^b)}{t^{\frac{1}{2}} X_{13}^{a} X_{23}^b (X_{13} - X_{13}^{-1})(X_{23} - X_{23}^{-1})} \ {\hat \delta}_{13}^{a} {\hat \delta}_{23}^{b}
\label{eq:Hamiltonian3}
\end{align}
\label{eq:Hamiltonians}
\end{subequations}
where $\hat\delta_{12},\hat\delta_{23},$ and $\hat\delta_{13}$ are $q$-shift operators acting on polynomials in three variables as
\begin{align*}
\hat\delta_{12}f(X_{12},X_{23},X_{13}) &=f\left(q^{\frac12}X_{12},X_{23},X_{13}\right),\qquad \textrm{for all}\quad f\in\mathcal H,
\end{align*}
with the other two operators obtained by permutation of indices.



An important question which was left beyond the scope of \cite{ArthamonovShakirov'2019} was to determine a complete set of relations on six operators listed in (\ref{eq:Mult}) and (\ref{eq:Hamiltonians}). In this paper we answer this question from the opposite direction: Based on extensive calculation of relations between generators of $A_{q,t}$, we introduce a new algebra $\mathcal A_{q,t}$ in terms of generators and relations, where we can solve the word problem and prove that $\mathcal A_{q,t}$ is a flat deformation of the commutative coordinate ring of $SL(2,\mathbb C)$-character variety. We then prove that all relations in the defining ideal of $\mathcal A_{q,t}$ are satisfied in $A_{q,t}$. This allows us to introduce a representation of $\mathcal A_{q,t}$ in terms $q$-difference operators
\begin{align}
\psi:\mathcal A_{q,t}\rightarrow A_{q,t}.
\label{eq:QDifferencePresentation}
\end{align}
Finally, we prove that representation (\ref{eq:QDifferencePresentation}) is faithful and $\mathcal A_{q,t}\simeq A_{q,t}$.

Our presentation of $\mathcal A_{q,t}$ is carefully chosen to make the computation of the classical limit $q\rightarrow 1$ maximally straightforward. One of the key features of our presentation is the extended set of generators. Namely, we are using the following collection of $15$ elements of $A_{q,t}$ as a prototype for the generating set of $\mathcal A_{q,t}$:
\begin{subequations}
\begin{align}
\hat O_{1}=&\hat O_{A_1},& \hat O_2=&\hat O_{B_{12}},& \hat O_3=&\hat O_{A_2},& \hat O_4=&\hat O_{B_{23}},& \hat O_5=&\hat O_{A_3},& \hat O_6=&\hat O_{B_{13}},
\label{eq:qdifflevelonegenerators}\\
\hat O_{12}=&\big[\hat O_1,\hat O_2\big]_q,&
\hat O_{23}=&\big[\hat O_2,\hat O_3\big]_q,&
\hat O_{34}=&\big[\hat O_3,\hat O_4\big]_q,&
\hat O_{45}=&\big[\hat O_4,\hat O_5\big]_q,&
\hat O_{56}=&\big[\hat O_5,\hat O_6\big]_q,&
\hat O_{61}=&\big[\hat O_6,\hat O_1\big]_q,
\label{eq:qdiffleveltwogenerators}
\end{align}
\begin{equation}
\begin{aligned}
\hat O_{123}=&\big[\big[\hat O_1,\hat O_2\big]_q,\hat O_3\big]_q=\big[\big[\hat O_4,\hat O_5\big]_q,\hat O_6\big]_q,\\
\hat O_{234}=&\big[\big[\hat O_2,\hat O_3\big]_q,\hat O_4\big]_q=\big[\big[\hat O_5,\hat O_6\big]_q,\hat O_1\big]_q,\\
\hat O_{345}=&\big[\big[\hat O_3,\hat O_4\big]_q,\hat O_5\big]_q=\big[\big[\hat O_6,\hat O_1\big]_q,\hat O_2\big]_q,
\end{aligned}
\label{eq:qdifflevelthreegenerators}
\end{equation}
\label{eq:qdiff15generators}
\end{subequations}
where
\begin{align}
[A,B]_{q^j}=\frac{q^{\frac j4}AB-q^{-\frac j4}BA}{q^{\frac12}-q^{-\frac12}}.
\label{eq:qCommutatorDef}
\end{align}

\section{Generators and Relations}

For this section we take the ground field $\mathbf k=\mathbb C(q^{\frac14},t^{\frac14})$ to be a field of rational functions in $q^{\frac14}$ and $t^{\frac14}$.
\begin{definition}
Let $\mathcal A_{q,t}$ be an associative algebra over $\mathbf k$ with 15 generators
\begin{align}
\begin{array}{cccccc}
O_1,&O_2,&O_3,&O_4,&O_5,&O_6,\\
O_{12},&O_{23},&O_{34},&O_{45},&O_{56},&O_{61},\\
O_{123},&O_{234},&O_{345}
\end{array}
\label{eq:15Generators}
\end{align}
subject to:
\begin{itemize}
\setlength\itemsep{0.5em}
\item ``Normal ordering'' relations listed in Table \ref{tab:QCommRel}. Here $(\pm i|X)$ entry in the $O_J$'th row and $O_K$'th column of the table stands for relation
\begin{subequations}
\begin{align}
\left[O_J,O_K\right]_{q^{\pm i}}=\pm X,
\label{eq:NormalOrderingRelations}
\end{align}
while the zero entry simply stands for commutation relation $O_JO_K=O_KO_J$.
\smallskip

\item ``$q$-Casimir'' relation
\begin{equation}
\begin{aligned}
    \frac13\big(O_{123}O_{234}O_{345} + O_{234}O_{345}O_{123} + O_{345}O_{123}O_{234}\big)\\[5pt]
    -\frac{1}{6}q^{-\frac{1}{2}}(q^2+2)
    \big(O_1O_4O_{345} +O_2O_5O_{123} +O_3O_6O_{234} +O_4O_1O_{345} +O_5O_2O_{123} +O_6O_3O_{234}\big)\\[5pt]
    +\frac{1}{6}q^{-\frac{5}{4}}t^{-\frac{1}{2}}(2 q+1) (q+t)
    \big(O_1O_6O_{61} +O_2O_1O_{12} +O_3O_2O_{23} +O_4O_3O_{34} +O_5O_4O_{45} +O_6O_5O_{56}\big)\\[5pt]
    -\frac{1}{3}q^{\frac{1}{2}}\big(O_1O_3O_5 +O_2O_4O_6 +O_3O_5O_1 +O_4O_6O_2 +O_5O_1O_3 +O_6O_2O_4\big)\\[5pt]
    -\frac{1}{6}q^{-\frac{3}{2}}t^{-\frac{1}{2}}(2 q+1) (q+t)
    \big(O_{12}^2+O_{23}^2+O_{34}^2+O_{45}^2+O_{56}^2+O_{61}^2\big)\\[5pt]
    \frac{1}{6}q^{-\frac{3}{2}}t^{-\frac{1}{2}}(q-1) (q+t)
    \big(O_1^2+O_2^2+O_3^2+O_4^2+O_5^2+O_6^2\big)\\[5pt]
    q^{-\frac{3}{2}}t^{-\frac{3}{2}}(t+1) (q+t) \left(q^2+t\right)&=0
\end{aligned}
\label{eq:qCasimirRelation}
\end{equation}
\label{eq:DefiningRelationsAqt}
\end{subequations}
\end{itemize}
\label{def:Aqt}
\end{definition}

\begin{table}
\begin{equation*}
\begin{array}{c|cccccccc}
& O_1& O_2& O_3& O_4& O_5& O_6& O_{12}& O_{23}\\\hline
 O_1& * & * & * & * & * & * & * & * \\
 O_2& \text{-1$|$}  O_{12} & * & * & * & * & * & * & * \\
 O_3& 0 & \text{-1$|$}  O_{23} & * & * & * & * & * & * \\
 O_4& 0 & 0 & \text{-1$|$}  O_{34} & * & * & * & * & * \\
 O_5& 0 & 0 & 0 & \text{-1$|$}  O_{45} & * & * & * & * \\
 O_6& \text{1$|$}  O_{61} & 0 & 0 & 0 & \text{-1$|$}  O_{56} & * & * & * \\
 O_{12}&\text{1$|$}  O_2 & \text{-1$|$}  O_1 & \text{1$|$}  O_{123} & 0 & 0 & \text{-1$|$} O_{345} & * & * \\
 O_{23}& \text{-1$|$}  O_{123} & \text{1$|$}  O_3 & \text{-1$|$}  O_2 & \text{1$|$}  O_{234} & 0 & 0 & \text{2$|$} \frac{(q+1)  O_1 O_3}{\sqrt{q}}-\frac{ O_5 (q+t)}{\sqrt{q} \sqrt{t}} & * \\
 O_{34}& 0 & \text{-1$|$}  O_{234} & \text{1$|$}  O_4 & \text{-1$|$}  O_3 & \text{1$|$} O_{345} & 0 & \text{-1$|$}  O_{56} & \text{2$|$} \frac{(q+1)  O_2 O_4}{\sqrt{q}}-\frac{ O_6 (q+t)}{\sqrt{q} \sqrt{t}} \\
 O_{45}& 0 & 0 & \text{-1$|$}  O_{345} & \text{1$|$}  O_5 & \text{-1$|$}  O_4 & \text{1$|$} O_{123} & 0 & \text{-1$|$}  O_{61} \\
 O_{56}& \text{1$|$}  O_{234} & 0 & 0 & \text{-1$|$}  O_{123} & \text{1$|$}  O_6 & \text{-1$|$}  O_5 & \text{1$|$}  O_{34} & 0\\
 O_{61}& \text{-1$|$}  O_6 & \text{1$|$}  O_{345} & 0 & 0 & \text{-1$|$}  O_{234} & \text{1$|$}  O_1 & \text{-2$|$} \frac{(q+1)  O_2 O_6}{\sqrt{q}}-\frac{ O_4 (q+t)}{\sqrt{q} \sqrt{t}} & \text{1$|$}  O_{45} \\
 O_{123}& \text{1$|$}  O_{23} & 0 & \text{-1$|$}  O_{12} & \text{1$|$}  O_{56} & 0 & \text{-1$|$}  O_{45} & \text{1$|$}  O_3 & \text{-1$|$}  O_1 \\
 O_{234}& \text{-1$|$}  O_{56} & \text{1$|$}  O_{34} & 0 & \text{-1$|$}  O_{23} & \text{1$|$} O_{61} & 0 & \text{+0$|$} O_1 O_{34}- O_2 O_{56} & \text{1$|$}  O_4 \\
 O_{345}& 0 & \text{-1$|$}  O_{61} & \text{1$|$}  O_{45} & 0 & \text{-1$|$}  O_{34} & \text{1$|$}  O_{12} & \text{-1$|$}  O_6 & \text{+0$|$} O_2 O_{45} - O_3 O_{61}
\end{array}
\end{equation*}
\vspace{0.3cm}

\begin{equation*}
\begin{array}{c|cccc}
& O_{34}& O_{45}& O_{56}& O_{61}\\\hline
 O_{34}& * & * & * & * \\
 O_{45}& \text{2$|$} \frac{(q+1)  O_3 O_5}{\sqrt{q}}-\frac{ O_1 (q+t)}{\sqrt{q} \sqrt{t}} & * & * & * \\
 O_{56}& \text{-1$|$}  O_{12} & \text{2$|$} \frac{(q+1)  O_4 O_6}{\sqrt{q}}-\frac{ O_2 (q+t)}{\sqrt{q} \sqrt{t}} & * & * \\
 O_{61}& 0 & \text{-1$|$}  O_{23} & \text{2$|$} \frac{(q+1)  O_1 O_5}{\sqrt{q}}-\frac{ O_3 (q+t)}{\sqrt{q} \sqrt{t}} & * \\
 O_{123}& \text{+0$|$}  O_3 O_{56} - O_4 O_{12} & \text{1$|$}  O_6 & \text{-1$|$}  O_4 & \text{+0$|$} O_6 O_{23} - O_1 O_{45} \\
 O_{234}& \text{-1$|$}  O_2 & \text{+0$|$} O_4 O_{61} - O_5 O_{23} & \text{1$|$}  O_1 & \text{-1$|$}  O_5 \\
 O_{345}& \text{1$|$}  O_5 & \text{-1$|$}  O_3 & \text{+0$|$}  O_5 O_{12} - O_6 O_{34} & \text{1$|$}  O_2 \\
\end{array}
\end{equation*}\vspace{0.3cm}

\begin{align*}
[O_{123},O_{345}]_{q^2}=-\frac{(q+t)O_1O_5O_6}{\sqrt{q}\sqrt{t}} +\frac{(q+t)O_5O_{61}}{q^{\frac34}\sqrt{t}} +\frac{(q+t)O_1O_{56}}{\sqrt[4]{q}\sqrt{t}} -\frac{(q+t)O_{234}}{\sqrt{q}\sqrt{t}}+\frac{(q+1)O_3O_6}{\sqrt{q}}\\
[O_{234},O_{123}]_{q^2}=-\frac{(q+t)O_2O_6O_1}{\sqrt{q}\sqrt{t}} +\frac{(q+t)O_6O_{12}}{q^{\frac34}\sqrt{t}} +\frac{(q+t)O_2O_{61}}{\sqrt[4]{q}\sqrt{t}} -\frac{(q+t)O_{345}}{\sqrt{q}\sqrt{t}}+\frac{(q+1)O_4O_1}{\sqrt{q}}\\
[O_{345},O_{234}]_{q^2}=-\frac{(q+t)O_3O_1O_2}{\sqrt{q}\sqrt{t}} +\frac{(q+t)O_1O_{23}}{q^{\frac34}\sqrt{t}} +\frac{(q+t)O_3O_{12}}{\sqrt[4]{q}\sqrt{t}} -\frac{(q+t)O_{123}}{\sqrt{q}\sqrt{t}}+\frac{(q+1)O_5O_2}{\sqrt{q}}
\end{align*}
\caption{Normal ordering relations between generators}
\label{tab:QCommRel}
\end{table}

\begin{lemma}
Six elements $O_1,O_2,O_3,O_4,O_5,O_6\in\mathcal A_{q,t}$ generate $\mathcal A_{q,t}$ as an associative algebra.
\label{lemm:SixGenerators}
\end{lemma}
\begin{proof}
Normal ordering relations, in particular, imply that
\begin{align}
\big[O_i,O_{i+1}\big]_q=O_{i,i+1},\qquad\textrm{for all}\quad 1\leqslant i\leqslant6
\label{eq:LevelTwoGeneratorsRelation}
\end{align}
Hereinafter, by index $(i+a)$ we mean $((i+a-1)\bmod 6)+1$. Similarly,
\begin{align}
\big[\big[O_i,O_{i+1}\big]_q,O_{i+2}]_q=O_{i,i+1,i+2}\qquad\textrm{for all}\quad 1\leqslant i\leqslant3.
\label{eq:LevelThreeGeneratorRelation}
\end{align}
Hence, all of the 15 generators can be obtained as $q$-commutators of the six elements $O_1,\dots,O_6$.
\end{proof}

\begin{lemma}
Normally ordered monomials of the form
\begin{align}
O_1^{n_1}O_2^{n_2}O_3^{n_3}O_4^{n_4}O_5^{n_5}O_6^{n_6} O_{12}^{n_7}O_{23}^{n_8}O_{34}^{n_9}O_{45}^{n_{10}}O_{56}^{n_{11}}O_{61}^{n_{12}} O_{123}^{n_{13}}O_{234}^{n_{14}}O_{345}^{n_{15}},\qquad n_1,\dots,n_{15}\in\mathbb Z_{\geqslant 0}.
\label{eq:NormallyOrderedMonomials}
\end{align}
provide a spanning set for algebra $\mathcal A_{q,t}$
\label{lemm:NormalOrderingLemma}
\end{lemma}

\begin{proof}
Consider free associative algebra $F$ with the same collection of generators (\ref{eq:15Generators}). Introduce grading on $\mathcal F$ by assigning the following degrees to generators
\begin{align}
\deg O_i=2,\qquad\deg O_{i,i+1}=3,\qquad\deg O_{i,i+1,i+2}=4,\qquad 1\leqslant i\leqslant6.
\label{eq:GeneratorWeights}
\end{align}
We have an increasing filtration on $F$ defined by the grading
\begin{align*}
F^{(0)}\subseteq F^{(1)}\subseteq F^{(2)}\subseteq F^{(3)}\subseteq\dots\subseteq F=\bigcup_{m\in\mathbb Z_{\geqslant0}} F^{(m)}.
\end{align*}
This induces an increasing filtration on $\mathcal A_{q,t}$, where filtered components
\begin{align*}
\mathcal A_{q,t}^{(m)}=\pi\big(F^{(m)}\big)
\end{align*}
are given by images of $F^{(m)}$ under natural projection map $\pi:F\rightarrow \mathcal A_{q,t}$.

We will prove the statement of the lemma by induction in degree $m$. The base case $m=0$ is immediate. Now, by inductive assumption suppose that every $x\in A_{q,t}^{(m-1)}$ can be presented as a linear combination of normally ordered monomials. Consider $y\in\mathcal A_{q,t}^{(m)}$ and note that each entry in Table \ref{tab:QCommRel} has degree strictly smaller than the sum of the degrees of the corresponding generators which label row and column respectively. Hence, applying a sequence of relations from Table \ref{tab:QCommRel} we can bring top degree monomials in $y$ to the normally ordered form at the expense of getting some subleading monomials. Together with inductive assumption this means that $y\in\mathcal A_{q,t}^{(m)}$ can be presented as a linear combination of normally ordered monomials.
\end{proof}

\begin{lemma}
The following permutation of generators defines an order 6 automorphism of $\mathcal A_{q,t}$:
\begin{align}
I=\big(O_1,O_2,O_3,O_4,O_5,O_6\big) \big(O_{12},O_{23},O_{34},O_{45},O_{56},O_{61}\big)\big(O_{123},O_{234},O_{345}\big)
\label{eq:IActionAqt}
\end{align}
\label{lemm:IAutomorphism}
\end{lemma}
\begin{proof}
The set of defining relations (\ref{eq:NormalOrderingRelations})--(\ref{eq:qCasimirRelation}) falls short of being explicitly invariant under the action of $I$ by just three instances where the action of $I$ spoils the ordering. In these cases we use explicit calculations below
\begin{align*}
I(\eta_{(56)(45)})=&\eta _{(61)(56)} -q^{-1}(q-1) (q+1)\eta _{(5)(1)},\\
I(\eta_{(61)(12)})=& -\eta _{(23)(12)} +q^{-1}(q-1) (q+1)\eta _{(3)(1)},\\
I(\eta_{(345)(234)})=&q^{-\frac{5}{4}}t^{-\frac{1}{2}}(q-1) (q+t)O_1\eta _{(6)(5)} +q^{-\frac{3}{4}}t^{-\frac{1}{2}}(q-1) (q+t)O_2\eta _{(4)(3)} +q^{-1}t^{-\frac{1}{2}}(q-1) (q+t)\eta _{(4)(2)}O_3\\
&-q^{-1}t^{-\frac{1}{2}}(q-1) (q+t)O_3\eta _{(4)(2)} -q^{-\frac{3}{4}}t^{-\frac{1}{2}}(q-1) (q+t)\eta _{(3)(2)}O_4 +q^{-\frac{5}{4}}t^{-\frac{1}{2}}(q-1) (q+t)\eta _{(6)(1)}O_5\\
&-\eta _{(345)(123)} +q^{-\frac{3}{2}}t^{-\frac{1}{2}}(q-1)^2 (q+t)\eta _{(61)(5)} +(q-1)\eta _{(45)(12)} +q^{-\frac{3}{2}}t^{-\frac{1}{2}}(q-1)^2 (q+t)\eta _{(23)(4)}\\
&-q^{-1}(q-1) (q+2)\eta _{(6)(3)} +q^{-\frac{1}{2}}(q-1)\rho _{15} -q^{-\frac{1}{2}}(q-1)\rho _0.
\end{align*}
\end{proof}

\subsection{$J$-relations}

Spanning set (\ref{eq:NormallyOrderedMonomials}) is not a basis, because there are relations between normally ordered monomials. To obtain examples of such relations, consider a product $O_KO_JO_I$ of three of generators in reverse lexicographic order. There exist at least two distinct ways of bringing the leading term of $O_KO_JO_I$ to lexicographic order via a sequence of relations (\ref{eq:NormalOrderingRelations}), we have
\begin{subequations}
\begin{equation}
\begin{aligned}
(O_KO_J)O_I=&\left(q^{-\frac{c_{K,J}}2}O_JO_K +q^{-\frac{c_{K,J}}4}[O_K,O_J]_{q^{c_{K,J}}}\right)O_I\\ =&q^{-\frac{c_{K,J}+c_{K,I}+c_{J,I}}2}O_IO_JO_K +q^{-\frac{c_{K,J}+c_{K,I}}2-\frac{c_{J,I}}4} [O_J,O_I]_{q^{c_{J,I}}}O_K\\
 &+q^{-\frac{c_{K,J}}2-\frac{c_{K,I}}4}O_J[O_K,O_I]_{q^{c_{K,I}}} +q^{-\frac{c_{K,J}}4}[O_K,O_J]_{q^{c_{K,J}}}O_I,
\end{aligned}
\label{eq:TripleOrderingI}
\end{equation}
\vspace{4pt}
\begin{equation}
\begin{aligned}
O_K(O_JO_I)=&q^{-\frac{c_{J,I}+c_{K,I}+c_{K,J}}2}O_IO_JO_K +q^{-\frac{c_{J,I}+c_{K,I}}2-\frac{c_{K,J}}4}O_I[O_K,O_J]_{q^{c_{K,J}}}\\ &+q^{-\frac{c_{J,I}}2-\frac{c_{K,I}}4}[O_K,O_I]_{q^{c_{K,I}}}O_J +q^{-\frac{c_{J,I}}4}O_K[O_J,O_I]_{q^{c_{J,I}}}.
\end{aligned}
\label{eq:TripleOrderingII}
\end{equation}
\end{subequations}
Here $c_{I,J}$ stands for the integer coefficient in the $O_I$'th row and $O_J$'th column of Table \ref{tab:QCommRel}. Subtracting (\ref{eq:TripleOrderingII}) from (\ref{eq:TripleOrderingI}) we get the identity
\begin{align}
[[O_I,O_J]_{q^{c_{I,J}}},O_K]_{q^{c_{J,K}+c_{I,K}}} =[O_I,[O_J,O_K]_{q^{c_{J,K}}}]_{q^{c_{I,J}+c_{I,K}}} -[O_J,[O_I,O_K]_{q^{c_{I,K}}}]_{q^{c_{J,I}+c_{J,K}}},
\label{eq:QJacobiRelation}
\end{align}
where we have used the fact that $c_{I,J}=-c_{J,I}$ for all generators $O_I,O_J$.

By examining (\ref{eq:QJacobiRelation}) for all triples of generators we find 18 distinct relations between normally ordered monomials. For symmetry reasons, however, it will be more convenient for us to consider a closely related collection of 18 relations, which we call \textit{$J$-relations}. This collection would have an advantage of being manifestly invariant under the action of finite-order automorphism $I$, even though individual monomials involved in these relations are not necessarily in the normal form (\ref{eq:NormallyOrderedMonomials}). As we will show later in the text, $J$-relations play the same role in the solution of the word problem in $\mathcal A_{q,t}$ as Jacobi relations of the usual Lie algebra do in Poincare-Birkhoff-Witt theorem.
\begin{lemma}[\textit{$J$-relations}]
The following relations hold in $\mathcal A_{q,t}$ for all $1\leqslant i\leqslant 6$
\begin{subequations}
\begin{align}
q^{-\frac12}O_{i+2}O_{i+4}+q^{\frac12}O_{i+3}O_{i+2,i+3,i+4}-O_{i+2,i+3} O_{i+3,i+4}-(q^{\frac12}t^{-\frac12}+q^{-\frac12}t^{\frac12})O_i=0,
\label{eq:JRelationsA}
\end{align}
\begin{align}
-q^{-1}O_{i+3}O_{i+5,i} -O_{i+4}O_{i+1,i+2}+q^{-\frac12}O_{i+3,i+4} O_{i+1,i+2,i+3}-&(q^{\frac12}t^{-\frac12}+t^{\frac12}q^{-\frac12})\left( O_{i,i+1}-q^{-\frac14}O_iO_{i+1}\right)=0,
\label{eq:JRelationsB}
\end{align}
\begin{equation}
\begin{aligned}
-q^{\frac12}O_{i+1,i+2}O_{i+4,i+5} +O_{i,i+1,i+2} O_{i+1,i+2,i+3}-q^{-\frac12}O_iO_{i+3} -&\left(q^{\frac12}t^{-\frac12}+q^{-\frac12}t^{\frac12}\right)\times
\\
\big(-\left(q-1+q^{-1}\right) O_{i+2,i+3,i+4}+q^{-3/4} O_{i+1}O_{i+5,i}+&q^{3/4} O_{i+5}O_{i,i+1}-O_iO_{i+1}O_{i+5}\big)=0.
\end{aligned}
\label{eq:JRelationsC}
\end{equation}
\label{eq:JRelations}
\end{subequations}
Here by index $(i+a)$ we mean $((i+a-1)\bmod 6)+1$.
\label{lemm:JRelations}
\end{lemma}

\begin{proof}
Note that by Lemma (\ref{lemm:IAutomorphism}) it is enough for us to prove only one relation of each type. Let $F$ be the free associative algebra with the same collection of generators. Denote the l.h.s. of (\ref{eq:JRelationsA}), (\ref{eq:JRelationsB}), and (\ref{eq:JRelationsC}) by $\rho_i, \rho_{6+i},\rho_{12+i}\in F$ respectively. In addition, for every pair of generators $O_I,O_J$, let $\eta_{I,J}\in F$ stand for the normal ordering relator from Table \ref{tab:QCommRel}:
\begin{align}
\eta_{I,J}=q^{\frac{c_{I,J}}4}O_IO_J-q^{-\frac{c_{I,J}}4}O_JO_I \mp(q^{\frac12}-q^{-\frac12})X.
\label{eq:EtaIJRelator}
\end{align}
We have the following identities in $F$:
\begin{subequations}
\begin{equation}
\begin{aligned}
\rho_1=&q^{-\frac12}O_{3}O_{5}+q^{\frac12}O_{4}O_{345}-O_{34} O_{45}-(q^{\frac12}t^{-\frac12}+q^{-\frac12}t^{\frac12})O_1\\[5pt]
=&\frac{q^{\frac1{2}}}{(q-1)}\Big(\big[\eta_{(5)(4)},O_{34}\big]_{q^0} -[\eta_{(34),(4)},O_5]_{q^1} -\big[O_4,\eta_{(34)(5)}\big]_{q^1}+\eta_{(345)(4)} -\eta_{(45)(34)} -\eta_{(5)(3)}\big),
\end{aligned}
\end{equation}
\begin{equation}
\begin{aligned}
\rho_7=&-q^{-1}O_{4}O_{61} -O_{5}O_{23}+q^{-\frac12}O_{45} O_{234}-(q^{\frac12}t^{-\frac12}+t^{\frac12}q^{-\frac12})\left( O_{12}-q^{-\frac14}O_1O_{2}\right)\\[5pt]
=&\frac1{q-1}\left(\big[O_{45},\eta_{(34)(2)}\big]_{q^2} +\big[\eta_{(45)(2)},O_{34}\big]_{q^3} +\big[O_2,\eta_{(45)(34)}\big]_{q^1}-t^{-\frac{1}{2}}(q+t)\eta_{(2)(1)} -q^{-\frac12}(q+1)\eta_{(3)(2)}O_5\right.\\
&\left.\qquad\qquad-q^{-\frac{3}{4}}(q+1)O_3\eta_{(5)(2)} +(q-1)\eta_{(23)(5)}
+q^{-1}(q-1)\eta_{(61)(4)} +q^{-\frac12}\eta_{(234)(45)}\right),
\end{aligned}
\end{equation}
\begin{equation}
\begin{aligned}
\rho_{13}=&-q^{\frac12}O_{23}O_{56} +O_{123} O_{234}-q^{-\frac12}O_1O_4 -\left(q^{\frac12}t^{-\frac12}+q^{-\frac12}t^{\frac12}\right)\\
&\qquad\qquad\times\big(-\left(q-1+q^{-1}\right) O_{345}+q^{-3/4} O_{2}O_{6,1}+q^{3/4} O_{6}O_{12}-O_iO_2O_6\big)\\[5pt]
=&\frac{q}{q-1}\left(\big[\eta_{(23)(4)},O_{123}\big]_{q^0} +\big[O_{23},\eta_{(123)(4)}\big]_{q^1}+[\eta_{(123)(23)},O_4]_{q^1} -q^{-\frac{7}{4}}t^{-\frac{1}{2}}(q-1)(q+t)\eta_{(2)(1)}O_6\right.\\
&\qquad\qquad+q^{-\frac{3}{2}}(q^2+q-1)\eta_{(4)(1)} -q^{-\frac{7}{4}}t^{-\frac{1}{2}}(q-1)(q+t)O_2\eta_{(6)(1)}\\
&\qquad\qquad\left.+q^{-2}t^{-\frac{1}{2}}(q-1)^2(q+t)\eta_{(12)(6)}-q^{-\frac{1}{2}}\eta_{(56)(23)} +q^{-\frac{1}{2}}\eta_{(234)(123)}\right).
\end{aligned}
\end{equation}
\label{eq:JrelationsViaNormal}
\end{subequations}
Hence $\rho_1,\rho_7, \rho_{13}$ belong to defining ideal. Applying Lemma \ref{lemm:IAutomorphism} we conclude the proof.
\end{proof}

\begin{remark}
It is worth noting that although $\rho_1,\dots,\rho_{18}$ belong to ideal in $F$ generated by normal ordering relations, the r.h.s. of expressions (\ref{eq:JrelationsViaNormal}) contain powers of $(q-1)$ in the denominator. In other words, $J$-relations become independent in the limit $q\rightarrow 1$.
\end{remark}

\subsection{Mapping Class Group Action}

In this section we show that Mapping Class Group $\mathrm{Mod}(\Sigma_2)$ of a genus two surface acts by automorphisms of the algebra $\mathcal A_{q,t}$. Finite presentation of genus two Mapping Class Group $\mathrm{Mod}(\Sigma_2)$ was originally constructed by J.~Birman and H.~Hilden \cite{BirmanHilden'1971}, it has five generators $d_1,d_2,d_3,d_4,d_5$ which are subject to
\begin{itemize}
\item Coxeter relations:
\begin{align*}
d_id_{i+1}d_i=&d_{i+1}d_id_{i+1},\qquad\textrm{for all}\quad 1\leqslant i\leqslant 4,\\
d_id_j=&d_jd_i,\qquad\textrm{when}\quad |i-j|>1.
\end{align*}
\item Finite order relations:
\begin{align*}
I^6=H^2=1,\quad\textrm{where}\quad I=d_1d_2d_3d_4d_5\quad\textrm{and}\quad H=d_1d_2d_3d_4d_5d_5d_4d_3d_2d_1.
\end{align*}
\end{itemize}
Five generators $d_1,d_2,d_3,d_4,d_5$ correspond respectively to Dehn twists along the cycles $A_1,B_{1,2},A_2,B_{2,3},A_3$ shown on Figure \ref{fig:qDiffGeneratingCycles}. At the first step, however, it will be easier for us to use a smaller generating set. It is known that Mapping Class Group of any closed surface can be generated by two elements \cite{Wajnryb'1996,Korkmaz'2005}. In the case of $\mathrm{Mod}(\Sigma_2)$ it is enough to take $d_1$ and $I$. Note that Coxeter relations imply
\begin{align}
d_i=I^{i-1}d_1I^{1-i}\qquad\textrm{for all}\quad 1\leqslant i\leqslant 5.
\label{eq:DehnTwistsViaD1}
\end{align}
We have already introduced the action of an order six automorphism $I$ in Lemma \ref{lemm:IAutomorphism}, so now we will introduce the action of $d_1$ on $\mathcal A_{q,t}$.

Let $F$ be the free algebra over $\mathbf k$ with 15 generators (\ref{eq:15Generators}). Define a homomorphism $d_1:F\rightarrow F$ by its action on generators
\begin{align}
d_1:\quad\left\{\begin{array}{cll}
 O_2 &\mapsto & q^{\frac14} O_1O_2-q^{\frac12} O_{12} \\[2pt]
 O_6 &\mapsto& O_{61} \\[2pt]
 O_{12} &\mapsto& O_2 \\[2pt]
 O_{23} &\mapsto& q^{\frac14} O_1O_{23}-q^{\frac12} O_{123} \\[2pt]
 O_{56} &\mapsto& O_{234} \\[2pt]
 O_{61} &\mapsto& q^{\frac14} O_1O_{61}-O_6 q^{\frac12} \\[2pt]
 O_{123} &\mapsto& O_{23} \\[2pt]
 O_{234} &\mapsto& q^{\frac14} O_1O_{234}-q^{\frac12} O_{56}
\end{array}\right.
\label{eq:d1Automorphism}
\end{align}
where we assume that the action of $d_1$ on the remaining generators is identical.
\begin{proposition}
Homomorphism $d_1:F\rightarrow F$ is invertible. Moreover, defining ideal (\ref{eq:DefiningRelationsAqt}) of $\mathcal A_{q,t}$ is invariant under the action of $d_1$.
\label{prop:d1Automorphism}
In other words, $d_1$ descends to an automorphism of $\mathcal A_{q,t}$.
\end{proposition}
\begin{proof}
The inverse homomorphism $d_1^{-1}:F\rightarrow F$ can be described by its action on generators
\begin{align}
d_1^{-1}:\quad\left\{\begin{array}{cll}
 O_2 &\mapsto& O_{12} \\[2pt]
 O_6 &\mapsto& q^{-\frac14}O_1O_6-q^{-\frac12}O_{61} \\[2pt]
 O_{12} &\mapsto& q^{-\frac14}O_1O_{12}-q^{-\frac12}O_2\\[2pt]
 O_{23} &\mapsto& O_{123} \\[2pt]
 O_{56} &\mapsto& q^{-\frac14}O_1O_{56}-q^{-\frac12}O_{234} \\[2pt]
 O_{61} &\mapsto& O_6 \\[2pt]
 O_{123} &\mapsto& q^{-\frac14}O_1O_{123}-q^{-\frac12}O_{23} \\[2pt]
 O_{234} &\mapsto& O_{56}
\end{array}\right.
\label{eq:d1InverseAutomorphism}
\end{align}
Again, here we assume that the action is identical on the remaining generators. As a result, we get
\begin{equation}
\begin{aligned}
(d_1^{-1}\circ d_1) (O_2)=&d_1^{-1}\left(q^{\frac14}O_1O_2-q^{\frac12}O_{12}\right)=q^{\frac14}O_1O_{12} -q^{\frac12}\left(q^{-\frac14}O_1O_2-q^{-\frac12}O_2\right)=O_2,\\
(d_1\circ d_1^{-1})(O_2)=&d_1\left(O_{12}\right)=O_2.
\end{aligned}
\label{eq:d1d1IActionO2GenericQ}
\end{equation}
Similar computation for $O_6,O_{12},O_{23},O_{56},O_{61},O_{123},O_{234}$ shows that (\ref{eq:d1InverseAutomorphism}) is in fact the inverse for (\ref{eq:d1Automorphism}).

The proof of the second part amounts to tedious but straightforward examination of all 106 generators of defining ideal (\ref{eq:DefiningRelationsAqt}). We present this calculation in Appendix \ref{sec:d1Automorphism}.
\end{proof}

Now we are ready to prove the main theorem of the current subsection that the Mapping Class Group $\mathrm{Mod}(\Sigma_2)$ of a closed genus two surface acts by automorphisms of $\mathcal A_{q,t}$. It first appeared as Theorem 25 of \cite{ArthamonovShakirov'2019}, even before we had a complete presentation of the algebra $\mathcal A_{q,t}$, and originally the theorem was formulated for the algebra of difference operators (\ref{eq:AqtDifferenceOperatorsAlgebra}). Our proof essentially repeats the logic of \cite{ArthamonovShakirov'2019} in the new context of Definition \ref{def:Aqt}.
\begin{theorem}
Two automorphisms $d_1,I:\mathcal A_{q,t}\rightarrow\mathcal A_{q,t}$ define the action of Mapping Class Group $\mathrm{Mod}(\Sigma_2)$ of genus two surface on $\mathcal A_{q,t}$.
\label{th:MCGActionGeneric}
\end{theorem}
\begin{proof}
By (\ref{eq:DehnTwistsViaD1}), the action of the remaining four Dehn twists on generators of $\mathcal A_{q,t}$ can be obtained by a simple shift of indexes in formula (\ref{eq:d1Automorphism}). Combining this fact with Lemma \ref{lemm:IAutomorphism} we note that it is enough for us to prove only four relations
\begin{align}
d_1d_2d_1=d_2d_1d_2,\qquad d_1d_3=d_3d_1,\qquad d_1d_4=d_4d_1,\qquad H^2=1.
\label{eq:SufficientMCGRelations}
\end{align}
Moreover, by Lemma \ref{lemm:SixGenerators} it is enough to check that (\ref{eq:SufficientMCGRelations}) is satisfied only on six generators $O_1,\dots, O_6$. We get
\begin{subequations}
\begin{equation}
\begin{aligned}
d_1d_2d_1(O_1)=\;&O_2=d_2d_1d_2(O_1),\\[5pt]
d_1d_2d_1(O_2)=\;&q^{\frac{1}{2}}O_2O_1O_2 -qO_1O_2^2 +q^{\frac{5}{4}}O_{12}O_2 -q^{\frac{3}{4}}O_2O_{12} +qO_1\\
\stackrel{(\ref{eq:NormalOrderingRelations})}{=}& q^{\frac{1}{4}}O_{12}O_2-q^{\frac{3}{4}}O_2O_{12}+qO_1=d_2d_1d_2(O_2),\\[5pt]
d_1d_2d_1(O_3)=\;&q^{\frac{1}{2}}O_1O_2O_3 -q^{\frac{3}{4}}O_{12}O_3 -q^{\frac{3}{4}}O_1O_{23} +qO_{123}\\
\stackrel{(\ref{eq:NormalOrderingRelations})}{=}&q^{\frac{3}{4}}O_{12}O_2^2O_3 -q^{\frac{5}{4}}O_2O_{12}O_2O_3 -qO_{12}O_2O_{23} +q^{\frac{3}{2}}O_2O_{12}O_{23}\\ &+q^{\frac{3}{2}}O_1O_2O_3 -q^{\frac{3}{4}}O_{12}O_3 -q^{\frac{7}{4}}O_1O_{23} +qO_{123} =d_2d_1d_2(O_3),\\[5pt]
d_1d_2d_1(O_4)=\;&O_4=d_2d_1d_2(O_4),\\[5pt]
d_1d_2d_1(O_5)=\;&O_5=d_2d_1d_2(O_5),\\[5pt]
d_1d_2d_1(O_6)=\;&O_{345}=d_2d_1d_2(O_6).
\end{aligned}
\end{equation}
For the second relation in (\ref{eq:SufficientMCGRelations}) we have
\begin{equation}
\begin{aligned}
d_1d_3(O_1)=&O_1=d_3d_1(O_1),\\
d_1d_3(O_2)=&q^{\frac{1}{4}}O_1O_{23}-q^{\frac{1}{2}}O_{123}=d_3d_1(O_2),\\
d_1d_3(O_3)=&O_3=d_3d_1(O_3),\\
d_1d_3(O_4)=&q^{\frac{1}{4}}O_3O_4-q^{\frac{1}{2}}O_{34}=d_3d_1(O_4),\\
d_1d_3(O_5)=&O_5=d_3d_1(O_5),\\
d_1d_3(O_6)=&O_{61}=d_3d_1(O_6).
\end{aligned}
\end{equation}
Next, we have
\begin{equation}
\begin{aligned}
d_1d_4(O_1)=&O_1=d_4d_1(O_1),\\
d_1d_4(O_2)=&q^{\frac{1}{4}}O_1O_2-q^{\frac{1}{2}}O_{12}=d_4d_1(O_2),\\
d_1d_4(O_3)=&O_{34}=d_4d_1(O_3),\\
d_1d_4(O_4)=&O_4=d_4d_1(O_4),\\
d_1d_4(O_5)=&q^{\frac{1}{4}}O_4O_5-q^{\frac{1}{2}}O_{45}=d_4d_1(O_5),\\
d_1d_4(O_6)=&O_{61}=d_4d_1(O_6).
\end{aligned}
\end{equation}
\label{eq:MCGRelatorsOnGeneratorsGeneric}
\end{subequations}

Finally, for the last relation in (\ref{eq:SufficientMCGRelations}) we will prove a slightly stronger statement that $H=1$, or equivalently, that
\begin{align*}
\widetilde I=d_5d_4d_3d_2d_1=I^{-1}
\end{align*}
realizes the inverse of an order 6 automorphism $I$.\footnote{This is closely related to the fact that hyperelliptic involution $H$, although being nontrivial mapping class, acts trivially on the $SL(2,\mathbb C)$-character variety. As we can see, this property also translates to the quantized and deformed $SL(2,\mathbb C)$-character variety.} Computing the action of $\widetilde I$ on 6 generators we get
\begin{align*}
\widetilde I(O_1)=&O_6,\\
\widetilde I(O_2)=&-q^{\frac{3}{4}}O_{61}O_6+q^{\frac{1}{4}}O_6O_{61}+qO_1 \stackrel{(\ref{eq:NormalOrderingRelations})}{=}O_1,\\
\widetilde I(O_3)=&-q^{\frac{3}{4}}O_{345}O_{61}+q^{\frac{1}{4}}O_{61}O_{345}+qO_2 \stackrel{(\ref{eq:NormalOrderingRelations})}{=}O_2,\\
\widetilde I(O_4)=&q^{\frac{1}{4}}O_{345}O_{45}-q^{\frac{3}{4}}O_{45}O_{345}+qO_3
\stackrel{(\ref{eq:NormalOrderingRelations})}{=}O_3,\\
\widetilde I(O_5)=&q^{\frac{1}{4}}O_{45}O_5-q^{\frac{3}{4}}O_5O_{45}+qO_4 \stackrel{(\ref{eq:NormalOrderingRelations})}{=}O_4,\\
\widetilde I(O_6)=&O_5.
\end{align*}
\end{proof}

\subsection{Presentation of $\mathcal A_{q,t}$ in $q$-difference operators.}

In this section we show that algebra $\mathcal A_{q,t}$ admits presentation in $q$-difference operators. This result is not at all surprising, as we have come up with Definition \ref{def:Aqt} from examining relations between difference operators.

\begin{proposition}
We have an algebra homorphism
\begin{align*}
\widehat{\Delta}:\mathcal A_{q,t}\rightarrow A_{q,t},\qquad O_i\mapsto \hat O_i,\quad 1\leqslant i\leqslant 6.
\end{align*}
\label{prop:qDifferenceHomomorphism}
\end{proposition}
\begin{proof}
Recall that by Lemma \ref{lemm:SixGenerators} we know that the remaining 9 generators in presentation of $\mathcal A_{q,t}$ are nothing but $q$-commutators of the six elements $O_i,\;1\leqslant i\leqslant 6$. To prove the statement of the Theorem we must show that all relations in defining ideal (\ref{eq:DefiningRelationsAqt}) are satisfied in difference operators.

First, note that both $\mathcal A_{q,t}$ and $A_{q,t}$ are equipped with the Mapping Class Group action compatible on generators. To this end compare Lemma \ref{lemm:IAutomorphism} and Proposition \ref{prop:d1Automorphism} with Lemma 23 and Proposition 24 of \cite{ArthamonovShakirov'2019}. As a result, when verifying normal ordering relations (\ref{eq:NormalOrderingRelations}) we can assume without loss of generality that the first generator is $O_1$. Hence, it is enough for us to examine only the first column of Table \ref{tab:QCommRel}. All these relations are either identity by construction, or were proved in Lemma 11 of \cite{ArthamonovShakirov'2019}.

As for the $q$-Casimir relation (\ref{eq:qCasimirRelation}), it can be proved by applying the corresponding $q$-difference operator to general function of three variables $f(x_{1,2},x_{2,3},x_{1,3})$ and combining the terms with the same $q$-shifts of the arguments. We omit cumbersome details of this calculation from the main text and invite careful reader to examine our Mathematica code \cite{Arthamonov-GitHub-Flat} which does such calculation.
\end{proof}


Proposition \ref{prop:qDifferenceHomomorphism} implies that for all generators $O_I\in\mathcal A_{q,t}$, the image $\widehat{\Delta}(O_I)$ is a $q$-difference operator with coefficients in the field $\mathbf k(x_{12},x_{23},x_{13})=\mathbb C(q^\frac14,t^{\frac14})(x_{12},x_{23},x_{13})$. In particular, $\widehat{\Delta}(O_{12})$ a priori might contain $(q-1)$ factor in the denominator, coming from the $q$-commutator formula (\ref{eq:qCommutatorDef}). However, this doesn't actually happen for our choice of generators. In fact, we can prove the following property, which plays crucial role in computation of quasiclassical limit in Section \ref{sec:QDifferenceQuasiclassicalLimit}.

\begin{lemma}
For all generators $O_I\in\mathcal A_{q,t}$, the corresponding $q$-difference operator
\begin{align*}
\widehat{\Delta}(O_I)\in\mathbb C[q^{\pm\frac14},t^{\pm\frac14}](x_{12},x_{23},x_{13}) \big\langle\hat\delta_{12}^{\pm1}, \hat\delta_{23}^{\pm1}, \hat\delta_{13}^{\pm1}\big\rangle
\end{align*}
has coefficients which are Laurent polynomials in variables $q^{\frac14}$ and $t^{\frac14}$.
\label{lemm:qDifferenceLaurentCoefficients}
\end{lemma}
\begin{proof}
From (\ref{eq:Mult})--(\ref{eq:Hamiltonians}) Next, we note that on the level of $q$-difference operators automorphism $I^2$ corresponds to cyclic permutation $(1,2,3)$ of indices of $x_{ij}$ and $\hat\delta_{ij}$. Hence, it will be enough for us to consider the action of $\widehat{\Delta}$ on generators $O_1,O_2,O_{12},O_{23},O_{123}$. By construction, the statement of the lemma is immediate for $\widehat{\Delta}(O_1)$ and $\widehat{\Delta}(O_2)$. As for the remaining three generators we get
\begin{equation}
\begin{aligned}
\widehat{\Delta}(O_{12})&\;\stackrel{(\ref{eq:LevelTwoGeneratorsRelation})}{=}\; \frac{q^{\frac14}\widehat{\Delta}(O_1)\widehat{\Delta}(O_2) -q^{-\frac14}\widehat{\Delta}(O_2)\widehat{\Delta}(O_1)}{q^{\frac12}-q^{-\frac12}}\\[5pt]
&\stackrel{(\ref{eq:Mult1}),(\ref{eq:Hamiltonian1})}{=} \sum_{a,b\in\{\pm1\}}ab\, q^{\frac14}t^{-\frac12}\,\frac{(1-t^{\frac12}X_{23}X_{12}^aX_{13}^b) (1-t^{\frac12}X_{23}^{-1}X_{12}^aX_{13}^b)} {X_{13}^b(X_{12}-X_{12}^{-1})(X_{13}-X_{13}^{-1})}\, \hat\delta_{12}^a\hat\delta_{13}^b,
\end{aligned}
\label{eq:DeltaHatO12}
\end{equation}
\smallskip
\begin{align*}
\widehat{\Delta}(O_{23})&\;\stackrel{(\ref{eq:LevelTwoGeneratorsRelation})}{=}\; \frac{q^{\frac14}\widehat{\Delta}(O_2)\widehat{\Delta}(O_3) -q^{-\frac14}\widehat{\Delta}(O_3)\widehat{\Delta}(O_2)}{q^{\frac12}-q^{-\frac12}}\\[5pt]
&\stackrel{(\ref{eq:Mult1}),(\ref{eq:Hamiltonian2})}{=} \sum_{a,b\in\{\pm1\}}ab\,q^{-\frac14}t^{-\frac12}\, \frac{(1-t^{\frac12}X_{13}X_{12}^aX_{23}^b) (1-t^{\frac12}X_{13}^{-1}X_{12}^aX_{23}^b)} {X_{12}^{2a}X_{23}^b(X_{23}-X_{23}^{-1})(X_{12}-X_{12}^{-1})}\, \hat\delta_{12}^a\hat\delta_{23}^b,\\[10pt]
\widehat{\Delta}(O_{123})&\;\stackrel{(\ref{eq:LevelThreeGeneratorRelation})}{=}\; \frac{q^{\frac14}\widehat{\Delta}(O_{12})\widehat{\Delta}(O_3) -q^{-\frac14}\widehat{\Delta}(O_3) \widehat{\Delta}(O_{12})}{q^{\frac12}-q^{-\frac12}}\\[5pt]
&\stackrel{(\ref{eq:DeltaHatO12}),(\ref{eq:Hamiltonian2})}{=}\sum_{a,b\in\{\pm1\}} ab\,\frac{(1-t^{\frac12}X_{12}X_{23}^aX_{13}^b)(1-t^{\frac12}X_{12}^{-1}X_{23}^aX_{13}^b)} {t^{\frac12}X_{23}^{2a}(X_{13}-X_{13}^{-1})(X_{23}-X_{23}^{-1})}\,
\hat\delta_{23}^a\hat\delta_{13}^b
\end{align*}
\end{proof}

\section{Classical Limit, $\mathcal A_{q=1,t}$}
\label{sec:ClassicalLimit}

One of the main goals of the current paper is to show that $\mathcal A_{q,t}$ is a quantization of a certain commutative Poisson algebra. To this end we have to study the limit $q\rightarrow 1$. Not every generating set of the defining ideal of $\mathcal A_{q,t}$ as an algebra over $\mathbf k=\mathbb C(q^{\frac14},t^{\frac14})$ behaves well when $q$ is specialized to a particular complex number.

Remarkably, when on top normal ordering relators $\eta_{I,J}$ (\ref{eq:EtaIJRelator}) we also add 18 extra $J$-relators (\ref{eq:JRelations}) to the defining ideal, this extended set of generators does behave nicely in the limit $q\rightarrow 1$. Note that for generic $q$, the $J$-relations would follow from normal ordering relations by Lemma \ref{eq:NormalOrderingRelations} as before. However, as $q\rightarrow 1$, they become independent. On the contrary, normal ordering relations $\eta_{I,J}=0$ in this limit turn into simple commutation relations between the generators. This motivates
\begin{definition}
Let $\mathcal A_{q=1,t}$ be a commutative algebra over $\mathbf k_t=\mathbb C(t^{\frac14})$ with 15 generators (\ref{eq:15Generators}), subject to
\begin{itemize}
\item 18 $J$-relations of the form
\begin{subequations}
\begin{align}
O_{i+2}O_{i+4}+O_{i+3}O_{i+2,i+3,i+4}-O_{i+2,i+3} O_{i+3,i+4}-\big(t^{\frac12}+t^{-\frac12}\big)O_i=0,
\label{eq:CommutativeJRelationsA}
\end{align}
\begin{align}
-O_{i+3}O_{i+5,i} -O_{i+4}O_{i+1,i+2}+O_{i+3,i+4} O_{i+1,i+2,i+3}-&\big(t^{\frac12}+t^{-\frac12}\big)\big( O_{i,i+1}-O_iO_{i+1}\big)=0,
\label{eq:CommutativeJRelationsB}
\end{align}
\begin{equation}
\begin{aligned}
-O_{i+1,i+2}O_{i+4,i+5} +O_{i,i+1,i+2} O_{i+1,i+2,i+3}-O_iO_{i+3} -&\big(t^{\frac12}+t^{-\frac12}\big)\times
\\
\big(- O_{i+2,i+3,i+4}+ O_{i+1}O_{i+5,i}+& O_{i+5}O_{i,i+1}-O_iO_{i+1}O_{i+5}\big)=0.
\end{aligned}
\label{eq:CommutativeJRelationsC}
\end{equation}
for all $1\leqslant i\leqslant 6$.
\smallskip
\item Casimir relation
\begin{equation}
\begin{aligned}
    O_{123}O_{234}O_{345}  -\big(O_1O_4O_{345} +O_2O_5O_{123} +O_3O_6O_{234}\big)\\[5pt]
    +\frac{1}{2}\big(t^{\frac{1}{2}}+t^{-\frac{1}{2}}\big)
    \big(O_1O_6O_{61} +O_2O_1O_{12} +O_3O_2O_{23} +O_4O_3O_{34} +O_5O_4O_{45} +O_6O_5O_{56}\big)\\[5pt]
    -\big(O_1O_3O_5 +O_2O_4O_6\big)
    -\frac{1}{2}\big(t^{\frac{1}{2}}+t^{-\frac{1}{2}}\big)
    \big(O_{12}^2+O_{23}^2+O_{34}^2+O_{45}^2+O_{56}^2+O_{61}^2\big)+
    \big(t^{\frac{1}{2}}+t^{-\frac{1}{2}}\big)^3&=0
\end{aligned}
\label{eq:CommutativeCasimirRelation}
\end{equation}
\label{eq:Aq1tDefiningIdeal}
\end{subequations}
\end{itemize}
\label{def:A1t}
\end{definition}

\subsection{Mapping Class Group action}

In this subsection we show that $\mathcal A_{q=1,t}$ is equipped with an action of the Mapping Class Group $\mathrm{Mod}(\Sigma_2)$ similarly to it generic counterpart $\mathcal A_{q,t}$. Moreover, the action of $I$ and $d_1$ on generators is described by the same formulae (\ref{eq:IActionAqt}) and (\ref{eq:d1Automorphism}) evaluated at $q=1$.

From Definition \ref{def:A1t}, we immediately obtain the analog of Lemma \ref{lemm:IAutomorphism}
\begin{lemma}
The following permutation of generators
\begin{align}
I=(O_1,O_2,O_3,O_4,O_5,O_6) (O_{12},O_{23},O_{34},O_{45},O_{56},O_{61}) (O_{123}O_{234}O_{456})
\label{eq:IActionCommutative}
\end{align}
extends to an order 6 automorphism of $\mathcal A_{q=1,t}$.
\label{lemm:IActionCommutative}
\end{lemma}

\begin{lemma}
Algebra $\mathcal A_{q=1,t}$ is equipped with an automorphism $d_1:\mathcal A_{q=1,t}\rightarrow\mathcal A_{q=1,t}$ which acts on generators as
\begin{align}
d_1:\quad\left\{\begin{array}{cll}
 O_2 &\mapsto & O_1O_2- O_{12} \\[2pt]
 O_6 &\mapsto& O_{61} \\[2pt]
 O_{12} &\mapsto& O_2 \\[2pt]
 O_{23} &\mapsto& O_1O_{23}- O_{123} \\[2pt]
 O_{56} &\mapsto& O_{234} \\[2pt]
 O_{61} &\mapsto& O_1O_{61}-O_6 \\[2pt]
 O_{123} &\mapsto& O_{23} \\[2pt]
 O_{234} &\mapsto& O_1O_{234}- O_{56}
\end{array}\right.
\qquad\qquad\qquad
d_1^{-1}:\quad\left\{\begin{array}{cll}
 O_2 &\mapsto& O_{12} \\[2pt]
 O_6 &\mapsto& O_1O_6-O_{61} \\[2pt]
 O_{12} &\mapsto& O_1O_{12}-O_2\\[2pt]
 O_{23} &\mapsto& O_{123} \\[2pt]
 O_{56} &\mapsto& O_1O_{56}-O_{234} \\[2pt]
 O_{61} &\mapsto& O_6 \\[2pt]
 O_{123} &\mapsto& O_1O_{123}-O_{23} \\[2pt]
 O_{234} &\mapsto& O_{56}
\end{array}\right.
\label{eq:d1ActionCommutative}
\end{align}
Here we have omitted all generators which are preserved by $d_1$.
\label{lemm:d1HomomorphismCommutative}
\end{lemma}
\begin{proof}
Let
\begin{align*}
 P=\mathbf k_t[O_1, O_2, O_3, O_4, O_5, O_6, O_{12}, O_{23}, O_{34}, O_{56}, O_{61}, O_{123}, O_{234}, O_{456}]
\end{align*}
be polynomial ring in 15 variables. First, one verifies that the two homomorphisms $d_1,d_1^{-1}:P\rightarrow P$ given by (\ref{eq:d1ActionCommutative}) are inverse to each other. To this end, we consider the action of $d_1\circ d_1^{-1}$ and $d_1^{-1}\circ d_1$ on all 15 generators of the polynomial ring $P$, similarly to (\ref{eq:d1d1IActionO2GenericQ}). After that, we have left to prove that defining ideal (\ref{eq:Aq1tDefiningIdeal}) is preserved by $d_1$.

For all $1\leqslant i\leqslant 6$ denote the l.h.s. of (\ref{eq:CommutativeJRelationsA}), (\ref{eq:CommutativeJRelationsB}), (\ref{eq:CommutativeJRelationsC}) by $r_i,r_{6+i},r_{12+i}$ respectively and let $r_0$ stands for the l.h.s. of the Casimir relation \ref{eq:CommutativeCasimirRelation}. First, we note that
\begin{align*}
r_{16}=&O_{3,4}r_5-O_{1,2}r_6-O_2r_{12}+O_3r_{10}+r_{13},\\
r_{17}=&O_{4,5}r_6-O_{2,3}r_1-O_3r_7+O_4r_{11}+r_{14},\\
r_{18}=&O_{5,6}r_1-O_{3,4}r_2-O_4r_8+O_5r_{12}+r_{15},
\end{align*}
so it will be enough to examine the action of $d_1$ on 16 relators $r_0,\dots,r_{15}$.

The following relators are preserved by $d_1$:
\begin{align*}
r_1,r_3,r_4,r_5,r_9,r_{10},r_{13}.
\end{align*}
As for the remaining relators, we use the following identities in $P$:

\begin{subequations}
\begin{align}
d_1(r_2)=&O_4 O_{61}+O_5 O_{23}-O_{45} O_{234}+(t^{\frac12}+t^{-\frac12})\big(O_{12}-O_1 O_2\big)\nonumber\\
=&-r_7,\label{eq:d1r2RelatorActionCommutative}\\[7pt]
d_1(r_6)=&-O_3 O_{56}-O_4 O_{12}+O_1 O_2 O_4+O_{34} O_{123}+O_1 O_3 O_{234}-O_1 O_{23} O_{34}-(t^{\frac12}+t^{-\frac12})O_{61}\nonumber\\
=&O_1r_6+r_{12},\\[7pt]
d_1(r_7)=&O_4 O_6+O_5 O_{123}-O_{45} O_{56}-O_1 O_4 O_{61}-O_1 O_5 O_{23}+O_1 O_{45} O_{234}\nonumber\\
&\quad+(t^{\frac12}+t^{-\frac12})\big(O_1^2 O_2-O_1 O_{12}-O_2\big),\\
=&r_2+O_1r_7,\nonumber\\[7pt]
d_1(r_8)=&-O_2 O_5-O_{34} O_{61}+O_{234} O_{345}+(t^{\frac12}+t^{-\frac12})\big(O_{123}-O_1 O_{23}-O_3 O_{12}+O_1 O_2 O_3\big)\nonumber\\
=&r_{14},\label{eq:d1r8RelatorActionCommutative}\\[7pt]
d_1(r_{11})=&O_3 O_6+O_{12} O_{45}-O_1 O_2 O_{45}-O_1 O_3 O_{61}-O_{123} O_{345}+O_1 O_{23} O_{345}\nonumber\\
&\quad+(t^{\frac12}+t^{-\frac12})\big(O_5 O_{61}-O_{234}\big)\\
=&O_{2,3}r_4-O_{6,1}r_5+O_2r_9-r_{15},\nonumber\\[7pt]
d_1(r_{12})=&-O_2 O_4-O_3 O_{234}+O_{23} O_{34}+(t^{\frac12}+t^{-\frac12})O_6\nonumber\\
=&-r_6,\\[7pt]
d_1(r_{14})=&O_5 O_{12}+O_6 O_{34}-O_1 O_2 O_5 -O_{56} O_{345}-O_1 O_{34} O_{61}+O_1 O_{234} O_{345}\nonumber\\
&\quad+(t^{\frac12}+t^{-\frac12})\big(O_{23} -O_2O_3 +O_1O_{123} -O_1^2O_{23} -O_1O_3O_{12} +O_1^2O_2O_3\big)\\
=&-r_8+O_1r_{14},\nonumber\\[7pt]
d_1(r_{15})=& -O_2 O_{45}-O_3 O_{61}+O_{23} O_{345}+(t^{\frac12}+t^{-\frac12})\big(-O_{56} +O_1O_{234} +O_4O_{123}\nonumber\\
&\qquad\quad+O_{12}O_{34} -O_1O_2O_{34} -O_1O_4O_{23} -O_3O_4O_{12} +O_1O_2O_3O_4\big)\\
=&-O_{234}r_1-O_{34}r_7+O_5r_6+r_{11}+O_4r_{14}.\nonumber
\end{align}
\label{eq:d1RelatorActionCommutative}
\end{subequations}
Finally, for $r_0$ we get
\begin{align*}
d_1(r_0)=& -O_1O_3O_5 -O_1O_4O_{345} +O_3O_{56}O_{61} +O_4O_{12}O_{61} +O_5O_{12}O_{23} -O_1O_2O_4O_{61} -O_1O_2O_5O_{23}\\
&-O_{23}O_{56}O_{345} -O_1O_3O_{61}O_{234} +O_1O_{23}O_{234}O_{345}
+\big(t^{\frac12}+t^{-\frac12}\big) \big(-O_2^2 -O_6^2 -O_{34}^2 -O_{45}^2\\
&-O_1O_2O_{12} +O_1O_6O_{61} +O_3O_4O_{34} +O_4O_5O_{45} -O_1O_3O_{12}O_{23} +O_1O_{23}O_{123} +O_3O_{12}O_{123}\\
&+O_1^2O_2O_3O_{23} +O_5O_{61}O_{234} +O_1^2O_2^2 -O_1^2O_{23}^2 -O_{123}^2 -O_{234}^2 -O_1O_2O_3O_{123}\big) +\big(t^{\frac12}+t^{-\frac12}\big)^3\\
=&r_0+\frac12O_2r_2 +\frac12O_{23}O_{234}r_4 -\frac12O_{61}O_{234}r_5 +\frac12(O_6-O_1O_{61})r_6 +\frac12(O_1O_2-O_{12})r_7\\
&-\frac12O_{23}r_8+\frac12O_2O_{234} -\frac12O_{56}r_{11} -\frac12O_{61}r_{12} +\frac12(O_1O_{23}-O_{123})r_{14} -\frac12 O_{234}r_{15}.
\end{align*}
\end{proof}
Now we are ready to prove the commutative counterpart of Theorem \ref{th:MCGActionGeneric}.
\begin{proposition}
Two automorphisms $d_1,I:\mathcal A_{q=1,t}\rightarrow\mathcal A_{q=1,t}$ define the action of Mapping Class Group $\mathrm{Mod}(\Sigma_2)$ of genus two surface on $\mathcal A_{q=1,t}$.
\label{prop:MCGActionAq1t}
\end{proposition}
\begin{proof}
Repeat calculations of (\ref{eq:MCGRelatorsOnGeneratorsGeneric}) for $q=1$, assuming generators commute. All equalities in (\ref{eq:MCGRelatorsOnGeneratorsGeneric}), where we have used normal ordering relations (\ref{eq:NormalOrderingRelations}) turn into verbatim identities in $\mathcal A_{q=1,t}$.
\end{proof}

\subsection{Quasiclassical limit of representation in difference operators}
\label{sec:QDifferenceQuasiclassicalLimit}

In this subsection we show that homomorphism $\widehat{\Delta}:\mathcal A_{q,t}\rightarrow A_{q,t}$ to the algebra of difference operators induces a homomorphism from $\mathcal A_{q=1,t}$ to rational functions in six variables, where the image $\Delta(O_I)$ of generators is obtained by the so-called quasiclassical limit of the corresponding $q$-difference operators $\widehat\Delta(O_I)$. To this end, we set
\begin{align*}
q=\mathrm e^{4\mathrm i\hbar},\qquad X_{12}=\mathrm e^{x_{12}},\qquad X_{23}=\mathrm e^{x_{23}},\qquad X_{13}=\mathrm e^{x_{13}}
\end{align*}
and consider the action of $q$-difference operator $\widehat{\Delta}(O_I)$ on Fourier harmonic
\begin{align*}
\widehat{\Delta}(O_I)\,\mathrm e^{\frac{\mathrm i}\hbar(p_{12}x_{12}+p_{23}x_{23}+p_{13}x_{13})},
\end{align*}
where we assume that $q$-shift operators $\hat\delta_{12}^{\pm1},\hat\delta_{23}^{\pm1},\hat\delta_{13}^{\pm1}$ act on $x_{12},x_{23},x_{13}$ by shifting the respective variable by $\mathrm i\hbar$,
\begin{equation}
\begin{aligned}
\hat\delta_{12}f(x_{12},x_{23},x_{13})=&f(x_{12}+\mathrm i\hbar,x_{23},x_{13}),\\
\hat\delta_{23}f(x_{12},x_{23},x_{13})=&f(x_{12},x_{23}+\mathrm i\hbar,x_{13}),\\
\hat\delta_{13}f(x_{12},x_{23},x_{13})=&f(x_{12},x_{23},x_{13}+\mathrm i\hbar).
\end{aligned}
\label{eq:ShiftOperatorsAction}
\end{equation}
\begin{proposition}
\phantom\newline
\begin{itemize}
\item For all generators $O_I\in\mathcal A_{q,t}$, the following limit exists
\begin{subequations}
\begin{align}
\Delta(O_I):=\lim_{\hbar\rightarrow0}\dfrac{\widehat{\Delta}(O_I)\left(\mathrm e^{\frac{\mathrm i}\hbar(p_{12}x_{12}+p_{23}x_{23}+p_{13}x_{13})}\right)}{\mathrm e^{\frac{\mathrm i}\hbar(p_{12}x_{12}+p_{23}x_{23}+p_{13}x_{13})}}.
\end{align}
\item Moreover, $\Delta(O_I)$ is a Laurent polynomial in three variables $P_{12},P_{23},P_{13}$
\begin{align}
\Delta(O_I)\in\mathbb C[t^{\pm\frac14}](X_{12},X_{23},X_{13})\big[P_{12}^{\pm1}, P_{23}^{\pm1}, P_{13}^{\pm1}\big],
\label{eq:DeltaImageOfGeneratorsLaurent}
\end{align}
where
\begin{align}
P_{12}=\mathrm e^{p_{12}},\qquad P_{23}=\mathrm e^{p_{23}},\qquad P_{13}=\mathrm e^{p_{13}}.
\end{align}
\item The action of $\Delta$ on generators extends to a homomorphism of commutative algebras
\begin{align}
\Delta:\mathcal A_{q=1,t}\rightarrow \mathbf k_t(X_{12},X_{23},X_{13},P_{12},P_{23},P_{13}).
\end{align}
\end{subequations}
\end{itemize}
\label{prop:qDifferenceQuasiClassicalLimit}
\end{proposition}
\begin{proof}
From (\ref{eq:ShiftOperatorsAction}) we immediately get
\begin{align*}
\hat\delta_{ij}\mathrm e^{\frac{\mathrm i}\hbar(p_{12}x_{12}+p_{23}x_{23}+p_{13}x_{13})}=P_{ij}\mathrm e^{\frac{\mathrm i}\hbar(p_{12}x_{12}+p_{23}x_{23}+p_{13}x_{13})},\qquad\textrm{for all}\quad (ij)\in\{(12),(23),(13)\}.
\end{align*}
Combining it with Lemma \ref{lemm:qDifferenceLaurentCoefficients} we immediately get the first two statements of the Proposition.

Now, in order to prove the last statement we need to show that all defining relations (\ref{eq:Aq1tDefiningIdeal}) of $\mathcal A_{q=1,t}$ are satisfied for $\Delta(O_I)$, images of generators. To this end, recall that each of the 19 defining relations of $\mathcal A_{q=1,t}$ has a noncommutative counterpart in $\mathcal A_{q,t}$. Namely, the 18 $J$-relations (\ref{eq:JRelations}) and the $q$-Casimir relation (\ref{eq:qCasimirRelation}). By Proposition \ref{prop:qDifferenceHomomorphism}, both (\ref{eq:JRelations}) and (\ref{eq:qCasimirRelation}) are satisfied for $q$-difference operators $\widehat{\Delta}(O_I)$. At the same time, the first two statements of the current Proposition imply that for every noncommutative monomial $O_{I_1}\dots O_{I_k}$ we have
\begin{align*}
\Delta(O_{I_1})\dots\Delta(O_{I_k}) =\lim_{\hbar\rightarrow0} \dfrac{\widehat{\Delta}(O_{I_1}\dots O_{I_k})\left(\mathrm e^{\frac{\mathrm i}\hbar(p_{12}x_{12}+p_{23}x_{23}+p_{13}x_{13})}\right)}{\mathrm e^{\frac{\mathrm i}\hbar(p_{12}x_{12}+p_{23}x_{23}+p_{13}x_{13})}}.
\end{align*}
As a corollary, all of the 19 defining relations of $\mathcal A_{q=1,t}$ must be satisfied for $\Delta(O_I)$.
\end{proof}

\section{Isomorphism between $\mathcal A_{q=t=1}$ and genus two character variety}
\label{sec:IsomorphismWithCharacterVariety}

Note that all Lemmas and Propositions of Section \ref{sec:ClassicalLimit} work verbatim if we replace formal variable $t^{\frac14}$ with a nonzero parameter. In this section we examine specialization of Definition \ref{def:A1t} corresponding to the value of parameter $t=1$. Indeed, consider a polynomial ring in 15 variables
\begin{align*}
 P=\mathbb C[O_1, O_2, O_3, O_4, O_5, O_6, O_{12}, O_{23}, O_{34}, O_{56}, O_{61}, O_{123}, O_{234}, O_{456}],
\end{align*}
and let $\mathcal I_{q=t=1}\subset P$ be an ideal generated by $t=1$ specialization of relators in (\ref{eq:Aq1tDefiningIdeal}). We define a quotient algebra
\begin{equation*}
\mathcal A_{q=t=1}:=\bigslant{P}{\mathcal I_{q=t=1}}.
\end{equation*}

The main goal of the current section is to prove that there exists a $\mathrm{Mod}(\Sigma_2)$-equivariant isomorphism
\begin{align}
\mathcal A_{q=t=1}\simeq\mathcal O(\mathrm{Hom}(\pi_1(\Sigma_2)),SL(2,\mathbb C))^{SL(2,\mathbb C)},
\label{eq:A1tIsomorphism}
\end{align}
between $\mathcal A_{q=t=1}$ and the coordinate ring of an $SL(2,\mathbb C)$-character variety of a closed genus two surface.

Turns out that we can compute Groebner basis of defining ideal $\mathcal I_{q=t=1}\subset P$ for some appropriate choice of monomial order. Without going into further details which won't be necessary until Section \ref{sec:WordProblem} we consider 61 polynomials in $P$ obtained by $q=t=1$ specialization of formulas listed in Appendix \ref{sec:qGroebnerBasis},
\begin{equation*}
g_j^{(q=t=1)}:=g_j\big|_{q=t=1}\in\mathcal I_{q=t=1},\qquad 1\leq j\leq 61.
\end{equation*}
\begin{proposition}
Collection of elements
\begin{align*}
\big\{g_i^{(q=t=1)}\;\big|\;1\leqslant i\leqslant61\big\}
\end{align*}
provides normalized Groebner basis for $\mathcal I_{q=t=1}\subset P$ w.r.t. weighted degree reverse lexicographic monomial order.
\label{prop:GroebnerbasisAqt1}
\end{proposition}

\subsection{Character variety}
\label{sec:CharacterVariety}
Genus two surface can be obtained as an identification space from an octagon with sides glued as shown on Figure \ref{fig:GenusTwoPiGenerators}.
\begin{figure}
\begin{tikzpicture}[scale=0.6]
\draw [dashed,darkgreen] ($0.8*(135:6)+0.2*(180:6)$) to ($0.8*(90:6)+0.2*(45:6)$);
\draw [darkgreen] (100:4.6) node {$c_2$};
\draw [dashed,darkgreen] ($0.8*(45:6)+0.2*(0:6)$) to ($0.8*(90:6)+0.2*(135:6)$);
\draw [darkgreen] (50:4.7) node {$c_1$};
\draw [dashed,darkgreen] ($0.8*(-135:6)+0.2*(-180:6)$) to ($0.8*(-90:6)+0.2*(-45:6)$);
\draw [darkgreen] (-130:4.7) node {$c_5$};
\draw [dashed,darkgreen] ($0.8*(-45:6)+0.2*(0:6)$) to ($0.8*(-90:6)+0.2*(-135:6)$);
\draw [darkgreen] (-80:4.6) node {$c_4$};
\draw [ultra thick,dotted,darkgreen] ($0.8*(-45:6)+0.2*(-90:6)$) to ($0.8*(0:6)+0.2*(45:6)$);
\draw [ultra thick,dotted,darkgreen] ($0.8*(-180:6)+0.2*(-135:6)$) to ($0.8*(135:6)+0.2*(90:6)$);
\draw [darkgreen] (-5:4.6) node {$c_3$};
\draw [darkgreen] (175:4.6) node {$c_3$};
\fill (0,0) circle (0.15);
\draw (0,0.5) node[above] {$p$};
\draw[->>-] [ultra thick] (0:6) to (45:6);
\draw[->-] [ultra thick] (45:6) to (90:6);
\draw[->>-] [ultra thick] (135:6) to (90:6);
\draw[->-] [ultra thick] (180:6) to (135:6);
\draw[->>>-] [ultra thick] (0:6) to (-45:6);
\draw[->>>>-] [ultra thick] (-45:6) to (-90:6);
\draw[->>>-] [ultra thick] (-135:6) to (-90:6);
\draw[->>>>-] [ultra thick] (-180:6) to (-135:6);
\draw [fill=black] (0:6) circle (0.12);
\draw [fill=black] (45:6) circle (0.12);
\draw [fill=black] (90:6) circle (0.12);
\draw [fill=black] (135:6) circle (0.12);
\draw [fill=black] (180:6) circle (0.12);
\draw [fill=black] (225:6) circle (0.12);
\draw [fill=black] (270:6) circle (0.12);
\draw [fill=black] (315:6) circle (0.12);
\draw[-<-] (0,0) to ($0.5*(0:6)+0.5*(45:6)$);
\draw[-<-] ($0.5*(90:6)+0.5*(135:6)$) to (0,0);
\draw (32:3) node {$Y_1$};
\draw (122:3) node {$Y_1$};
\draw[->-] (0,0) to ($0.5*(45:6)+0.5*(90:6)$);
\draw[->-] ($0.5*(180:6)+0.5*(135:6)$) to (0,0);
\draw (77:3) node {$X_1$};
\draw (167:3) node {$X_1$};
\draw[-<-] (0,0) to ($0.5*(-180:6)+0.5*(-135:6)$);
\draw[-<-] ($0.5*(-45:6)+0.5*(-90:6)$) to (0,0);
\draw (-58:3) node {$Y_2$};
\draw (-148:3) node {$Y_2$};
\draw[->-] ($0.5*(0:6)+0.5*(-45:6)$) to (0,0);
\draw[->-] (0,0) to ($0.5*(-90:6)+0.5*(-135:6)$);
\draw (-20:3) node [above] {$X_2$};
\draw (-110:3) node [right] {$X_2$};
\end{tikzpicture}
\caption{Choice of generators for $\pi_1(\Sigma_2,p)$.}
\label{fig:GenusTwoPiGenerators}
\end{figure}
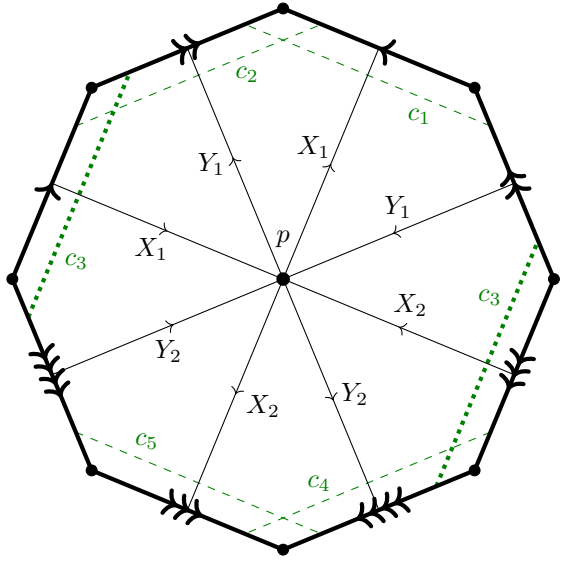
Let $p\in\Sigma_2$ be the center of an octagon and consider the fundamental group $\pi_1(\Sigma_2,p)$ based at $p$. We can choose generators $X_1,X_2,Y_1,Y_2$ of $\pi_1(\Sigma)$ corresponding to loops crossing the sides of an octagon exactly once. As a result, we obtain the following presentation for the fundamental group
\begin{align}
\pi_1(\Sigma_2,p)=\langle X_1,Y_1,X_2,Y_2\;|\; X_1Y_1X_1^{-1}Y_1^{-1}X_2Y_2X_2^{-1}Y_2^{-1}=1\rangle.
\label{eq:GenusTwoFundamentalGroup}
\end{align}
Hereinafter we use notation in which the composition of paths is read from right to left. Note that this is opposite to the standard convention in topology.

As follows from Corollary 40 in \cite{Sikora'2012} (see also \cite{RapinchukBenyash-KrivetzChernousov'1996}), the coordinate ring of the representation variety
$\mathcal O(\mathrm{Hom}(\pi_1(\Sigma_2),SL(2,\mathbb C)))$ is an integral domain. Hence, its invariant subring $\mathcal O(\mathrm{Hom}(\pi_1(\Sigma_2,p),SL(2,\mathbb C)))^{SL(2,\mathbb C)}$ must also be an integral domain. The Krull dimension of the latter can be computed by examining the general point of its spectrum, which corresponds to the orbit of an irreducible $SL(2,\mathbb C)$-representation of $\pi_1(\Sigma_2,p)$. This gives
\begin{align*}
\dim\; \mathrm{Hom}(\pi_1(\Sigma_2,p),SL(2,\mathbb C))\git SL(2,\mathbb C)=9-3=6.
\end{align*}

\subsection{$\mathrm{Mod}(\Sigma_2)$-equivariant homomorphism}
Mapping Class Group $\mathrm{Mod}(\Sigma_2)$ of a genus two surface acts on conjugacy classes of $\pi_1(\Sigma_2)$. This group is generated by five left Dehn twists $d_1,d_2,d_3,d_4,d_5$ about the cycles $c_1,c_2,c_3,c_4,c_5$ shown on Figure \ref{fig:GenusTwoPiGenerators}.\footnote{Note that cycles $c_1,c_2,c_3,c_4,c_5$ on Figure \ref{fig:GenusTwoPiGenerators} correspond to cycles $A_1,B_{12},A_2,B_{23},$ and $A_3$ on Figure \ref{fig:qDiffGeneratingCycles} respectively. So our notations for five generating Dehn twists are consistent with Theorem \ref{th:MCGActionGeneric} and Proposition \ref{prop:MCGActionAq1t}.} The action of the five Dehn twists on generators of the fundamental group is given by
\begin{equation}
\begin{aligned}
d_1:\;\left\{\begin{array}{ccl}
X_1&\mapsto& X_1Y_1\\
Y_1&\mapsto& Y_1\\
X_2&\mapsto& X_2\\
Y_2&\mapsto& Y_2
\end{array}\right.&,
\qquad
d_2:\;\left\{\begin{array}{ccl}
X_1&\mapsto& X_1\\
Y_1&\mapsto& Y_1X_1^{-1}\\
X_2&\mapsto& X_2\\
Y_2&\mapsto& Y_2
\end{array}\right.,
\qquad
d_3:\;\left\{\begin{array}{ccl}
X_1&\mapsto& Y_2Y_1X_1\\
Y_1&\mapsto& Y_1\\
X_2&\mapsto& Y_1Y_2X_2\\
Y_2&\mapsto& Y_2
\end{array}\right.,\\[10pt]
d_4:\;&\left\{\begin{array}{ccl}
X_1&\mapsto& X_1\\
Y_1&\mapsto& Y_1\\
X_2&\mapsto& X_2\\
Y_2&\mapsto& Y_2X_2^{-1}
\end{array}\right.,
\qquad
d_5:\;\left\{\begin{array}{ccl}
X_1&\mapsto& X_1\\
Y_1&\mapsto& Y_1\\
X_2&\mapsto& X_2Y_2\\
Y_2&\mapsto& Y_2
\end{array}\right..
\end{aligned}
\label{eq:DehnTwistsActionOnFundamentalGroup}
\end{equation}

Recall that (\ref{eq:IActionCommutative}) and (\ref{eq:d1ActionCommutative}) define a pair of automorphisms of the polynomial ring $P$ which we will denote by the same letters $d_1,I:P\rightarrow P$.
\begin{definition}
Let $\Psi$ be a homomorphism of the polynomial ring
\begin{align*}
\Psi:P\rightarrow \mathcal O(\mathrm{Hom}(\pi_1(\Sigma_2),SL(2,\mathbb C)))^{SL(2,\mathbb C)}
\end{align*}
defined on generators as
\begin{equation}
\begin{aligned}
    \Psi(O_1)=&\tau_{Y_1},
    &\Psi(O_{12})=&\tau_{Y_1X_1^{-1}},
    &\Psi(O_{123})=&\tau_{Y_1X_1^{-1}Y_1^{-1}Y_2^{-1}},\\
    \Psi(O_2)=&\tau_{X_1},
    &\Psi(O_{23})=&\tau_{X_1^{-1}Y_1^{-1}Y_2^{-1}},
    &\Psi(O_{234})=&\tau_{X_1Y_1Y_2^2X_2^{-1}Y_2^{-1}}\\
    \Psi(O_3)=&\tau_{Y_1Y_2},
    &\Psi(O_{34})=&\tau_{Y_2^{-1}Y_1^{-1}X_2},
    &\Psi(O_{345})=&\tau_{X_2Y_1^{-1}},\\
    \Psi(O_4)=&\tau_{X_2},
    &\Psi(O_{45})=&\tau_{Y_2X_2},\\
    \Psi(O_5)=&\tau_{Y_2},
    &\Psi(O_{56})=&\tau_{X_1Y_2^2X_2^{-1}Y_2^{-1}},\\
    \Psi(O_6)=&\tau_{X_1Y_2X_2^{-1}Y_2^{-1}},
    &\Psi(O_{61})=&\tau_{X_1Y_1Y_2X_2^{-1}Y_2^{-1}},
\end{aligned}
\label{eq:PsiActionOnGenerators}
\end{equation}
where for all $M\in\pi_1(\Sigma_2,p)$ we denote by
\begin{align*}
\tau_M\in \mathcal O(\mathrm{Hom}(\pi_1(\Sigma_2),SL(2,\mathbb C)))^{SL(2,\mathbb C)}
\end{align*}
the trace of the corresponding product of matrices.
\end{definition}
\begin{lemma}
Homomorphism $\Psi$ is equivariant with respect to the action of $d_1$ and $I$ on both sides.
\label{lemm:PsiMCGEquivariance}
\end{lemma}
\begin{proof}
From (\ref{eq:DehnTwistsActionOnFundamentalGroup}) we compute the action of $I=d_1d_2d_3d_4d_5$ on generators of the fundamental group:
\begin{align}
I:\quad\left\{\begin{array}{ccl}
X_1&\mapsto& Y_2Y_1,\\[5pt]
Y_1&\mapsto& X_1^{-1},\\[5pt]
X_2&\mapsto& (X_1^{-1}Y_2X_2)Y_2(X_1^{-1}Y_2X_2)^{-1},\\[5pt]
Y_2&\mapsto& Y_2(X_1^{-1}Y_2X_2)^{-1}.
\end{array}\right.
\label{eq:IActionOnFundamentalGroup}
\end{align}
On the character variety this translates to
\begin{align*}
I:\qquad\left\{\begin{array}{l}
 \tau_{Y_1}\quad\mapsto\quad \tau_{X_1^{-1}}=\tau_{X_1}\quad\mapsto\quad
 \tau_{Y_2Y_1}\quad\mapsto\quad \tau_{Y_2X_2^{-1}Y_2^{-1}}=\tau_{X_2}\quad\mapsto\\[5pt]
 \tau_{(X_1^{-1}Y_2X_2)Y_2(X_1^{-1}Y_2X_2)^{-1}}=\tau_{Y_2}\quad\mapsto\quad \tau_{Y_2X_2^{-1}Y_2^{-1}X_1}=\tau_{X_1Y_2X_2^{-1}Y_2^{-1}}\quad\mapsto\\[5pt] \tau_{Y_2Y_1Y_2(X_1^{-1}Y_2X_2)^{-1}(X_1^{-1}Y_2X_2)Y_2^{-1}(X_1^{-1}Y_2X_2)^{-1} (X_1^{-1}Y_2X_2)Y_2^{-1}}=\tau_{Y_1}
\end{array}
\right.
\end{align*}
The latter compares well to the cyclic permutation of elements in the first column of (\ref{eq:PsiActionOnGenerators}). Similar calculations verify the other two orbits of $I$ corresponding to second and third columns of (\ref{eq:PsiActionOnGenerators}). As a result, we conclude that
\begin{align*}
I(\Psi(O_J))=\Psi(I(O_J))\qquad\textrm{for all generators}\quad O_J\in P.
\end{align*}

Now we will prove that $\Psi$ is $d_1$-equivariant. From (\ref{eq:DehnTwistsActionOnFundamentalGroup}) we immediately note that $d_1$ acts identically on all $\tau$ not involving $X_1$, namely on
\begin{align*}
\tau_{Y_1},\quad \tau_{Y_1Y_2},\quad \tau_{X_2},\quad \tau_{Y_2},\quad  \tau_{Y_2^{-1}Y_1^{-1}X_2},\quad \tau_{Y_2X_2},\qquad \tau_{X_2Y_1^{-1}}.
\end{align*}
For the remaining 8 elements which appear on the r.h.s. of (\ref{eq:PsiActionOnGenerators}) we obtain
\begin{equation}
d_1:\qquad\left\{\begin{array}{ccl}
\tau_{X_1}&\mapsto& \tau_{X_1Y_1}=\tau_{X_1}\tau_{Y_1}-\tau_{Y_1X_1^{-1}},\\[5pt]
\tau_{X_1Y_2X_2^{-1}Y_2^{-1}}&\mapsto& \tau_{X_1Y_1Y_2X_2^{-1}Y_2^{-1},},\\[5pt]
\tau_{Y_1X_1^{-1}}&\mapsto& \tau_{X_1^{-1}}=\tau_{X_1},\\[5pt]
\tau_{X_1^{-1}Y_1^{-1}Y_2^{-1}}&\mapsto& \tau_{Y_1^{-1}X_1^{-1}Y_1^{-1}Y_2^{-1}} =\tau_{Y_1}\tau_{X_1^{-1}Y_1^{-1}Y_2^{-1}}-\tau_{Y_1X_1^{-1}Y_1^{-1}Y_2^{-1}}\\[5pt]
\tau_{X_1Y_2^2X_2^{-1}Y_2^{-1}}&\mapsto& \tau_{X_1Y_1Y_2^2X_2^{-1}Y_2^{-1}},\\[5pt]
\tau_{X_1Y_1Y_2X_2^{-1}Y_2^{-1}}&\mapsto& \tau_{X_1Y_1^2Y_2X_2^{-1}Y_2^{-1}} =\tau_{Y_1}\tau_{X_1Y_1Y_2X_2^{-1}Y_2^{-1}}-\tau_{X_1Y_2X_2^{-1}Y_2^{-1}},\\[5pt]
\tau_{Y_1X_1^{-1}Y_1^{-1}Y_2^{-1}}&\mapsto& \tau_{X_1^{-1}Y_1^{-1}Y_2^{-1}},\\[5pt]
\tau_{X_1Y_1Y_2^2X_2^{-1}Y_2^{-1}}&\mapsto& \tau_{X_1Y_1^2Y_2^2X_2^{-1}Y_2^{-1}} =\tau_{Y_1}\tau_{X_1Y_1Y_2^2X_2^{-1}Y_2^{-1}}-\tau_{X_1Y_2^2X_2^{-1}Y_2^{-1}}.
\end{array}\right.
\label{eq:d1ActionTau}
\end{equation}
Comparing (\ref{eq:d1ActionTau}) with (\ref{eq:d1ActionCommutative}) we conclude that
\begin{align*}
d_1(\Psi(O_J))=\Psi(d_1(O_J))\qquad\textrm{for all generators}\quad O_J\in P.
\end{align*}
\end{proof}

Now we are ready to prove that (\ref{eq:PsiActionOnGenerators}) defines a homomorphism from $\mathcal A_{q=t=1}$ to the coordinate ring of the genus two character variety.
\begin{proposition}
Defining ideal $\mathcal I_{q=t=1}\subset P$ of $\mathcal A_{q=t=1}$ belongs to the kernel of $\Psi$. In other words $\Psi$ descends to a homomorphism of the quotient algebra, which we denote by the same letter:
\begin{equation}
\Psi:\mathcal A_{q=t=1}\rightarrow\mathcal O(\mathrm{Hom}(\pi_1(\Sigma_2),SL(2,\mathbb C)))^{SL(2,\mathbb C)}.
\label{eq:PsiHomomorphismFromAq1t1}
\end{equation}
\label{prop:PsiHomomorphism}
\end{proposition}
\begin{proof}
Because $\Psi$ is equivariant w.r.t. $d_1$ and $I$, so is $\ker\Psi$. Recall that 19 defining relations (\ref{eq:Aq1tDefiningIdeal}) split into 4 orbits of $I$. Moreover, from (\ref{eq:d1r2RelatorActionCommutative}) and (\ref{eq:d1r8RelatorActionCommutative}) we know that
\begin{align*}
d_1(r_2)=-r_7,\qquad d_1(r_8)=r_{14}.
\end{align*}
Hence, it will be enough for us to prove that $r_0,r_1\in\ker\Psi$ and the rest will follow by equivariance. We get
\begin{equation}
\begin{aligned}
\Psi(r_1)=&\tau_{Y_1Y_2}\tau_{Y_2}+\tau_{X_2}\tau_{X_2Y_1^{-1}} -\tau_{Y_2^{-1}Y_1^{-1}X_2}\tau_{Y_2X_2}-2\tau_{Y_1},\\[5pt]
\Psi(r_0)=&8-\tau_{Y_1X_1^{-1}}^2-\tau_{Y_2X_2}^2-\tau_{X_1^{-1}Y_1^{-1}Y_2^{-1}}^2 -\tau_{Y_2^{-1}Y_1^{-1}X_2}^2-\tau_{X_1Y_1Y_2X_2^{-1}Y_2^{-1}}^2\\ &-\tau_{Y_1Y_2}\tau_{X_1Y_2X_2^{-1}Y_2^{-1}}\tau_{X_1Y_1Y_2^2X_2^{-1}Y_2^{-1}}
+\tau_{X_2Y_1^{-1}}\tau_{Y_1X_1^{-1}Y_1^{-1}Y_2^{-1}}\tau_{X_1Y_1Y_2^2X_2^{-1}Y_2}\\
&+\tau_{Y_1Y_2}\tau_{X_1^{-1}Y_1^{-1}Y_2^{-1}}\tau_{X_1}
+\tau_{Y_1Y_2}\tau_{Y_2^{-1}Y_1^{-1}X_2}\tau_{X_2}
-\tau_{X_1Y_2X_2^{-1}Y_2^{-1}}\tau_{X_1}\tau_{X_2}\\
&+\tau_{X_1Y_2X_2^{-1}Y_2^{-1}}\tau_{X_1Y_1Y_2X_2^{-1}Y_2^{-1}}\tau_{Y_1}
+\tau_{Y_1X_1^{-1}}\tau_{X_1}\tau_{Y_1}
-\tau_{X_2Y_1^{-1}}\tau_{X_2}\tau_{Y_1}\\
&+\tau_{X_1Y_2X_2^{-1}Y_2^{-1}}\tau_{X_1Y_2^2X_2^{-1}Y_2^{-1}}\tau_{Y_2}
-\tau_{Y_1X_1^{-1}Y_1^{-1}Y_2^{-1}}\tau_{X_1}\tau_{Y_2}
+\tau_{Y_2X_2}\tau_{X_2}\tau_{Y_2}
-\tau_{Y_1Y_2}\tau_{Y_1}\tau_{Y_2}.
\end{aligned}
\label{eq:PsiImageRelations01}
\end{equation}

Both expressions on the right hand side of (\ref{eq:PsiImageRelations01}) give rise to invariant polynomials in the coordinates of the full representation variety $\mathrm{Hom}(\pi_1(\Sigma_2,SL(2,\mathbb C)))$. We have used Mathematica to compute Groebner basis for defining ideal of the coordinate ring $\mathcal O(\mathrm{Hom}(\pi_1(\Sigma_2,SL(2,\mathbb C))))$ of the full representation variety in degree reverse lexicographic order. This basis has 82 elements. Using this basis we have reduced both expressions to confirm that both of them have zero remainder modulo defining ideal. We omit routine details from the main text and invite careful reader to examine our calculations at \cite{Arthamonov-GitHub-Flat}.
\end{proof}

\subsection{The Isomorphism}

For an $SL(2,\mathbb C)$-character variety of one-relator group (\ref{eq:GenusTwoFundamentalGroup}) we can utilize the algorithm from \cite{AshleyBurelleLawton'2018}. This algorithm provides a complete set of generators of the coordinate ring of character variety together with, possibly incomplete, set of relations which are enough however to provide a set-theoretic cut-out of the variety of irreducible representations.

This algorithm provides us with 14 generators:
\begin{align}
    \tau _{X_1},\;\tau _{Y_1},\;\tau _{X_2},\;\tau _{Y_2},\;\tau _{X_1Y_1},\;\tau _{X_1X_2},\;\tau _{X_1Y_2},\;\tau _{Y_1X_2},\;\tau _{Y_1Y_2},\;\tau _{X_2Y_2},\;\tau _{X_1Y_1X_2},\;\tau _{X_1Y_1Y_2},\;\tau _{X_1X_2Y_2},\;\tau _{Y_1X_2Y_2},
    \label{eq:ABL18Generators}
\end{align}
subject to 19 relations which we omit from the main text for the sake of brevity. Now let
\begin{equation*}
P^{\mathrm{ABL}}:=\mathbb C[\tau_{X_1}, \tau_{Y_1}, \tau_{X_2}, \tau_{Y_2}, \tau_{X_1Y_1}, \tau_{X_1X_2}, \tau_{X_1Y_2}, \tau_{Y_1X_2}, \tau_{Y_1Y_2}, \tau_{X_2Y_2}, \tau_{X_1Y_1X_2}, \tau_{X_1Y_1Y_2}, \tau_{X_1X_2Y_2}, \tau_{Y_1X_2Y_2}]
\end{equation*}
be the polynomial ring in 14 variables (\ref{eq:ABL18Generators}) and let $\mathcal I^{\mathrm{ABL}}\subset P^{\mathrm{ABL}}$ be the ideal generated by 19 relations mentioned above. By Theorem 3.4 from \cite{AshleyBurelleLawton'2018} (see also Theorem 3.2 in \cite{GonzalezMontesinos'1993}), we have
\begin{equation}
\bigslant{P^{\mathrm{ABL}}}{\sqrt{\mathcal I^{\mathrm{ABL}}}}\;\simeq\; \bigslant{\mathcal O(\mathrm{Hom}(\pi_1(\Sigma_2),SL(2,\mathbb C)))^{SL(2,\mathbb C)}}{\sqrt{(0)}} \;\stackrel{\textrm{Sec.} \ref{sec:CharacterVariety}}=\; \mathcal O(\mathrm{Hom}(\pi_1(\Sigma_2),SL(2,\mathbb C)))^{SL(2,\mathbb C)}.
\label{eq:ABL18SetTheoreticCutout}
\end{equation}

\begin{definition}
Let $\Phi$ be a homomorphism of the polynomial ring
\begin{align*}
\Phi:P^{\mathrm{ABL}}\rightarrow \mathcal A_{q=t=1}
\end{align*}
defined on generators as
\begin{align}
\Phi(\tau _{X_1})=& O_2,&
\Phi(\tau _{Y_2})=& O_5,&
\Phi(\tau _{X_1Y_1})=& O_1 O_2-O_{1,2},\nonumber\\
\Phi(\tau _{Y_1})=& O_1,&
\Phi(\tau _{Y_1Y_2})=& O_3,&
\Phi(\tau _{Y_1X_2})=& O_1 O_4-O_{3,4,5},\nonumber\\
\Phi(\tau _{X_2})=& O_4,&\Phi(\tau _{X_2Y_2})=& O_{4,5},&
\Phi(\tau _{Y_1X_2Y_2})=& O_{3,4}+O_1 O_{4,5}-O_5 O_{3,4,5},
\label{eq:PhiActionOnGenerators}
\end{align}
\begin{align}
\Phi(\tau _{X_1X_2})=& -O_1 O_{6,1}+O_1 O_2 O_{3,4,5}-O_{1,2} O_{3,4,5}+O_6,\nonumber\\
\Phi(\tau _{X_1Y_2})=& O_3 O_{1,2}+O_1 O_{2,3}-O_{1,2,3}-O_1 O_2 O_3+O_2 O_5,\nonumber\\
\Phi(\tau _{X_1Y_1X_2})=&- O_1^2 O_{6,1}+O_2 O_1^2 O_{3,4,5}-O_1 O_{1,2} O_{3,4,5}+O_{6,1}-O_2 O_{3,4,5}+O_6 O_1,\nonumber\\
\Phi(\tau _{X_1Y_1Y_2})=& -O_5 O_{1,2}-O_4 O_5 O_{6,1}-O_{2,3}+O_{4,5} O_{6,1}+ O_4 O_{2,3,4}+O_1 O_2 O_5,\nonumber\\
\Phi(\tau _{X_1X_2Y_2})=& -O_1 O_2 O_{3,4}+O_{1,2} O_{3,4}-O_{5,6}-O_1 O_5 O_{6,1}+O_1 O_{2,3,4}+O_1 O_2 O_5 O_{3,4,5}-O_5 O_{1,2} O_{3,4,5}+O_5 O_6.\nonumber
\end{align}
\end{definition}
\begin{lemma}
Ideal $\mathcal I^{\mathrm{ABL}}\subset P^{\mathrm{ABL}}$ is annihilated by the above homomorphism:
\begin{equation*}
\mathcal I^{\mathrm{ABL}}\subset\ker\Phi.
\end{equation*}
\label{lemm:ABL18RelationsAnnihilatedByPhi}
\end{lemma}
\begin{proof}
To this end we calculate the Groebner basis of defining ideal of $\mathcal A_{q=t=1}$ and reduce the $\Phi$-image of generators of $\mathcal I^{\mathrm{ABL}}$ modulo this ideal. In particular, for the first generator of $\mathcal I^{\mathrm{ABL}}$ we get
\begin{equation*}
\begin{aligned}
&\Phi\Big(\tau _{X_1X_2}^2+\tau _{X_1}^2+\tau _{X_2}^2+\tau _{X_1X_2} \left(-\tau _{X_1} \tau _{X_2}+\tau _{X_1Y_1} \tau _{Y_1X_2}-\tau _{Y_1} \tau _{X_1Y_1X_2}\right)+\tau _{X_1Y_1}^2+\tau _{Y_1X_2}^2+\tau _{X_1Y_1X_2}^2\\
&\quad\qquad-\tau _{X_1} \tau _{Y_1X_2} \tau _{X_1Y_1X_2}-\tau _{X_2} \tau _{Y_1} \tau _{Y_1X_2}+\tau _{X_1} \tau _{X_2} \tau _{Y_1} \tau _{X_1Y_1X_2}-\tau _{X_1Y_1} \left(\tau _{X_2} \tau _{X_1Y_1X_2}+\tau _{X_1} \tau _{Y_1}\right)+\tau _{Y_1}^2-4\Big)\\
&\quad=36\Big(O_1 O_2 O_6 O_{345} -O_6 O_{12} O_{345} -O_2 O_{61} O_{345} +O_4 O_{12} O_{61} -O_1 O_4 O_{345} +O_{345}^2\\
&\quad\qquad-O_1 O_2 O_{12} -O_1 O_6 O_{61} -O_2 O_4 O_6 +O_{12}^2 +O_{61}^2 +O_1^2 +O_2^2 +O_4^2 +O_6^2 -4\Big)\\
&\quad=36 \left(O_4 g_2^{(q=t=1)}-g_{18}^{(q=t=1)}+g_{47}^{(q=t=1)}\right)\equiv 0\bmod{\mathcal I}_{q=t=1}.
\end{aligned}
\end{equation*}

For the sake of brevity we omit similar calculations for the remaining generators of $\mathcal I^{\mathrm{ABL}}$. A complete set of formulas along with a Mathematica algorithm used to obtain them can be found in supplementary materials to this paper \cite{Arthamonov-GitHub-Flat}.
\end{proof}

\begin{lemma}
The nilradical $\sqrt{(0)}\subset\mathcal A_{q=t=1}$ of commutative algebra $\mathcal A_{q=t=1}$ is trivial.
\label{lemm:NilradicalAq1t1IsTrivial}
\end{lemma}
\begin{proof}
To prove this statement we utilize computer algebra software SINGULAR with the source code of our computations available at \cite{Arthamonov-GitHub-Flat}. Due to the high computational complexity we cannot directly use the built-in function which involves making arbitrary choices while disrespecting the symmetry of the problem. Instead, we follow essentially the same algorithm but utilizing the symmetry of the defining ideal in our choices. This makes computations possible on a standard computer.

We reduce calculation of the nilradical to a zero-dimensional case as described in Section 4.2 of \cite{GreuelPfisterBachmannLossenSchonemann'2008}. On the first step we choose a maximal algebraically independent subset
\begin{equation}
O_1,O_2,O_3,O_4,O_5,O_6\quad\in\quad P.
\label{eq:MaximalAlgebraicallyIndependentSet}
\end{equation}
and consider an algebra
\begin{equation*}
\mathcal R:=\mathbb C(O_1,O_2,O_3,O_4,O_5,O_6)[O_{12},O_{23},O_{34},O_{56},O_{61},O_{123},O_{234},O_{345}].
\end{equation*}
over the field of rational functions in variables (\ref{eq:MaximalAlgebraicallyIndependentSet}).

Now, using built-in function we verify that zero-dimensional ideal $\mathcal I_{q=t=1}\mathcal R\subset\mathcal R$ equals to its own radical:
\begin{equation*}
\sqrt{\mathcal I_{q=t=1}\mathcal R}=\mathcal I_{q=t=1}\mathcal R\qquad\subset\quad\mathcal R.
\end{equation*}

Next, we compute the Groebner basis of $\mathcal I_{q=t=1}\mathcal R$ in weighted degree reverse lexicographic order and clear denominators to obtain a finite set of polynomials $S\subset P$. The least common multiple of leading coefficients in $S$ reads
\begin{equation*}
h=O_1O_2O_3O_4O_5O_6(O_1^2-O_4^2)(O_2^2-O_5^2)(O_3^2-O_6^2)(O_1O_2O_6-O_3O_4O_5)\qquad\in\quad P.
\end{equation*}
We verify using SINGULAR that
\begin{equation*}
\mathcal I_{q=t=1}=\mathcal I_{q=t=1}:\langle h\rangle=\mathcal I_{q=t=1}:\langle h^2\rangle,
\end{equation*}
so the original six-dimensional ideal equals to the saturation
\begin{equation*}
\mathcal I_{q=t=1}=\mathcal I_{q=t=1}:\langle h^\infty\rangle.
\end{equation*}
As a corollary, dimensional recursion terminates on the first step and we obtain
\begin{equation*}
\sqrt{\mathcal I_{q=t=1}}=\sqrt{\mathcal I_{q=t=1}\mathcal R}\cap P=\mathcal I_{q=t=1}\mathcal R\cap P=\mathcal I_{q=t=1}.
\end{equation*}
\end{proof}

\begin{proposition}
Radical ideal $\sqrt{I^{\mathrm{ABL}}}\subset P^{\mathrm{ABL}}$ is annihilated by $\Phi$. In other words $\Phi$ descends to a homomorphism of the quotient ring, which we denote by the same letter
\begin{equation}
\Phi: \bigslant{P^{\mathrm{ABL}}}{\sqrt{\mathcal I^{\mathrm{ABL}}}}\rightarrow \mathcal A_{q=t=1}.
\label{eq:PhiFromRadicalQuotient}
\end{equation}
\end{proposition}
\begin{proof}
Combining Lemma \ref{lemm:ABL18RelationsAnnihilatedByPhi} with Lemma \ref{lemm:NilradicalAq1t1IsTrivial} we get
\begin{equation*}
\Phi\Big(\sqrt{\mathcal I^{\mathrm{ABL}}}\Big)\subset \sqrt{\Phi(\mathcal I^{\mathrm{ABL}})}= \sqrt{(0)}=(0).
\end{equation*}
\end{proof}

\begin{theorem}
We have an $\mathrm{Mod}(\Sigma_2)$-equivariant isomorphism of commutative algebras
\begin{equation*}
\begin{tikzcd}
\Psi:\mathcal A_{q=t=1}\ar[r,rightarrow,"\sim"]&\mathcal O(\mathrm{Hom}(\pi_1(\Sigma_2),SL(2,\mathbb C)))^{SL(2,\mathbb C)}.
\end{tikzcd}
\end{equation*}
\label{th:IsomorphismAq1t1CoordinateRingCharacterVariety}
\end{theorem}
\begin{proof}
Combining (\ref{eq:PsiHomomorphismFromAq1t1}), (\ref{eq:ABL18SetTheoreticCutout}), and (\ref{eq:PhiFromRadicalQuotient}) we get a system of homomorphisms which can be arranged on the following diagram
\begin{equation*}
\begin{tikzcd}
\bigslant{P^{\mathrm{ABL}}}{\sqrt{\mathcal I^{\mathrm{ABL}}}}\ar[r,rightarrow,"\Phi"]&\mathcal A_{q=t=1}\ar[dd,rightarrow,"\Psi"]\\\\
&\mathcal O(\mathrm{Hom}(\pi_1(\Sigma_2),SL(2,\mathbb C)))^{SL(2,\mathbb C)}\ar[luu,rightarrow,"\sim","\iota"']
\end{tikzcd}
\end{equation*}
Here $\iota$ stands for the isomorphism which identifies $SL(2,\mathbb C)$-invariant elements of the coordinate ring of representation variety with the elements of the quotient ring realizing its particular presentation.

From Lemma \ref{lemm:PsiMCGEquivariance} we already know that $\Psi$ is $\mathrm{Mod}(\Sigma_2)$-equivariant, so the only thing we have to prove is that it is bijective. To this end we verify that
\begin{equation}
\Phi\circ\iota\circ\Psi=\mathrm{Id}_{\mathcal A_{q=t=1}},\qquad \Psi\circ\Phi\circ\iota=\mathrm{Id}_{\mathcal O(\mathrm{Hom}(\pi_1(\Sigma_2)),SL(2,\mathbb C))^{SL(2,\mathbb C)}}.
\label{eq:IdentityCompositionPhiPsi}
\end{equation}
Note that we can compute Groebner bases for defining ideals of both $\mathcal A_{q=t=1}$ and $\mathcal O(\mathrm{Hom}(\pi_1(\Sigma_2),SL(2,\mathbb C)))$, the coordinate ring of the full representation variety. This, combined with the fact that (\ref{eq:ABL18Generators}) provides a complete set of generators of the invariant subring, allows us to perform the calculations. We omit the details of the verification from the main text and instead summarize them in Appendix \ref{sec:IdentiyCompositionPsiPhi}.
\end{proof}

\begin{corollary}
Algebra $\mathcal A_{q=t=1}$ is an integral domain of Krull dimension 6.
\label{cor:A1tIntegralDomain}
\end{corollary}

\section{Word problem and monomial basis}
\label{sec:WordProblem}

In this section we prove that the three algebras $\mathcal A_{q,t}$, $\mathcal A_{q=1,t}$, and $\mathcal A_{q=t=1}$ all share the same monomial basis. Our approach is based on the computation of Groebner basis for defining ideal of $\mathcal A_{q=1,t}$. It turns out that the result specializes to the Groebner basis for defining ideal of $\mathcal A_{q=t=1}$. But, even more importantly, this Groebner basis admits a noncommutative deformation and gives rise to a family of relations in $\mathcal A_{q,t}$:
\begin{align*}
g_i=0,\qquad 1\leqslant i\leqslant 61.
\end{align*}
We refer to this deformation as \textit{$q$-Groebner basis}. Relators $g_i$ are listed in Appendix \ref{sec:qGroebnerBasis}, each of the relators
\begin{align}
g_i\;\in\;\mathbb C[q^{\pm\frac14},t^{\pm\frac14}]\langle O_1,O_2,O_3,O_4,O_5,O_6,O_{12},O_{23},O_{13},O_{123},O_{234},O_{345}\rangle
\label{eq:giDomain}
\end{align}
is a noncommutative polynomial in generators (\ref{eq:15Generators}) with coefficient 1 in the leading monomial for weighted degree reverse lexicographic order. The coefficients in subleading monomials are Laurent polynomials in $q^{\frac14},t^{\frac14}$ and hence specialize well to $q=1$ and to $q=t=1$. The relations $g_i|_{q=1}$ and $g_i|_{q=t=1}$ provide Groebner basis for $\mathcal A_{q=1,1}$ and $\mathcal A_{q=t=1}$ respectively.

\subsection{Monomial basis for $\mathcal A_{q=1,t}$ and $\mathcal A_{q=t=1}$}
Denote the natural images of relators $g_i$ in commutative polynomials by
\begin{align}
g_i\big|_{q=1}\;\in\;\mathbb C[t^{\pm\frac14}][O_1,\dots, O_{345}],\qquad\qquad
g_i\big|_{q=t=1}\;\in\;\mathbb C[O_1,\dots, O_{345}].
\label{eq:giSpecializedDomain}
\end{align}

In our presentations for commutative algebras $\mathcal A_{q=1,t}$ and $\mathcal A_{q=t=1}$ we fix weighted degree reverse lexicographic monomial order with weights assigned to generators by (\ref{eq:GeneratorWeights}).
\begin{proposition}
Collection of relators
\begin{align}
\big\{g_i|_{q=1}\;\big|\;1\leqslant i\leqslant61\big\}
\label{eq:GBAq1t}
\end{align}
provides normalized Groebner basis for defining ideal of $\mathcal A_{q=1,t}$ w.r.t. the fixed choice of monomial order.
\label{prop:GroebnerBasisAq1t}
\end{proposition}
\begin{proof}
We have verified the statement of the proposition using two independent programs in SINGULAR and Mathematica which are available at \cite{Arthamonov-GitHub-Flat}. With our choice of monomial ordering the computation doesn't require any significant resources and can be performed independently using virtually any computer algebra software.
\end{proof}

It is worth noting, that (\ref{eq:GBAq1t}) is a Groebner basis over $\mathbf k_t=\mathbb C(t^{\frac14})$, the ground field of $\mathcal A_{q=1,t}$. A priori, it doesn't necessarily provide Groebner basis when $t$ is specialized to some complex parameter. However, specialization $t=1$ turns out to be good due to the fact that all subleading coefficients in normalized elements $g_i\big|_{q=1}$ happen to be Laurent polynomials in $t^{\frac14}$. Namely, comparing Proposition \ref{prop:GroebnerBasisAq1t} with Proposition \ref{prop:GroebnerbasisAqt1} we get
\begin{theorem}
Algebra $\mathcal A_{q=1,t}$ is a flat deformation of $\mathcal A_{q=t=1}$.
\label{th:FlatDeformationAq1t}
\end{theorem}
\begin{proof}
Indeed, both algebras have monomial basis which consists of monomials which are not divisible by the leading powers or normalized Groebner bases
\begin{align*}
\mathrm{l.p.}\; g_i|_{q=1}=\mathrm{l.p.}\; g_i|_{q=t=1},\qquad \textrm{for all}\quad 1\leqslant i\leqslant 61.
\end{align*}
We denote this common monomial basis by $B$.

Structure constants of $\mathcal A_{q=1,t}$ in basis $B$ can be computed using reduction of the product of two monomials by Groebner relators (\ref{eq:GBAq1t}). Recall that by (\ref{eq:giSpecializedDomain}) all subleading coefficients in Groebner relators are Laurent polynomials in $t^\frac14$. As a corollary, so are the structure constants
\begin{align}
C_{b_1b_2}^{b_3}\;\in\; \mathbb C[t^{\pm\frac14}] \qquad\textrm{for all}\quad b_1,b_2,b_3\in B.
\label{eq:StructureConstantsAq1t}
\end{align}
Because normalized Groebner relators $g_i|_{t=1}$ of $\mathcal A_{q=1,t}$ specialize to Groebner relators of $\mathcal A_{q=t=1}$ at $t=1$, we conclude that structure constants of $\mathcal A_{q=t=1}$ in monomial basis $B$ can be obtained
using the $t=1$ specialization of (\ref{eq:StructureConstantsAq1t}).
\end{proof}

\subsection{Embedding into rational functions}

Homomorphism of commutative algebras $\Delta$, which we introduced in Proposition \ref{prop:qDifferenceQuasiClassicalLimit}, specializes nicely to $t=1$. This is shown by
\begin{lemma}
The following action on generators
\begin{equation*}
\widetilde{\Delta}(O_I):=\Delta(O_I)\Big|_{t=1}
\end{equation*}
extends to a homoorphism of commutative algebras
\begin{align*}
\widetilde{\Delta}:\mathcal A_{q=t=1}\rightarrow\mathbb C(X_{12},X_{23},X_{13})[P_{12}^{\pm1},P_{23}^{\pm1},P_{13}^{\pm1}].
\end{align*}
\end{lemma}
\begin{proof}
Note that image of generators in (\ref{eq:DeltaImageOfGeneratorsLaurent}) has coefficients which are Laurent polynomials in $t^{\frac14}$. As a corollary, any relation that holds between elements $\Delta(O_I)$, when specialized to $t=1$, must also hold between the $\widetilde{\Delta}(O_I)$.
\end{proof}

\begin{proposition}
We have an isomorphism
\begin{align}
\mathcal A_{q=t=1}\simeq \widetilde\Delta(\mathcal A_{q=t=1})\quad\subset\quad\mathbb C(X_{12},X_{23},X_{13})[P_{12}^{\pm1},P_{23}^{\pm1},P_{13}^{\pm1}]
\label{eq:Aqt1IsomorphismRationalFunctions}
\end{align}
between $\mathcal A_{q=t=1}$ and a subalgebra of rational functions in six variables.
\label{prop:Aqt1IsomorphismRationalFunctions}
\end{proposition}
\begin{proof}
Recall that $\widetilde\Delta$ is a homomorphism of commutative algebras by Proposition \ref{prop:qDifferenceQuasiClassicalLimit}. Since $\mathcal A_{q=t=1}$ is finitely generated, then so is $\widetilde\Delta(\mathcal A_{q=t=1})$. On the other hand, $\widetilde\Delta(\mathcal A_{q=t=1})$ is a subring of an integral domain, hence it is an integral domain on its own. Now we will show that Krull dimension of $\widetilde\Delta(\mathcal A_{q=t=1})$ is bounded from below
\begin{align}
\dim(\widetilde\Delta(\mathcal A_{q=t=1}))\geqslant 6.
\label{eq:DimLowerBoundDeltaAqt1}
\end{align}
To this end, consider
\begin{align*}
\widetilde\Delta(O_1)=&\sum\limits_{a,b \in \{\pm 1\}} \ a b \ \dfrac{(1 - X_{23} X_{12}^a X_{13}^b)(1 - X_{23}^{-1} X_{12}^a X_{13}^b)}{X_{12}^{a} X_{13}^b (X_{12} - X_{12}^{-1})(X_{13} - X_{13}^{-1})} \ P_{12}^{a} P_{13}^{b},\\[5pt]
\widetilde\Delta(O_2)=&X_{12}+X_{12}^{-1},\\[5pt]
\widetilde\Delta(O_3)=&\sum\limits_{a,b \in \{\pm 1\}} \ a b \ \dfrac{(1 - X_{13} X_{12}^a X_{23}^b)(1 - X_{13}^{-1} X_{12}^a X_{23}^b)}{X_{12}^{a} X_{23}^b (X_{12} - X_{12}^{-1})(X_{23} - X_{23}^{-1})} \ P_{12}^{a} P_{23}^{b},\\[5pt]
\widetilde\Delta(O_4)=&X_{23}+X_{23}^{-1},\\[5pt]
\widetilde\Delta(O_5)=& \ \sum\limits_{a,b \in \{\pm 1\}} \ a b \ \dfrac{(1 - X_{12} X_{13}^a X_{23}^b)(1 - X_{12}^{-1} X_{13}^a X_{23}^b)}{ X_{13}^{a} X_{23}^b (X_{13} - X_{13}^{-1})(X_{23} - X_{23}^{-1})} \ P_{13}^{a} P_{23}^{b},\\[5pt]
\widetilde\Delta(O_6)=&X_{13}+X_{13}^{-1}.
\end{align*}
We claim that the six elements above are algebraically independent. Indeed, suppose there is a polynomial relation
\begin{align*}
F\Big(\widetilde\Delta(O_1), \widetilde\Delta(O_2), \widetilde\Delta(O_3), \widetilde\Delta(O_4), \widetilde\Delta(O_5), \widetilde\Delta(O_6)\Big)=0.
\end{align*}
Consider the leading powers in $P_{12},P_{23},P_{13}$, we have
\begin{align}
\mathrm{l.p.}\,\widetilde\Delta(O_1)=&P_{12}P_{13},&
\mathrm{l.p.}\,\widetilde\Delta(O_3)=&P_{12}P_{23},&
\mathrm{l.p.}\,\widetilde\Delta(O_5)=&P_{13}P_{23}
\label{eq:LeadingPowersDeltaO}
\end{align}
From (\ref{eq:LeadingPowersDeltaO}) we conclude that $F$ cannot depend on $\widetilde\Delta(O_1), \widetilde\Delta(O_3), \widetilde\Delta(O_5)$. But then, because there is no polynomial relations on $\widetilde\Delta(O_2), \widetilde\Delta(O_4), \widetilde\Delta(O_6)$ we conclude that $F$ is a zero polynomial. Applying Noether Normalization Theorem we get (\ref{eq:DimLowerBoundDeltaAqt1}).

Because $\widetilde\Delta(\mathcal A_{q=t=1})$ is an integral domain, $\ker\widetilde\Delta\subseteq\mathcal A_{q=t=1}$ must be a prime ideal of $\mathcal A_{q=t=1}$. Combining Corollary \ref{cor:A1tIntegralDomain} with (\ref{eq:DimLowerBoundDeltaAqt1}) we conclude that $\ker\widetilde\Delta$ must be of height zero and thus $\ker\widetilde\Delta=\{0\}$.
\end{proof}

Now, using the common monomial basis for $\mathcal A_{q=1,t}$ and $\mathcal A_{q=t=1}$ we can prove the analog of Proposition \ref{prop:Aqt1IsomorphismRationalFunctions} for $\mathcal A_{q=1,t}$.
\begin{proposition}
We have an isomorphism
\begin{align*}
\mathcal A_{q=1,t}\simeq \Delta(\mathcal A_{q=1,t})\quad\subset\quad\mathbb C(X_{12},X_{23},X_{13})[P_{12}^{\pm1},P_{23}^{\pm1},P_{13}^{\pm1}]
\end{align*}
between $\mathcal A_{q=1,t}$ and a subalgebra of rational functions in six variables.
\label{prop:Aq1tIsomorphismRationalFunctions}
\end{proposition}
\begin{proof}
Consider a monomial basis $B\subset\mathcal A_{q=1,t}$ and suppose that there exists at least one nontrivial linear combination of basis elements $P=P(O_1,\dots,O_{345})\in\ker\Delta$. Multiplying it by an appropriate power of $(t-1)$ we can always make coefficients in all monomials regular with at least one coefficient having nonzero limit as $t\rightarrow 1$. So without loss of generality we can assume that $P_{t=1}=\lim_{t\rightarrow 1}P(O_1,\dots,O_{345})$ exists and nontrivial.

By definition of $\ker\Delta$ we have
\begin{align}
P(\Delta(O_1),\dots,\Delta(O_{345}))=0.
\label{eq:RelationInKerDelta}
\end{align}
On the other hand, by Lemma \ref{lemm:qDifferenceLaurentCoefficients}, $\Delta(O_I)$ is a Laurent polynomial in $t^{\frac14}$ for every generator $O_I\in\mathcal A_{q=1,t}$. Combining it with the fact that all coefficients of $P$ are regular at $t=1$ we conclude that the left hand side of (\ref{eq:RelationInKerDelta}) converges as $t\rightarrow 1$. As a corollary,
\begin{align*}
P_{t=1}(\widetilde\Delta(O_1),\dots\widetilde\Delta(O_{345}))=0
\end{align*}
holds in $\mathbb C(X_{12},X_{23},X_{12})[P_{12}^{\pm1},P_{23}^{\pm1},P_{13}^{\pm1}]$.

Next, by isomorphism (\ref{eq:Aqt1IsomorphismRationalFunctions}) we obtain a relation $P_{t=1}(O_1,\dots,O_{345})=0$ in $\mathcal A_{q=t=1}$. However, because $B$ is also a monomial basis for $\mathcal A_{q=t=1}$ here we come to contradiction with our initial assumption that $P$ is a regular linear combination of basis elements with nontrivial limit as $t\rightarrow 1$.
\end{proof}

\subsection{Monomial basis for $\mathcal A_{q,t}$}

Based on the existence of $q$-Groebner basis we introduce an algorithm which brings any monomial in $\mathcal A_{q,t}$ into normal form by reducing it with respect to the $q$-Groebner basis. Finally, we prove that this normal form is unique, using Lemma \ref{lemm:qDifferenceLaurentCoefficients} and Proposition \ref{prop:Aq1tIsomorphismRationalFunctions}.

By Lemma \ref{lemm:NormalOrderingLemma} we know that normally ordered monomials of the form (\ref{eq:NormallyOrderedMonomials}) provide a spanning set for $\mathcal A_{q,t}$. On the set of normally ordered monomials we can introduce a total order $\prec$ by comparing their commutative counterparts. The following notation will often be convenient for us when we deal with subleading terms.
\begin{definition}
Let $m=O_1^{k_1}O_2^{k_2}\dots O_{345}^{k_{345}}$ be a normally ordered product of generators. We say that $f\in\mathcal A_{q,t}$ is Laurent-subleading to $m$ if
\begin{subequations}
\begin{equation}
f=\sum_{j=1}^n\lambda_jO_1^{l_1^{(j)}}O_2^{l_2^{(j)}}\dots O_{345}^{l_{345}^{(j)}}
\end{equation}
where
\begin{equation}
\lambda_j\in\mathbb C[q^{\pm\frac14},t^{\pm\frac14}],\qquad O_1^{l_1^{(j)}}O_2^{l_2^{(j)}}\dots O_{345}^{l_{345}^{(j)}}\prec O_1^{k_1}O_2^{k_2}\dots O_{345}^{k_{345}}\qquad\textrm{for all}\quad 1\leq j\leq n.
\end{equation}
\label{eq:LaurentSubleadingElement}
\end{subequations}
In this case we write $f=\boldsymbol\sigma(O_1^{k_1}O_2^{k_2}\dots O_{345}^{k_{345}})$.
\label{def:LaurentSubleading}
\end{definition}
\begin{remark}
Note that we allowed only Laurent polynomials in coefficients of (\ref{eq:LaurentSubleadingElement}). This will be important for us later when we deal with various specializations of parameters $q^{\frac14},t^{\frac14}$.
\end{remark}

We start by stating an immediate consequence of Definition \ref{def:LaurentSubleading} which we will then use without further mentioning.
\begin{lemma}
For a pair of normally ordered products of generators.
\begin{equation*}
m=O_1^{k_1}O_2^{k_2}\dots O_{345}^{k_{345}},\qquad\mu=O_1^{\kappa_1}O_2^{\kappa_2}\dots O_{345}^{\kappa_{345}},\qquad m\preceq\mu
\end{equation*}
we have
\begin{equation*}
\boldsymbol\sigma(m)+\boldsymbol\sigma(\mu)=\boldsymbol\sigma(\mu).
\end{equation*}
\end{lemma}

\begin{proposition}
Let $m=O_1^{k_1}O_2^{k_2}\dots O_{345}^{k_{345}}$ be a normally ordered monomial and $O_I\in\mathcal A_{q,t}$ be any generator. The following properties hold in $\mathcal A_{q,t}$
\begin{subequations}
\begin{align}
O_Im=&q^{-\frac12(k_1c_{I,1}+\dots+k_{I'}c_{I,I'})}O_1^{k_1}\dots O_I^{k_I+1}\dots O_{345}^{k_{345}}\;+\;\boldsymbol\sigma(O_1^{k_1}\dots O_I^{k_I+1}\dots O_{345}^{k_{345}}),\\
mO_I=&q^{-\frac12(k_{I''}c_{I'',I}+\dots+k_{345}c_{345,I})}O_1^{k_1}\dots O_I^{k_I+1}\dots O_{345}^{k_{345}}\;+\;\boldsymbol\sigma(O_1^{k_1}\dots O_I^{k_I+1}\dots O_{345}^{k_{345}}),
\end{align}
\begin{align}
O_I\boldsymbol\sigma(m)=&\boldsymbol\sigma(O_1^{k_1}\dots O_I^{k_I+1}\dots O_{345}^{k_{345}})=\boldsymbol\sigma(m)O_I
\end{align}
\label{eq:SigmaPropertiesSingleGenerator}
\end{subequations}
here $I'$ stands for the index of a previous generator, while $I''$ stands for the next generator following $O_I$ in (\ref{eq:15Generators}).
\label{prop:LeadingTermMonomialProduct}
\end{proposition}
\begin{proof}
We will prove all three properties in (\ref{eq:SigmaPropertiesSingleGenerator}) by simultaneous induction in $|k|:=k_1+k_2+\dots+k_{345}$. The base case $|k|=0$ is a tautology. As for the step of induction, recall that normal ordering relations (\ref{eq:NormalOrderingRelations}) imply
\begin{align*}
O_KO_J=q^{-\frac{c_{K,J}}2}O_JO_K\;+\;\boldsymbol\sigma(O_JO_K)
\end{align*}
\end{proof}

\begin{corollary}
Let $p=O_{I_1}O_{I_2}\dots O_{I_n}$ be any product of generators, not necessarily normally ordered. There exists a unique integer $M(I_1,\dots,I_n)\in\mathbb Z$ such that
\begin{equation*}
p=q^{\frac{M(I_1,\dots,I_n)}2}O_1^{k_1}O_2^{k_2}\dots O_{345}^{k_{345}}\;+\;\boldsymbol\sigma(O_1^{k_1}O_2^{k_2}\dots O_{345}^{k_{345}}).
\end{equation*}
Here each $k_J$ stands for the number of occurrences of generator $O_J$ in $p$.
\label{cor:GeneralProductLeadingTerm}
\end{corollary}

\begin{definition}
We say that normally ordered product of generators $m=O_1^{k_1}O_2^{k_2}\dots O_{345}^{k_{345}}$ is reducible by $q$-Groebner basis if its commutative counterpart divides commutative counterpart of the leading term of $g_i$ for some $1\leqslant i\leqslant 61$.
\end{definition}
Hereinafter, let $\mathbf B$ denote the set of all normally ordered products of generators which are not reducible by the $q$-Groebner basis. Note that commutative counterpart of $\mathbf B$ is nothing but $B$, the common monomial basis of $\mathcal A_{q=1,t}$ and $\mathcal A_{q=t=1}$. Our main goal for this section is to prove that $\mathbf B$ provides a monomial basis for $\mathcal A_{q,t}$.

\begin{proposition}
Let $m=O_1^{k_1}O_2^{k_2}\dots O_{345}^{k_{345}}$ be a normally ordered product of generators which is reducible by $q$-Groebner basis. Then there exists the following decomposition of $m$ as element of $\mathcal A_{q,t}$
\begin{equation}
m=\sum_{i=1}^n\lambda_ib_i,\qquad{where}\qquad \lambda_i\in\mathbb C[q^{\pm\frac14},t^{\pm\frac14}],\quad b_i\in\mathbf B\quad\textrm{for all}\quad 1\leq i\leq n.
\label{eq:NormallyOrderedProductBasisDecomposition}
\end{equation}
\label{prop:ReducibleNormallyOrderedProducts}
\end{proposition}
\begin{proof}
There is only finitely many normally ordered products of generators with a given total weight. As a consequence, weighted degree reverse lexicographic monomial ordering allows us to enumerate all such products by natural numbers in the increasing order
\begin{equation*}
\nu_i\prec\nu_{i+1}\qquad\textrm{for all}\quad i\in\mathbb N.
\end{equation*}

We will use induction by the index $i$ in order to prove the statement of the Lemma. There is only one normally ordered monomial of degree 0, so $\nu_1=1$ is the unity element of $\mathcal A_{q,t}$. The latter belongs to $\mathbf B$, hence we get the base of our induction.

For the step of induction we assume that the statement of the lemma holds for all $\nu_i$ with $i< j$ for some $j\in\mathbb N$. Now let $\nu_j=O_1^{\kappa_1}O_2^{\kappa_2}\dots O_{345}^{\kappa_{345}}$. If $\nu_j\in\mathbf B$ then we are done, so without loss of generality we can assume that $\nu_j\not\in\mathbf B$ is reducible by $q$-Groebner basis. In other words, there exists at least one element $g_r$ with the leading term
\begin{equation*}
\mathrm{l.p.}(g_r)=O_1^{l_1}O_2^{l_2}\dots O_{345}^{l_{345}}\qquad\textrm{such that}\qquad
l_1\leq\kappa_1,\quad l_2\leq\kappa_2,\quad\dots,\quad l_{345}<\kappa_{345}.
\end{equation*}
Consider a normally ordered product of generators
\begin{equation*}
\mu:=O_1^{\kappa_1-l_1}O_2^{\kappa_2-l_2}\dots O_{345}^{\kappa_{345}-l_{345}}.
\end{equation*}
By Corollary \ref{cor:GeneralProductLeadingTerm} we have the following identity in $\mathcal A_{q,t}$
\begin{equation*}
0=\mu g_r=q^{\frac M2}\nu_j+\boldsymbol\sigma(\nu_j)\qquad\textrm{for some}\quad M\in\mathbb Z.
\end{equation*}
In other words, as element of $\mathcal A_{q,t}$, the product $\nu_j$ must be equal to a linear combination of subleading monomials $\nu_i,i<j$ with coefficients in $\mathbb C[q^{\pm\frac14},t^{\pm\frac14}]$. Using the inductive assumption we conclude that $\nu_j$ itself must be of the form (\ref{eq:NormallyOrderedProductBasisDecomposition}).
\end{proof}

Combining Proposition \ref{prop:ReducibleNormallyOrderedProducts} with Corollary \ref{cor:GeneralProductLeadingTerm} we get
\begin{corollary}
Let $p=O_{I_1}O_{I_2}\dots O_{I_n}$ be any product of generators, not necessarily normally ordered. There exists the following decomposition of $p$ as element of $\mathcal A_{q,t}$
\begin{equation}
p=\sum_{i=1}^n\lambda_ib_i,\qquad\textrm{where}\qquad\lambda_i\in\mathbb C[q^{\pm\frac14},t^{\pm\frac14}],\quad b_i\in\mathbf B\quad\textrm{for all}\quad 1\leq i\leq n.
\label{eq:GeneralProductBasisDecomposition}
\end{equation}
\label{cor:GeneralProductBasisDecomposition}
\end{corollary}

The latter, in particular, implies that $\mathbf B$ provides a spanning set for $\mathcal A_{q,t}$. We are now ready to prove the main theorem of this section.
\begin{theorem}
\phantom\newline
\begin{enumerate}[(i)]
\item\label{it:BasisB} Set $\mathbf B$ provides a monomial basis for $\mathcal A_{q,t}$.
\smallskip
\item\label{it:StructureConstantsAreLaurentPolynomials} Structure constants of multiplication in basis $\mathbf B$ are Laurent polynomials in parameters $q^{\frac14},t^{\frac14}$.
\smallskip
\item\label{it:IsomorphismWithAlgebraOfqDifferenceOperators} We have a $\mathbf k$-algebra isomorphism
\begin{align*}
\widehat\Delta: \mathcal A_{q,t}\xrightarrow{\sim} A_{q,t}
\end{align*}
between $\mathcal A_{q,t}$ and the algebra of $q$-difference operators (\ref{eq:AqtDifferenceOperatorsAlgebra}) introduced in \cite{ArthamonovShakirov'2019}.
\end{enumerate}
\label{th:BasisAqt}
\end{theorem}
\begin{proof}
Note that by Corollary \ref{cor:GeneralProductBasisDecomposition} we know that $\mathbf B$ is a spanning set, so the only thing we have to prove for part (\ref{it:BasisB}) is that there is no nontrivial linear combinations of elements of $\mathbf B$ which vanish in $\mathcal A_{q,t}$. We will prove parts (\ref{it:BasisB}) and (\ref{it:IsomorphismWithAlgebraOfqDifferenceOperators}) of the theorem simultaneously by showing that there is no linear combination of elements of $\mathbf B$ which corresponds to the trivial $q$-difference operator in $A_{q,t}$. The logic of the proof will be similar to one of Propostion \ref{prop:Aq1tIsomorphismRationalFunctions}.

For the sake of contradiction, suppose there is a linear combination
\begin{align*}
P=P(O_1,\dots,O_{345})=c_1b_1+\dots c_lb_l,\qquad c_i\in\mathbf k,\quad b_i\in\mathbf B,\quad\textrm{for all}\quad 1\leqslant i\leqslant l
\end{align*}
of elements of $\mathbf B$ which corresponds to a zero element of $\mathcal A_{q,t}$. Without loss of generality we can assume that all $c_i,\;1\leqslant i\leqslant l$ are regular at $q=1$ and at least one of the coefficients has nontrivial limit as $q\rightarrow 1$. In other words, the limit
\begin{align*}
P_{q=1}=\lim_{q\rightarrow 1}P(O_1,\dots, O_{345})
\end{align*}
exists and is nontrivial linear combination of monomials.

By Proposition \ref{prop:qDifferenceHomomorphism} we conclude that the following identity holds in $\mathbf k(X_{12},X_{23},X_{13}) \langle\hat\delta_{12}, \hat\delta_{23}, \hat\delta_{13}\rangle$
\begin{align}
P(\hat O_1,\dots, \hat O_{345})=0.
\label{eq:PVanishingQDifference}
\end{align}
Because all $c_i$ are regular at $q=1$ and by Lemma \ref{lemm:qDifferenceLaurentCoefficients} all images of generators are Laurent polynomials in $q^{\frac14}$, the left hand side of (\ref{eq:PVanishingQDifference}) converges as $q\rightarrow1$ and we get the identity
\begin{align*}
P_{q=1}(\Delta(O_1),\dots,\Delta(O_{345}))=0
\end{align*}
in $\mathbf k_t(X_{12},X_{23},X_{13})[P_{12}^{\pm1},P_{23}^{\pm1},P_{13}^{\pm1}]$. By Proposition \ref{prop:Aq1tIsomorphismRationalFunctions} we conclude that the following relation holds in $\mathcal A_{q=1,t}$
\begin{align*}
P_{q=1}(O_1,\dots,O_{345})=0.
\end{align*}
In other words, there exists nontrivial linear combination of elements of $B$ which vanish in $\mathcal A_{q=1,t}$. Here we come to the contradiction with our initial assumption, because $B$ is a monomial basis for $\mathcal A_{q=1,t}$.

Finally, part (\ref{it:StructureConstantsAreLaurentPolynomials}) of the theorem now follows from Corollary \ref{cor:GeneralProductBasisDecomposition}. Indeed, let $b_\alpha,b_\beta\in\mathbf B$ be a pair of basis elements. By part (\ref{it:BasisB}) of the theorem their product $b_1b_2=\sum_{i=1}^n{C_{b_\alpha,b_\beta}^{b_i}}b_i$ has a unique decomposition in basis $\mathbf B$ which must coincide with (\ref{eq:GeneralProductBasisDecomposition}) and thus all $C_{b_{\alpha},b_\beta}^{b_i}\in\mathbb C[q^{\frac14},t^{\frac14}]$ are Laurent polynomials.
\end{proof}

\begin{corollary}
For all triples of basis elements $b_1,b_2,b_3\in\mathbf B$, the structure constant $C_{b_1,b_2}^{b_3}\in\mathbb C[q^{\pm\frac14},t^{\pm\frac14}]$ of associative multiplication in $\mathcal A_{q,t}$ specialize to
\begin{enumerate}[(i)]
\item\label{it:StructureConstantsq1Specialization} The structure constant of commutative multiplication in $\mathcal A_{q=1,t}$ as $q^{\frac14}\rightarrow1$.
\item The structure constant of commutative multiplication in $\mathcal A_{q=t=1}$ as $q^{\frac14},t^{\frac14}\rightarrow1$.
\end{enumerate}
\label{cor:SpecializationOfStructureConstants}
\end{corollary}

\section{Poisson brackets}

As a byproduct of Theorem \ref{th:BasisAqt} we can equip our deformed commutative algebra $\mathcal A_{q=1,t}$ with a Poisson bracket. Indeed, let $b_1,b_2\in\mathbf B$ be a pair of normally ordered products of generators. According to the Theorem \ref{th:BasisAqt}, the products of the corresponding elements of $\mathcal A_{q,t}$ have the following form
\begin{equation}
b_1b_2=\sum_{b_3\in\mathbf B}C_{b_1,b_2}^{b_3}b_3,\qquad b_2b_1=\sum_{b_3\in\mathbf B}C_{b_2,b_1}^{b_3}b_3
\label{eq:BasisElementsMultiplicationInTwoDifferentOrders}
\end{equation}
where structure constants $C_{b_1,b_2}^{b_3},C_{b_2,b_1}^{b_3}\in\mathbb C[q^{\frac14},t^{\frac14}]$ are Laurent polynomials in parameters. Also note, that only finitely many terms on the right hand sides of (\ref{eq:BasisElementsMultiplicationInTwoDifferentOrders}) are nonzero.

\begin{lemma}
For all triples of elements $b_1,b_2,b_3\in\mathbf B$, the following limit exists and is a Laurent polynomial in $t^{\frac14}$
\begin{equation*}
\lim_{q^{\frac14}\rightarrow1}\frac{C_{b_1,b_2}^{b_3}-C_{b_2,b_1}^{b_3}}{q^{\frac14}-1}\quad\in\quad \mathbb C[t^{\pm\frac14}].
\end{equation*}
\end{lemma}
\begin{proof}
By Corollary \ref{cor:SpecializationOfStructureConstants}, part (\ref{it:StructureConstantsq1Specialization}) we know
\begin{equation*}
\lim_{q^{\frac14}\rightarrow 1}C_{b_1,b_2}^{b_3}=\lim_{q^{\frac14}\rightarrow1}C_{b_2,b_1}^{b_3},\qquad\textrm{for all}\quad b_1,b_2,b_3\in\mathbf B.
\end{equation*}
On the other hand, both $C_{b_1,b_2}^{b_3},C_{b_2,b_1}^{b_3}\in\mathbb C[q^{\pm\frac14},t^{\pm\frac14}]$ are Laurent polynomials, so their difference must divide $q^{\frac14}-1$.
\end{proof}

\begin{definition}
Let $\{,\}:\mathbf B\times\mathbf B\rightarrow \mathbb C[t^{\pm\frac14}]\mathbf B$ be a map given by
\begin{equation}
\{b_1,b_2\}:=\frac14\lim_{q^{\frac14}\rightarrow 1}\frac{b_1b_2-b_2b_1}{q^{\frac14}-1}=\sum_{b_3\in\mathbf B}\left(\lim_{q^{\frac14}\rightarrow1} \frac{C_{b_1,b_2}^{b_3}-C_{b_2,b_1}^{b_3}}{q^{\frac14}-1}\right)b_3.
\label{eq:BracketOfBasisElements}
\end{equation}
\label{def:BracketOfBasisElements}
\end{definition}

Recall that the collection $\mathbf B$ of normally ordered products of generators providing basis of $\mathcal A_{q,t}$ is in bijection with collection $B$ of commutative products of generators providing basis on $\mathcal A_{q=1,t}$. This allows us to extend Definition \ref{eq:BracketOfBasisElements} to $\mathbb C(t^{\frac14})$-linear map
\begin{equation*}
\{,\}:\mathcal A_{q=1,t}\otimes\mathcal A_{q=1,t}\rightarrow\mathcal A_{q=1,t},
\end{equation*}
which we denote by the same figure brackets.

\begin{proposition}
Map $\{,\}$ defines a $\mathrm{Mod}(\Sigma_2)$-equivariant Poisson bracket on $\mathcal A_{q=1,t}$.
\label{prop:Aq1tPoissonBracket}
\end{proposition}
\begin{proof}
The skew-symmetry and Leibnitz identity follow immediately from Definition \ref{def:BracketOfBasisElements}. The Jacobi identity is a corollary of the fact that $C_{b_1,b_2}^{b_3}$ are the structure constants of the associative algebra. Finally, the $\mathrm{Mod}(\Sigma_2)$-equivariance follows by Theorem \ref{th:MCGActionGeneric}.
\end{proof}

One can read off the action of the Poisson bracket on generators of $\mathcal A_{q=1,t}$ from Table \ref{tab:QCommRel}. Recall that $(\pm c_{J,K}|X)$ entry in the $O_J$'th row and and $O_K$'th column corresponds to a normal ordering relation in $\mathcal A_{q,t}$
\begin{equation*}
q^{\frac{c_{J,K}}4}O_JO_K-q^{-\frac{c_{J,K}}4}O_KO_J\mp(q^{\frac12}-q^{-\frac12})X=0
\end{equation*}
Note that by linearity, we can use the the left equality in (\ref{eq:BracketOfBasisElements}) for an arbitrary linear combination of basis elements, in particular, for any normally ordered product of generators. Without loss of generality we assume that $O_KO_J$ is a normally ordered monomial, then the Poisson bracket between the two is given by
\begin{equation*}
\{O_J,O_K\}=\frac14\lim_{q^{\frac14}\rightarrow1}\frac{O_JO_K-O_KO_J}{q^{\frac14}-1} =\frac14\lim_{q^{\frac14}\rightarrow1}\frac{(-1+q^{-\frac{c_{J,K}}2})O_KO_J\pm q^{-\frac{c_{J,K}}4}(q^{\frac12}-q^{-\frac12})X}{q^{\frac14}-1}=-\frac{c_{J,K}}2O_KO_J\pm X.
\end{equation*}

Finally, note that coefficients in (\ref{eq:BracketOfBasisElements}) are Laurent polynomials in $t^{\frac14}$, hence specialize well to $t=1$, this allows one to introduce a $\mathbb C$-linear map
\begin{equation*}
\{,\}_{t=1}:\mathcal A_{q=t=1}\otimes\mathcal A_{q=t=1}\rightarrow \mathcal A_{q=t=1}
\end{equation*}
\begin{corollary}
Map $\{,\}_{t=1}$ is a $\mathrm{Mod}(\Sigma_2)$-equivariant Poisson bracket which coincides with the Goldman Poisson bracket on
\begin{equation*}
\mathcal A_{q=t=1}\simeq\mathcal O[\mathrm{Hom}(\pi_1(\Sigma_2),SL(2,\mathbb C))]^{SL(2,\mathbb C)}.
\end{equation*}
\label{cor:Aq1t1BracketAgreesWithGoldman}
\end{corollary}
\begin{proof}
The first part of the statement follows immediately by $t^{\frac14}=1$ specialization of Proposition \ref{prop:Aq1tPoissonBracket} because all the structure constants involved are Laurent polynomials in $t^{\frac14}$.

As for the agreement with the Goldman Poisson bracket, we can use Corollary 5.11 from \cite{CookeSamuelson'2021}, which states that $q=t$ specialization $A_{q=t}\simeq\mathrm{Sk}_{\Sigma_2}$ of the algebra of difference operators is isomorphic to the skein algebra of the genus two surface. Because the latter is the quantization of the Poisson algebra $\mathcal A_{q=t=1}$ we finalize the proof.
\end{proof}

As an illustration of Corollary \ref{cor:Aq1t1BracketAgreesWithGoldman} we have also computed the Poisson brackets between the images of generators of $A_{q=t=1}$ in the invariant subring of $\mathcal O[\mathrm{Hom}(\pi_1(\Sigma_2),SL(2,\mathbb C))]$ using completely independent method of Double Quasi Poisson brackets \cite{MassuyeauTuraev'2014}. We have verified that the result agrees with entries in Table \ref{tab:QCommRel}. The Mathematica code for this illustration can be found at \cite{Arthamonov-GitHub-Flat}.

\section{Discussion and future directions}

\subsection{Symplectic resolutions of character varieties}

Remarkably enough, the $SL(2,\mathbb C)$-character variety of a genus two surface appears to be very special when character varieties are considered in a seemingly unrelated context. It was shown in \cite{BellamySchedler'2019} that $SL(n,\mathbb C)$ character varieties of closed genus $g>0$ surfaces are irreducible symplectic singularities. A symplectic singularity admits symplectic resolution if the symplectic form on the smooth locus can be extended to a symplectic form on the resolution. This turns out to be a rather strong requirement and symplectic singularities are rare. In the same paper \cite{BellamySchedler'2019} G.~Bellamy and T.~Schedler have shown that $SL(n,\mathbb C)$-character varieties admit symplectic resolutions only when $g=1$ or $(g,n)=(2,2)$.

The case $g=1$ is well-known and symplectic resolution can be described explicitly using (a closed subscheme of) the Hilbert Scheme $\mathrm{Hilb}^n(\mathbb C^\times\times\mathbb C^\times)$ \cite{Nakajima'1999}. One-parameter deformation of $\mathrm{Hilb}^n(\mathbb C^\times\times\mathbb C^\times)$ is nothing but completed configuration space of the trigonometric Ruijsenaars-Shneider integrable system \cite{Oblomkov'2004-CM-Spaces}, the same integrable system which in the quantum case gives rise to sDAHA.

It is natural to expect that our deformed commutative algebra $\mathcal A_{q=1,t}$ appears in the $(g,n)=(2,2)$ case of \cite{BellamySchedler'2019} and thus provides a multiplicative analogue of \cite{LehnSorger'2006}. We are planning to investigate this question in a sequel publication.

\subsection{Fourier Duality and root systems}

In \cite{DiFrancescoKedem'2023} P.~Di~Franceso and R.~Kedem have studied Fourier duality of Macdonald $q$-Difference operators in a family of examples which includes the genus two algebra $A_{q,t}$. In particular, authors proposed the analogue of an affine root system which can be used to define the difference operators and duality.

It would be extremely interesting to investigate the potential role of this ``root-like`` system in the ``PBW-like`` approach to the solution of the word problem in $\mathcal A_{q,t}\simeq A_{q,t}$ that we present in the current manuscript. Of course, by analogy with the usual Double Affine Hecke Algebra one would expect both roots and co-roots to play equal role. Without going into much details we would like to conclude this comment by highlighting the striking, although rather indirect, similarity between formulas for extra generators (\ref{eq:qdiffleveltwogenerators}), (\ref{eq:qdifflevelthreegenerators}) and how non-simple roots are expressed through the simple roots in the usual Lie algebra.

\subsection{Relation to DAHA polynomials for double torus knots}

In \cite{Hikami'2019} K.~Hikami has constructed a DAHA representation of the skein algebra of the genus two surface and has shown in \cite{Hikami'2024} that this representation is equivariant under the action of the Mapping Class Group of a genus two surface. This approach is slightly different from the deformation of the skein algebra itself that we examine in this manuscript. In particular, the construction from our join papers with Sh.~Shakirov \cite{ArthamonovShakirov'2019, ArthamonovShakirov'2020} does not produce any embedding of the undeformed algebra $A_{q=t}\simeq Sk_q(\Sigma_2)$ into the $A_{q,t}$. Establishing the precise relation between $A_{q,t}$ from \cite{ArthamonovShakirov'2019} and $SH^{gen}_{q,{\mathbf t}^*}$ from \cite{Hikami'2019,Hikami'2024} is a very interesting problem.

It is worth noting that both algebras can be used to compute amplitudes for each embedding of a knot on a genus two surface. These amplitudes refine Jones amplitudes $|J_n(q)|:=J_n(q)J_n(q^{-1})$ and it would be very interesting to compare the them. It is still an open question, however, whether the knot amplitude defined in Conjecture II of \cite{ArthamonovShakirov'2020} is independent of the choice of embedding of a given knot.

\subsection{Relation to cluster algebras}

Last, but not the least, in \cite{ChekhovShapiro'2023} L.~Chekhov and M.~Shapiro have proposed a cluster algebra construction for an $SL(2,\mathbb C)$ character variety of a closed genus two surface. The authors have shown that one can introduce an extra parameter in such cluster algebra. Based on that, the cluster interpretation for the algebra of $q$-difference operators $A_{q,t}$ was constructed in \cite{ArthamonovChekhovDiFrancescoKedemShapiroShapiro'2024}. In particular, the common eigenfunctions $\Psi_{j_1,j_2,j_3}$ for $q$-difference operators $\widehat{O}_{A_1},\widehat{O}_{A_2},\widehat{O}_{A_3}\in A_{q,t}$ defined in \cite{ArthamonovShakirov'2019} as genus two Macdonald polynomials potentially can be extended beyond integral weights using the cluster algebra approach. The algebraic interpretation of such potential extension can be a very interesting question for further research.

\section*{Acknowledgements}

I am very grateful to Shamil Shakirov, this paper would have not been possible without several years of prior joint work with Shamil on \cite{ArthamonovShakirov'2020} and \cite{ArthamonovShakirov'2019}. I am also grateful to Pavel Etingof, Andrey Okounkov, and Nicolai Reshetikhin for many fruitful discussions and remarks.

\appendix

\section{Calculations for Proposition \ref{prop:d1Automorphism}}
\label{sec:d1Automorphism}

In this section we prove the second part of Proposition \ref{prop:d1Automorphism}, namely that homomorphism $d_1:F\rightarrow F$ introduced in (\ref{eq:d1Automorphism}) preserves defining ideal (\ref{eq:DefiningRelationsAqt}) of $\mathcal A_{q,t}$. Throughout the section we will utilize the same notations as in the proof of Lemma \ref{lemm:JRelations}. It will be convenient for us to express the right hand side of the action using both the original relators $\eta_{I,J}$ introduced in (\ref{eq:EtaIJRelator}) as well as their special combinations $\rho_i,1\leq i\leq 18$ as in (\ref{eq:JrelationsViaNormal}). Also, by $\rho_0\in F$ we denote the q-Casimir relator which appears on the left hand side of (\ref{eq:qCasimirRelation}).

First we note that $d_1$ acts trivially on the following relators:
\begin{align*}
&\eta_{(3)(1)}, \eta_{(4)(1)}, \eta_{(4)(3)}, \eta_{(5)(1)}, \eta_{(5)(3)}, \eta_{(5)(4)},\\ &\eta_{(34)(1)}, \eta_{(34)(3)}, \eta_{(34)(4)}, \eta_{(34)(5)},\\
&\eta_{(45)(1)}, \eta_{(45)(3)}, \eta_{(45)(4)}, \eta_{(45)(5)}, \eta_{(45)(34)},\\
&\eta_{(345)(1)}, \eta_{(345)(3)}, \eta_{(345)(4)}, \eta_{(345)(5)}, \eta_{(345)(34)}, \eta_{(345)(45)}.
\end{align*}
As for the remaining relators we get
\begin{align*}
d_1(\eta_{(2)(1)})=&O_1O_2O_1 -q^{\frac{1}{2}}O_1^2O_2 -q^{\frac{1}{4}}O_{12}O_1 +q^{\frac{3}{4}}O_1O_{12} +q^{-\frac{1}{2}}(q-1)O_2=q^{\frac{1}{4}}O_1\eta _{(2)(1)} -\eta _{(12)(1)},\\[0.3em]
d_1(\eta_{(3)(2)})=&O_3O_1O_2 -q^{\frac{1}{2}}O_1O_2O_3 +q^{\frac{3}{4}}O_{12}O_3 -q^{\frac{1}{4}}O_3O_{12} +q^{-\frac{1}{4}}(q-1)O_1O_{23} -(q-1)O_{123}\\
=&q^{\frac{1}{4}}O_1\eta _{(3)(2)} +\eta _{(3)(1)}O_2 +q^{\frac{1}{2}}\eta _{(12)(3)},\\[0.3em]
d_1(\eta_{(4)(2)})=&q^{\frac{1}{4}}O_4O_1O_2 -q^{\frac{1}{4}}O_1O_2O_4 +q^{\frac{1}{2}}O_{12}O_4 -q^{\frac{1}{2}}O_4O_{12}=q^{\frac{1}{4}}O_1\eta _{(4)(2)} +q^{\frac{1}{4}}\eta _{(4)(1)}O_2 +q^{\frac{1}{2}}\eta _{(12)(4)},\\[0.3em]
d_1(\eta_{(5)(2)})=&q^{\frac{1}{4}}O_5O_1O_2 -q^{\frac{1}{4}}O_1O_2O_5 +q^{\frac{1}{2}}O_{12}O_5 -q^{\frac{1}{2}}O_5O_{12}=q^{\frac{1}{4}}O_1\eta _{(5)(2)} +q^{\frac{1}{4}}\eta _{(5)(1)}O_2 +q^{\frac{1}{2}}\eta _{(12)(5)},\\[0.3em]
d_1(\eta_{(6)(1)})=&q^{\frac{1}{4}}O_{61}O_1 -q^{\frac{3}{4}}O_1O_{61} +(q-1)O_6=q^{\frac{1}{2}}\eta _{(61)(1)},\\[0.3em]
d_1(\eta_{(6)(2)})=&q^{\frac{1}{4}}O_{61}O_1O_2 -q^{\frac{1}{4}}O_1O_2O_{61} -q^{\frac{1}{2}}O_{61}O_{12} +q^{\frac{1}{2}}O_{12}O_{61}\\
=&q^{\frac{1}{2}}O_1\eta _{(61)(2)} +q^{\frac{1}{2}}\eta _{(61)(1)}O_2 -\eta _{(61)(12)} -q^{-1}(q-1) (q+1)\eta _{(6)(2)} +q^{-\frac{1}{2}}(q-1)\rho _4,\\[0.3em]
d_1(\eta_{(6)(3)})=&O_{61}O_3 -O_3O_{61}=\eta _{(61)(3)},\\[0.3em]
d_1(\eta_{(6)(4)})=&O_{61}O_4 -O_4O_{61}=\eta _{(61)(4)},\\[0.3em]
d_1(\eta_{(6)(5)})=&q^{-\frac{1}{4}}O_{61}O_5 -q^{\frac{1}{4}}O_5O_{61} +q^{-\frac{1}{2}}(q-1)O_{234}=\eta _{(61)(5)},\\[0.3em]
d_1(\eta_{(12)(1)})=&q^{\frac{1}{4}}O_2O_1 -q^{\frac{3}{4}}O_1O_2 +(q-1)O_{12}=q^{\frac{1}{2}}\eta _{(2)(1)},\\[0.3em]
d_1(\eta_{(12)(2)})=&O_2O_1O_2 -q^{\frac{1}{2}}O_1O_2^2 +q^{\frac{3}{4}}O_{12}O_2 -q^{\frac{1}{4}}O_2O_{12} +q^{-\frac{1}{2}}(q-1)O_1=q^{\frac{1}{4}}\eta _{(2)(1)}O_2 +\eta _{(12)(2)},\\[0.3em]
d_1(\eta_{(12)(3)})=& -q^{-\frac{1}{4}}O_3O_2 +q^{\frac{1}{4}}O_2O_3 -q^{-\frac{1}{2}}(q-1)O_{23}= -\eta _{(3)(2)},\\[0.3em]
d_1(\eta_{(12)(4)})=& -O_4O_2 +O_2O_4= -\eta _{(4)(2)},\\[0.3em]
d_1(\eta_{(12)(5)})=& -O_5O_2 +O_2O_5= -\eta _{(5)(2)},\\[0.3em]
d_1(\eta_{(12)(6)})=& -q^{\frac{1}{4}}O_{61}O_2 +q^{-\frac{1}{4}}O_2O_{61} +q^{-\frac{1}{2}}(q-1)O_{345}= -\eta _{(61)(2)},\\[0.3em]
d_1(\eta_{(23)(1)})=&O_1O_{23}O_1 -q^{\frac{1}{2}}O_1^2O_{23} -q^{\frac{1}{4}}O_{123}O_1 +q^{\frac{3}{4}}O_1O_{123} +q^{-\frac{1}{2}}(q-1)O_{23}=q^{\frac{1}{4}}O_1\eta _{(23)(1)} -\eta _{(123)(1)},\\[0.3em]
d_1(\eta_{(23)(2)})=&q^{\frac{3}{4}}O_1O_{23}O_1O_2 -q^{\frac{1}{4}}O_1O_2O_1O_{23} +q^{\frac{1}{2}}O_{12}O_1O_{23} -qO_1O_{23}O_{12} -qO_{123}O_1O_2 +q^{\frac{1}{2}}O_1O_2O_{123}\\
&+q^{\frac{5}{4}}O_{123}O_{12} -q^{\frac{3}{4}}O_{12}O_{123} -q^{-\frac{1}{2}}(q-1)O_3\\
=& -q^{\frac{1}{2}}O_1\eta _{(2)(1)}O_{23} +qO_1^2\eta _{(23)(2)} +qO_1\eta _{(23)(1)}O_2 +q^{\frac{1}{4}}\eta _{(12)(1)}O_{23} -q^{\frac{3}{2}}O_1\eta _{(123)(2)} -q^{\frac{1}{2}}O_1\eta _{(23)(12)}\\
&-(q-1)O_1\rho _5 -q^{\frac{3}{4}}\eta _{(123)(1)}O_2 +q\eta _{(123)(12)} -(q-1)\eta _{(23)(2)},\\[0.3em]
d_1(\eta_{(23)(3)})=& -q^{\frac{1}{2}}O_3O_1O_{23} +O_1O_{23}O_3 -q^{\frac{1}{4}}O_{123}O_3 +q^{\frac{3}{4}}O_3O_{123} +q^{-\frac{1}{4}}(q-1)O_1O_2 -(q-1)O_{12}\\
=& -q^{\frac{1}{2}}\eta _{(3)(1)}O_{23} +q^{\frac{1}{4}}O_1\eta _{(23)(3)} -q^{\frac{1}{2}}\eta _{(123)(3)},\\[0.3em]
d_1(\eta_{(23)(4)})=& -O_4O_1O_{23} +q^{\frac{1}{2}}O_1O_{23}O_4 -q^{\frac{3}{4}}O_{123}O_4 +q^{\frac{1}{4}}O_4O_{123} -q^{-\frac{1}{4}}(q-1)O_1O_{234} +(q-1)O_{56}\\
=& -\eta _{(4)(1)}O_{23} +q^{\frac{1}{4}}O_1\eta _{(23)(4)} -q^{\frac{1}{2}}\eta _{(123)(4)},\\[0.3em]
d_1(\eta_{(23)(5)})=& -q^{\frac{1}{4}}O_5O_1O_{23} +q^{\frac{1}{4}}O_1O_{23}O_5 -q^{\frac{1}{2}}O_{123}O_5 +q^{\frac{1}{2}}O_5O_{123}= -q^{\frac{1}{4}}\eta _{(5)(1)}O_{23} +q^{\frac{1}{4}}O_1\eta _{(23)(5)} -q^{\frac{1}{2}}\eta _{(123)(5)},\\[0.3em]
d_1(\eta_{(23)(6)})=& -q^{\frac{1}{4}}O_{61}O_1O_{23} +q^{\frac{1}{4}}O_1O_{23}O_{61} -q^{\frac{1}{2}}O_{123}O_{61} +q^{\frac{1}{2}}O_{61}O_{123}\\
=& -q^{\frac{1}{2}}\eta _{(61)(1)}O_{23} -q^{\frac{1}{2}}O_1\eta _{(61)(23)} -q^{\frac{1}{2}}\eta _{(123)(61)} +(q-1)\eta _{(45)(1)} -(q-1)\eta _{(23)(6)},\\[0.3em]
d_1(\eta_{(23)(12)})=& -q^{-\frac{1}{4}}O_2O_1O_{23} +q^{\frac{3}{4}}O_1O_{23}O_2 -qO_{123}O_2 +O_2O_{123} -q^{-1}(q-1) (q+1)O_1O_3\\
&+q^{-1}t^{-\frac{1}{2}}(q-1) (q+t)O_5\\
=& -\eta _{(2)(1)}O_{23} +q^{\frac{1}{2}}O_1\eta _{(23)(2)} -q\eta _{(123)(2)} -q^{-\frac{1}{2}}(q-1)\rho _5,\\[0.3em]
d_1(\eta_{(34)(2)})=&O_{34}O_1O_2 -q^{\frac{1}{2}}O_1O_2O_{34} -q^{\frac{1}{4}}O_{34}O_{12} +q^{\frac{3}{4}}O_{12}O_{34} +q^{-\frac{1}{4}}(q-1)O_1O_{234} -(q-1)O_{56}\\
=&q^{\frac{1}{4}}O_1\eta _{(34)(2)} +\eta _{(34)(1)}O_2 -q^{\frac{1}{2}}\eta _{(34)(12)},\\[0.3em]
d_1(\eta_{(34)(6)})=& -O_{61}O_{34} +O_{34}O_{61}= -\eta _{(61)(34)},\\[0.3em]
d_1(\eta_{(34)(12)})=&q^{-\frac{1}{4}}O_{34}O_2 -q^{\frac{1}{4}}O_2O_{34} +q^{-\frac{1}{2}}(q-1)O_{234}=\eta _{(34)(2)},\\[0.3em]
d_1(\eta_{(34)(23)})=&q^{\frac{3}{4}}O_{34}O_1O_{23} -q^{-\frac{1}{4}}O_1O_{23}O_{34} +O_{123}O_{34} -qO_{34}O_{123} -q^{-\frac{3}{4}}(q-1) (q+1)O_1O_2O_4\\ &+q^{-\frac{1}{2}}(q-1) (q+1)O_{12}O_4 +q^{-1}t^{-\frac{1}{2}}(q-1) (q+t)O_{61}\\
=&q^{\frac{3}{4}}\eta _{(34)(1)}O_{23} +q^{\frac{1}{4}}O_1\eta _{(34)(23)} +\eta _{(123)(34)} +q^{-\frac{1}{2}}(q-1)\eta _{(56)(3)} +q^{\frac{1}{2}}(q-1)\eta _{(12)(4)}\\ &+q^{-\frac{1}{2}}t^{-\frac{1}{2}}(q-1) (q+t)\eta _{(6)(1)} -q^{\frac{1}{2}}(q-1)\rho _{12},\\[0.3em]
d_1(\eta_{(45)(2)})=&q^{\frac{1}{4}}O_{45}O_1O_2 -q^{\frac{1}{4}}O_1O_2O_{45} -q^{\frac{1}{2}}O_{45}O_{12} +q^{\frac{1}{2}}O_{12}O_{45}=q^{\frac{1}{4}}O_1\eta _{(45)(2)} +q^{\frac{1}{4}}\eta _{(45)(1)}O_2 -q^{\frac{1}{2}}\eta _{(45)(12)},\\[0.3em]
d_1(\eta_{(45)(6)})=& -q^{-\frac{1}{4}}O_{61}O_{45} +q^{\frac{1}{4}}O_{45}O_{61} -q^{-\frac{1}{2}}(q-1)O_{23}= -\eta _{(61)(45)},\\[0.3em]
d_1(\eta_{(45)(12)})=&O_{45}O_2 -O_2O_{45}=\eta _{(45)(2)},\\[0.3em]
d_1(\eta_{(45)(23)})=&O_{45}O_1O_{23} -q^{\frac{1}{2}}O_1O_{23}O_{45} +q^{\frac{3}{4}}O_{123}O_{45} -q^{\frac{1}{4}}O_{45}O_{123} +q^{-\frac{1}{4}}(q-1)O_1O_{61} -(q-1)O_6\\
=&\eta _{(45)(1)}O_{23} +q^{\frac{1}{4}}O_1\eta _{(45)(23)} +q^{\frac{1}{2}}\eta _{(123)(45)},\\[0.3em]
d_1(\eta_{(56)(1)})=&q^{\frac{1}{4}}O_{234}O_1 -q^{\frac{3}{4}}O_1O_{234} +(q-1)O_{56}=q^{\frac{1}{2}}\eta _{(234)(1)},\\[0.3em]
d_1(\eta_{(56)(2)})=&q^{\frac{1}{4}}O_{234}O_1O_2 -q^{\frac{1}{4}}O_1O_2O_{234} -q^{\frac{1}{2}}O_{234}O_{12} +q^{\frac{1}{2}}O_{12}O_{234}\\
=&q^{\frac{1}{2}}O_1\eta _{(234)(2)} +q^{\frac{1}{2}}\eta _{(234)(1)}O_2 -q^{\frac{1}{2}}\eta _{(234)(12)} -(q-1)\eta _{(34)(1)},\\[0.3em]
d_1(\eta_{(56)(3)})=&O_{234}O_3 -O_3O_{234}=\eta _{(234)(3)},\\[0.3em]
d_1(\eta_{(56)(4)})=&q^{-\frac{1}{4}}O_{234}O_4 -q^{\frac{1}{4}}O_4O_{234} +q^{-\frac{1}{2}}(q-1)O_{23}=\eta _{(234)(4)},\\[0.3em]
d_1(\eta_{(56)(5)})=&q^{\frac{1}{4}}O_{234}O_5 -q^{-\frac{1}{4}}O_5O_{234} -q^{-\frac{1}{2}}(q-1)O_{61}=\eta _{(234)(5)},\\[0.3em]
d_1(\eta_{(56)(6)})=&q^{-\frac{1}{4}}O_{234}O_{61} -q^{\frac{1}{4}}O_{61}O_{234} +q^{-\frac{1}{2}}(q-1)O_5=\eta _{(234)(61)},\\[0.3em]
d_1(\eta_{(56)(12)})=&q^{\frac{1}{4}}O_{234}O_2 -q^{-\frac{1}{4}}O_2O_{234} -q^{-\frac{1}{2}}(q-1)O_{34}=\eta _{(234)(2)},\\[0.3em]
d_1(\eta_{(56)(23)})=&q^{\frac{1}{4}}O_{234}O_1O_{23} -q^{\frac{1}{4}}O_1O_{23}O_{234} -q^{\frac{1}{2}}O_{234}O_{123} +q^{\frac{1}{2}}O_{123}O_{234}\\
=&q^{\frac{1}{2}}\eta _{(234)(1)}O_{23} +q^{\frac{1}{2}}O_1\eta _{(234)(23)} +q^{-\frac{5}{4}}t^{-\frac{1}{2}}(q-1) (q+t)O_2\eta _{(6)(1)} +q^{-\frac{5}{4}}t^{-\frac{1}{2}}(q-1) (q+t)\eta _{(2)(1)}O_6\\
&-\eta _{(234)(123)} -(q-1)\eta _{(56)(23)} -q^{-\frac{3}{2}}t^{-\frac{1}{2}}(q-1)^2 (q+t)\eta _{(12)(6)} -q^{-1}(q-1) (q+1)\eta _{(4)(1)}\\
&+q^{-\frac{1}{2}}(q-1)\rho _{13},\\[0.3em]
d_1(\eta_{(56)(34)})=&q^{-\frac{1}{4}}O_{234}O_{34} -q^{\frac{1}{4}}O_{34}O_{234} +q^{-\frac{1}{2}}(q-1)O_2=\eta _{(234)(34)},\\[0.3em]
d_1(\eta_{(56)(45)})=&q^{\frac{1}{2}}O_{234}O_{45} -q^{-\frac{1}{2}}O_{45}O_{234} -q^{-1}(q-1) (q+1)O_4O_{61} +q^{-\frac{3}{4}}t^{-\frac{1}{2}}(q-1) (q+t)O_1O_2\\ &-q^{-\frac{1}{2}}t^{-\frac{1}{2}}(q-1) (q+t)O_{12}\\
=&q^{\frac{1}{2}}\eta _{(234)(45)} +(q-1)\eta _{(61)(4)} -(q-1)\eta _{(23)(5)} +(q-1)\rho _7,\\[0.3em]
d_1(\eta_{(61)(1)})=&O_1O_{61}O_1 -q^{\frac{1}{2}}O_1^2O_{61} -q^{\frac{1}{4}}O_6O_1 +q^{\frac{3}{4}}O_1O_6 +q^{-\frac{1}{2}}(q-1)O_{61}=q^{\frac{1}{4}}O_1\eta _{(61)(1)} -\eta _{(6)(1)},\\[0.3em]
d_1(\eta_{(61)(2)})=&q^{\frac{3}{4}}O_1O_{61}O_1O_2 -q^{\frac{1}{4}}O_1O_2O_1O_{61} +q^{\frac{1}{2}}O_{12}O_1O_{61} -qO_1O_{61}O_{12} -qO_6O_1O_2 +q^{\frac{1}{2}}O_1O_2O_6 -q^{\frac{3}{4}}O_{12}O_6\\
&+q^{\frac{5}{4}}O_6O_{12} -q^{-\frac{1}{2}}(q-1)O_{345}\\
=& -q^{\frac{1}{2}}O_1\eta _{(2)(1)}O_{61} +qO_1^2\eta _{(61)(2)} +qO_1\eta _{(61)(1)}O_2 +q^{\frac{1}{4}}\eta _{(12)(1)}O_{61} -q^{\frac{1}{2}}O_1\eta _{(61)(12)}\\ &-q^{-\frac{1}{2}}(q^2+q-1)O_1\eta _{(6)(2)} +(q-1)O_1\rho _4 -q^{\frac{3}{4}}\eta _{(6)(1)}O_2 -(q-1)\eta _{(61)(2)} -q\eta _{(12)(6)},\\[0.3em]
d_1(\eta_{(61)(3)})=& -q^{\frac{1}{4}}O_3O_1O_{61} +q^{\frac{1}{4}}O_1O_{61}O_3 -q^{\frac{1}{2}}O_6O_3 +q^{\frac{1}{2}}O_3O_6= -q^{\frac{1}{4}}\eta _{(3)(1)}O_{61} +q^{\frac{1}{4}}O_1\eta _{(61)(3)} -q^{\frac{1}{2}}\eta _{(6)(3)},\\[0.3em]
d_1(\eta_{(61)(4)})=& -q^{\frac{1}{4}}O_4O_1O_{61} +q^{\frac{1}{4}}O_1O_{61}O_4 -q^{\frac{1}{2}}O_6O_4 +q^{\frac{1}{2}}O_4O_6= -q^{\frac{1}{4}}\eta _{(4)(1)}O_{61} +q^{\frac{1}{4}}O_1\eta _{(61)(4)} -q^{\frac{1}{2}}\eta _{(6)(4)},\\[0.3em]
d_1(\eta_{(61)(5)})=& -q^{\frac{1}{2}}O_5O_1O_{61} +O_1O_{61}O_5 +q^{-\frac{1}{4}}(q-1)O_1O_{234} -q^{\frac{1}{4}}O_6O_5 +q^{\frac{3}{4}}O_5O_6 -(q-1)O_{56}\\
=& -q^{\frac{1}{2}}\eta _{(5)(1)}O_{61} +q^{\frac{1}{4}}O_1\eta _{(61)(5)} -q^{\frac{1}{2}}\eta _{(6)(5)},\\[0.3em]
d_1(\eta_{(61)(6)})=& -O_{61}O_1O_{61} +q^{\frac{1}{2}}O_1O_{61}^2 +q^{\frac{1}{4}}O_{61}O_6 -q^{\frac{3}{4}}O_6O_{61} -q^{-\frac{1}{2}}(q-1)O_1= -q^{\frac{1}{4}}\eta _{(61)(1)}O_{61} +\eta _{(61)(6)},\\[0.3em]
d_1(\eta_{(61)(12)})=& -q^{\frac{3}{4}}O_2O_1O_{61} +q^{-\frac{1}{4}}O_1O_{61}O_2 +q^{-\frac{3}{4}}(q-1) (q+1)O_1O_2O_{61} -q^{-\frac{1}{2}}(q-1) (q+1)O_{12}O_{61} -O_6O_2\\
&+qO_2O_6 -q^{-1}t^{-\frac{1}{2}}(q-1) (q+t)O_4\\
=& -q\eta _{(2)(1)}O_{61} +q^{-\frac{1}{2}}O_1\eta _{(61)(2)} +q^{-1}(q-1)\eta _{(61)(12)} -q^{-2}(q^2+q-1)\eta _{(6)(2)} +q^{-\frac{3}{2}}(q-1)\rho _4,\\[0.3em]
d_1(\eta_{(61)(23)})=&q^{\frac{3}{4}}O_1O_{61}O_1O_{23} -q^{\frac{1}{4}}O_1O_{23}O_1O_{61} +q^{\frac{1}{2}}O_{123}O_1O_{61} -qO_1O_{61}O_{123} -qO_6O_1O_{23} +q^{\frac{1}{2}}O_1O_{23}O_6\\
&-q^{\frac{3}{4}}O_{123}O_6 +q^{\frac{5}{4}}O_6O_{123} -q^{-\frac{1}{2}}(q-1)O_{45}\\
=&qO_1\eta _{(61)(1)}O_{23} -q^{\frac{1}{2}}O_1\eta _{(23)(1)}O_{61} +qO_1^2\eta _{(61)(23)} -q^{\frac{3}{4}}\eta _{(6)(1)}O_{23} +q^{\frac{1}{4}}\eta _{(123)(1)}O_{61} +qO_1\eta _{(123)(61)}\\
&-q^{\frac{1}{2}}(q-1)O_1\eta _{(45)(1)} +q^{\frac{3}{2}}O_1\eta _{(23)(6)} -q\eta _{(123)(6)} -(q-1)\eta _{(61)(23)},\\[0.3em]
d_1(\eta_{(61)(34)})=& -q^{\frac{1}{4}}O_{34}O_1O_{61} +q^{\frac{1}{4}}O_1O_{61}O_{34} +q^{\frac{1}{2}}O_{34}O_6 -q^{\frac{1}{2}}O_6O_{34}= -q^{\frac{1}{4}}\eta _{(34)(1)}O_{61} +q^{\frac{1}{4}}O_1\eta _{(61)(34)} +q^{\frac{1}{2}}\eta _{(34)(6)},\\[0.3em]
d_1(\eta_{(61)(45)})=& -q^{\frac{1}{2}}O_{45}O_1O_{61} +O_1O_{61}O_{45} +q^{-\frac{1}{4}}(q-1)O_1O_{23} +q^{\frac{3}{4}}O_{45}O_6 -q^{\frac{1}{4}}O_6O_{45} -(q-1)O_{123}\\
=& -q^{\frac{1}{2}}\eta _{(45)(1)}O_{61} +q^{\frac{1}{4}}O_1\eta _{(61)(45)} +q^{\frac{1}{2}}\eta _{(45)(6)},\\[0.3em]
d_1(\eta_{(61)(56)})=& -q^{-\frac{1}{4}}O_{234}O_1O_{61} +q^{\frac{3}{4}}O_1O_{61}O_{234} +O_{234}O_6 -qO_6O_{234} -q^{-1}(q-1) (q+1)O_1O_5\\
&+q^{-1}t^{-\frac{1}{2}}(q-1) (q+t)O_3\\
=& -\eta _{(234)(1)}O_{61} -q^{\frac{1}{2}}O_1\eta _{(234)(61)} +\eta _{(234)(6)} +q^{-1}(q-1)\eta _{(5)(1)} -q^{-\frac{1}{2}}(q-1)\rho _3,\\[0.3em]
d_1(\eta_{(123)(1)})=&q^{\frac{1}{4}}O_{23}O_1 -q^{\frac{3}{4}}O_1O_{23} +(q-1)O_{123}=q^{\frac{1}{2}}\eta _{(23)(1)},\\[0.3em]
d_1(\eta_{(123)(2)})=&q^{\frac{1}{4}}O_{23}O_1O_2 -q^{\frac{1}{4}}O_1O_2O_{23} -q^{\frac{1}{2}}O_{23}O_{12} +q^{\frac{1}{2}}O_{12}O_{23}\\
=&q^{\frac{1}{2}}O_1\eta _{(23)(2)} +q^{\frac{1}{2}}\eta _{(23)(1)}O_2 -(q-1)\eta _{(123)(2)} -\eta _{(23)(12)} -q^{-\frac{1}{2}}(q-1)\rho _5,\\[0.3em]
d_1(\eta_{(123)(3)})=&q^{-\frac{1}{4}}O_{23}O_3 -q^{\frac{1}{4}}O_3O_{23} +q^{-\frac{1}{2}}(q-1)O_2=\eta _{(23)(3)},\\[0.3em]
d_1(\eta_{(123)(4)})=&q^{\frac{1}{4}}O_{23}O_4 -q^{-\frac{1}{4}}O_4O_{23} -q^{-\frac{1}{2}}(q-1)O_{234}=\eta _{(23)(4)},\\[0.3em]
d_1(\eta_{(123)(5)})=&O_{23}O_5 -O_5O_{23}=\eta _{(23)(5)},\\[0.3em]
d_1(\eta_{(123)(6)})=& -q^{\frac{1}{4}}O_{61}O_{23} +q^{-\frac{1}{4}}O_{23}O_{61} +q^{-\frac{1}{2}}(q-1)O_{45}= -\eta _{(61)(23)},\\[0.3em]
d_1(\eta_{(123)(12)})=&q^{\frac{1}{4}}O_{23}O_2 -q^{-\frac{1}{4}}O_2O_{23} -q^{-\frac{1}{2}}(q-1)O_3=\eta _{(23)(2)},\\[0.3em]
d_1(\eta_{(123)(23)})=&O_{23}O_1O_{23} -q^{\frac{1}{2}}O_1O_{23}^2 +q^{\frac{3}{4}}O_{123}O_{23} -q^{\frac{1}{4}}O_{23}O_{123} +q^{-\frac{1}{2}}(q-1)O_1=q^{\frac{1}{4}}\eta _{(23)(1)}O_{23} +\eta _{(123)(23)},\\[0.3em]
d_1(\eta_{(123)(34)})=& -O_{34}O_{23} +O_{23}O_{34} -q^{-\frac{1}{2}}(q-1)O_{234}O_3 +q^{-\frac{1}{2}}(q-1)O_2O_4\\
=& -q^{-\frac{1}{2}}(q-1)\eta _{(234)(3)} -q^{-\frac{1}{2}}\eta _{(34)(23)} -q^{-1}(q-1)\rho _6,\\[0.3em]
d_1(\eta_{(123)(45)})=& -q^{-\frac{1}{4}}O_{45}O_{23} +q^{\frac{1}{4}}O_{23}O_{45} -q^{-\frac{1}{2}}(q-1)O_{61}= -\eta _{(45)(23)},\\[0.3em]
d_1(\eta_{(123)(56)})=& -q^{\frac{1}{4}}O_{234}O_{23} +q^{-\frac{1}{4}}O_{23}O_{234} +q^{-\frac{1}{2}}(q-1)O_4= -\eta _{(234)(23)},\\[0.3em]
d_1(\eta_{(123)(61)})=&q^{\frac{1}{4}}O_{23}O_1O_{61} -q^{\frac{1}{4}}O_1O_{61}O_{23} -q^{-\frac{1}{4}}(q-1)O_1O_{23}O_{61} +(q-1)O_{123}O_{61} -q^{\frac{1}{2}}O_{23}O_6 +q^{\frac{1}{2}}O_6O_{23}\\
&+q^{-\frac{1}{2}}(q-1)O_{45}O_1\\
=&q^{\frac{1}{2}}\eta _{(23)(1)}O_{61} -O_1\eta _{(61)(23)} +q^{-\frac{1}{2}}(q-1)\eta _{(45)(1)} -q^{\frac{1}{2}}\eta _{(23)(6)},\\[0.3em]
d_1(\eta_{(234)(1)})=&O_1O_{234}O_1 -q^{\frac{1}{2}}O_1^2O_{234} -q^{\frac{1}{4}}O_{56}O_1 +q^{\frac{3}{4}}O_1O_{56} +q^{-\frac{1}{2}}(q-1)O_{234}=q^{\frac{1}{4}}O_1\eta _{(234)(1)} -\eta _{(56)(1)},\\[0.3em]
d_1(\eta_{(234)(2)})=&q^{\frac{3}{4}}O_1O_{234}O_1O_2 -q^{\frac{1}{4}}O_1O_2O_1O_{234} +q^{\frac{1}{2}}O_{12}O_1O_{234} -qO_1O_{234}O_{12} -qO_{56}O_1O_2 +q^{\frac{1}{2}}O_1O_2O_{56}\\
&+q^{\frac{5}{4}}O_{56}O_{12} -q^{\frac{3}{4}}O_{12}O_{56} -q^{-\frac{1}{2}}(q-1)O_{34}\\
=& -q^{\frac{1}{2}}O_1\eta _{(2)(1)}O_{234} +qO_1^2\eta _{(234)(2)} +qO_1\eta _{(234)(1)}O_2 +q^{\frac{1}{4}}\eta _{(12)(1)}O_{234} -qO_1\eta _{(234)(12)} -q^{\frac{1}{2}}O_1\eta _{(56)(2)}\\
&-q^{\frac{1}{2}}(q-1)O_1\eta _{(34)(1)} -q^{\frac{3}{4}}\eta _{(56)(1)}O_2 -(q-1)\eta _{(234)(2)} +q\eta _{(56)(12)},\\[0.3em]
d_1(\eta_{(234)(3)})=& -q^{\frac{1}{4}}O_3O_1O_{234} +q^{\frac{1}{4}}O_1O_{234}O_3 -q^{\frac{1}{2}}O_{56}O_3 +q^{\frac{1}{2}}O_3O_{56}= -q^{\frac{1}{4}}\eta _{(3)(1)}O_{234} +q^{\frac{1}{4}}O_1\eta _{(234)(3)} -q^{\frac{1}{2}}\eta _{(56)(3)},\\[0.3em]
d_1(\eta_{(234)(4)})=& -q^{\frac{1}{2}}O_4O_1O_{234} +O_1O_{234}O_4 +q^{-\frac{1}{4}}(q-1)O_1O_{23} -q^{\frac{1}{4}}O_{56}O_4 +q^{\frac{3}{4}}O_4O_{56} -(q-1)O_{123}\\
=& -q^{\frac{1}{2}}\eta _{(4)(1)}O_{234} +q^{\frac{1}{4}}O_1\eta _{(234)(4)} -q^{\frac{1}{2}}\eta _{(56)(4)},\\[0.3em]
d_1(\eta_{(234)(5)})=& -O_5O_1O_{234} +q^{\frac{1}{2}}O_1O_{234}O_5 -q^{\frac{3}{4}}O_{56}O_5 +q^{\frac{1}{4}}O_5O_{56} -q^{-\frac{1}{4}}(q-1)O_1O_{61} +(q-1)O_6\\
=& -\eta _{(5)(1)}O_{234} +q^{\frac{1}{4}}O_1\eta _{(234)(5)} -q^{\frac{1}{2}}\eta _{(56)(5)},\\[0.3em]
d_1(\eta_{(234)(6)})=& -q^{\frac{1}{4}}O_{61}O_1O_{234} +q^{\frac{1}{4}}O_1O_{234}O_{61} +q^{\frac{1}{2}}O_{61}O_{56} -q^{\frac{1}{2}}O_{56}O_{61}\\
=& -q^{\frac{1}{2}}\eta _{(61)(1)}O_{234} +q^{\frac{1}{2}}O_1\eta _{(234)(61)} +\eta _{(61)(56)} -q^{-1}(q-1)\eta _{(5)(1)} +q^{-\frac{1}{2}}(q-1)\rho _3,\\[0.3em]
d_1(\eta_{(234)(12)})=&q^{-\frac{1}{4}}(q-1)O_{234}O_1O_2 -q^{\frac{1}{4}}O_2O_1O_{234} +q^{\frac{1}{4}}O_1O_{234}O_2 -(q-1)O_{234}O_{12} -q^{-\frac{1}{2}}(q-1)O_{34}O_1\\
& -q^{\frac{1}{2}}O_{56}O_2 +q^{\frac{1}{2}}O_2O_{56}\\
=& -q^{\frac{1}{2}}\eta _{(2)(1)}O_{234} +qO_1\eta _{(234)(2)} +(q-1)\eta _{(234)(1)}O_2 -(q-1)\eta _{(234)(12)} -q^{\frac{1}{2}}\eta _{(56)(2)} -q^{\frac{1}{2}}(q-1)\eta _{(34)(1)},\\[0.3em]
d_1(\eta_{(234)(23)})=&q^{\frac{3}{4}}O_1O_{234}O_1O_{23} -q^{\frac{1}{4}}O_1O_{23}O_1O_{234} +q^{\frac{1}{2}}O_{123}O_1O_{234} -qO_1O_{234}O_{123} -qO_{56}O_1O_{23} +q^{\frac{1}{2}}O_1O_{23}O_{56}\\
&-q^{\frac{3}{4}}O_{123}O_{56} +q^{\frac{5}{4}}O_{56}O_{123} -q^{-\frac{1}{2}}(q-1)O_4\\
=& -q^{\frac{1}{2}}O_1\eta _{(23)(1)}O_{234} +qO_1\eta _{(234)(1)}O_{23} +qO_1^2\eta _{(234)(23)} +q^{-\frac{3}{4}}t^{-\frac{1}{2}}(q-1) (q+t)O_1O_2\eta _{(6)(1)}\\ &+q^{-\frac{3}{4}}t^{-\frac{1}{2}}(q-1) (q+t)O_1\eta _{(2)(1)}O_6 +q^{\frac{1}{4}}\eta _{(123)(1)}O_{234} -q^{\frac{3}{4}}\eta _{(56)(1)}O_{23} -q^{\frac{1}{2}}O_1\eta _{(234)(123)}\\ &-q^{\frac{3}{2}}O_1\eta _{(56)(23)}-q^{-1}t^{-\frac{1}{2}}(q-1)^2 (q+t)O_1\eta _{(12)(6)} -q^{-\frac{1}{2}}(q-1) (q+1)O_1\eta _{(4)(1)} +(q-1)O_1\rho _{13}\\
&-(q-1)\eta _{(234)(23)}-q\eta _{(123)(56)},\\[0.3em]
d_1(\eta_{(234)(34)})=& -q^{\frac{1}{2}}O_{34}O_1O_{234} +O_1O_{234}O_{34} -q^{\frac{1}{4}}O_{56}O_{34} +q^{\frac{3}{4}}O_{34}O_{56} +q^{-\frac{1}{4}}(q-1)O_1O_2 -(q-1)O_{12}\\
=& -q^{\frac{1}{2}}\eta _{(34)(1)}O_{234} +q^{\frac{1}{4}}O_1\eta _{(234)(34)} -q^{\frac{1}{2}}\eta _{(56)(34)},\\[0.3em]
d_1(\eta_{(234)(45)})=& -q^{\frac{1}{4}}O_{45}O_1O_{234} +q^{\frac{1}{4}}O_1O_{234}O_{45} +q^{-\frac{1}{4}}(q-1)O_1O_{23}O_5 -q^{-\frac{1}{4}}(q-1)O_1O_{61}O_4 -q^{\frac{1}{2}}O_{56}O_{45}\\
&+q^{\frac{1}{2}}O_{45}O_{56} -(q-1)O_{123}O_5 +(q-1)O_6O_4\\
=& -q^{\frac{1}{4}}\eta _{(45)(1)}O_{234} +q^{\frac{1}{4}}O_1\eta _{(234)(45)} -(q-1)\eta _{(123)(5)} -\eta _{(56)(45)} +(q-1)\eta _{(6)(4)} -q^{-\frac{1}{2}}(q-1)\rho _2,\\[0.3em]
d_1(\eta_{(234)(56)})=& -O_{234}O_1O_{234} +q^{\frac{1}{2}}O_1O_{234}^2 +q^{\frac{1}{4}}O_{234}O_{56} -q^{\frac{3}{4}}O_{56}O_{234} -q^{-\frac{1}{2}}(q-1)O_1\\
=& -q^{\frac{1}{4}}\eta _{(234)(1)}O_{234} +\eta _{(234)(56)},\\[0.3em]
d_1(\eta_{(234)(61)})=&q^{\frac{1}{4}}O_1O_{234}O_1O_{61} -q^{\frac{3}{4}}O_1O_{61}O_1O_{234} -q^{\frac{1}{2}}O_{56}O_1O_{61} +qO_1O_{61}O_{56} +qO_6O_1O_{234} -q^{\frac{1}{2}}O_1O_{234}O_6\\
&+q^{\frac{3}{4}}O_{56}O_6 -q^{\frac{5}{4}}O_6O_{56} +q^{-\frac{1}{2}}(q-1)O_5\\
=& -qO_1\eta _{(61)(1)}O_{234} +q^{\frac{1}{2}}O_1\eta _{(234)(1)}O_{61} +qO_1^2\eta _{(234)(61)} +q^{\frac{3}{4}}\eta _{(6)(1)}O_{234} -q^{\frac{1}{4}}\eta _{(56)(1)}O_{61} -q^{\frac{1}{2}}O_1\eta _{(234)(6)}\\
&+q^{\frac{1}{2}}O_1\eta _{(61)(56)} -q^{-\frac{1}{2}}(q-1)O_1\eta _{(5)(1)} +(q-1)O_1\rho _3 -(q-1)\eta _{(234)(61)} +q\eta _{(56)(6)},\\[0.3em]
d_1(\eta_{(234)(123)})=&q^{-\frac{3}{4}}t^{-\frac{1}{2}}(q-1) (q+t)O_1O_2O_{61}O_1 -q^{-\frac{1}{4}}t^{-\frac{1}{2}}(q-1) (q+t)O_1O_2O_1O_{61} -q^{-\frac{1}{4}}O_{23}O_1O_{234} \\
&+q^{\frac{3}{4}}O_1O_{234}O_{23}-q^{-\frac{1}{2}}t^{-\frac{1}{2}}(q-1) (q+t)O_{12}O_{61}O_1 +t^{-\frac{1}{2}}(q-1) (q+t)O_{12}O_1O_{61} \\
& +t^{-\frac{1}{2}}(q-1) (q+t)O_1O_2O_6 -qO_{56}O_{23} +O_{23}O_{56} -q^{\frac{1}{4}}t^{-\frac{1}{2}}(q-1) (q+t)O_{12}O_6 \\
&-q^{-\frac{5}{4}}t^{-\frac{1}{2}}(q-1) (q+t)O_{61}O_2 -q^{-1}(q-1) (q+1)O_4O_1 +q^{-1}t^{-\frac{1}{2}}(q-1) (q+t)O_{345}\\
=&q^{-\frac{1}{2}}t^{-\frac{1}{2}}(q-1) (q+t)O_1O_2\eta _{(61)(1)} -\eta _{(23)(1)}O_{234} -q^{-\frac{1}{4}}t^{-\frac{1}{2}}(q-1) (q+t)O_{12}\eta _{(61)(1)} +q^{\frac{1}{2}}O_1\eta _{(234)(23)}\\
&-q^{-\frac{3}{2}}t^{-\frac{1}{2}}(q-1) (q+t)\eta _{(61)(2)} -q\eta _{(56)(23)} -q^{-\frac{1}{2}}t^{-\frac{1}{2}}(q-1) (q+t)\eta _{(12)(6)} -q^{-1}(q-1) (q+1)\eta _{(4)(1)}\\
& +q^{-\frac{1}{2}}(q-1)\rho _{13},\\[0.3em]
d_1(\eta_{(345)(2)})=&O_{345}O_1O_2 -q^{\frac{1}{2}}O_1O_2O_{345} -q^{\frac{1}{4}}O_{345}O_{12} +q^{\frac{3}{4}}O_{12}O_{345} +q^{-\frac{1}{4}}(q-1)O_1O_{61} -(q-1)O_6\\
=&q^{\frac{1}{4}}O_1\eta _{(345)(2)} +\eta _{(345)(1)}O_2 -q^{\frac{1}{2}}\eta _{(345)(12)},\\[0.3em]
d_1(\eta_{(345)(6)})=&q^{\frac{1}{4}}O_{345}O_{61} -q^{-\frac{1}{4}}O_{61}O_{345} -q^{-\frac{1}{2}}(q-1)O_2=\eta _{(345)(61)},\\[0.3em]
d_1(\eta_{(345)(12)})=&q^{-\frac{1}{4}}O_{345}O_2 -q^{\frac{1}{4}}O_2O_{345} +q^{-\frac{1}{2}}(q-1)O_{61}=\eta _{(345)(2)},\\[0.3em]
d_1(\eta_{(345)(23)})=&q^{\frac{1}{4}}O_{345}O_1O_{23} -q^{\frac{1}{4}}O_1O_{23}O_{345} -q^{\frac{1}{2}}O_{345}O_{123} +q^{\frac{1}{2}}O_{123}O_{345} -q^{-\frac{1}{4}}(q-1)O_{45}O_1O_2\\
&+q^{-\frac{1}{4}}(q-1)O_1O_{61}O_3 +(q-1)O_{45}O_{12} -(q-1)O_6O_3\\
=&q^{\frac{1}{4}}\eta _{(345)(1)}O_{23} +q^{\frac{1}{4}}O_1\eta _{(345)(23)} +q^{-\frac{5}{4}}t^{-\frac{1}{2}}(q-1) (q+t)O_1\eta _{(6)(5)} -q^{-\frac{1}{4}}(q-1)\eta _{(45)(1)}O_2\\
&+q^{-\frac{5}{4}}t^{-\frac{1}{2}}(q-1) (q+t)\eta _{(6)(1)}O_5 -\eta _{(345)(123)} +q^{-\frac{3}{2}}t^{-\frac{1}{2}}(q-1)^2 (q+t)\eta _{(61)(5)} +(q-1)\eta _{(45)(12)}\\
&-q^{-1}(q-1) (q+1)\eta _{(6)(3)} -q^{-\frac{1}{2}}(q-1)\rho _0,\\[0.3em]
d_1(\eta_{(345)(56)})=&O_{345}O_{234} -O_{234}O_{345} +q^{-\frac{1}{2}}(q-1)O_{34}O_{61} -q^{-\frac{1}{2}}(q-1)O_2O_5\\
=& -q^{-\frac{5}{4}}t^{-\frac{1}{2}}(q-1) (q+t)O_1\eta _{(3)(2)} -q^{-\frac{3}{2}}t^{-\frac{1}{2}}(q-1) (q+t)\eta _{(3)(1)}O_2 +q^{-\frac{3}{2}}t^{-\frac{1}{2}}(q-1) (q+t)O_2\eta _{(3)(1)}\\ &+q^{-\frac{5}{4}}t^{-\frac{1}{2}}(q-1) (q+t)\eta _{(2)(1)}O_3 +q^{-\frac{1}{2}}\eta _{(345)(234)} -q^{-2}t^{-\frac{1}{2}}(q-1)^2 (q+t)\eta _{(12)(3)}\\
&+q^{-\frac{3}{2}}(q-1) (q+1)\eta _{(5)(2)} -q^{-1}(q-1)\rho _{14},\\[0.3em]
d_1(\eta_{(345)(61)})=&q^{\frac{1}{2}}O_{345}O_1O_{61} -O_1O_{61}O_{345} -q^{\frac{3}{4}}O_{345}O_6 +q^{\frac{1}{4}}O_6O_{345} -q^{-\frac{1}{4}}(q-1)O_1O_2 +(q-1)O_{12}\\
=&q^{\frac{1}{2}}\eta _{(345)(1)}O_{61} +q^{\frac{1}{4}}O_1\eta _{(345)(61)} -q^{\frac{1}{2}}\eta _{(345)(6)},\\[0.3em]
d_1(\eta_{(345)(123)})=&q^{-1}t^{-\frac{1}{2}}(q-1) (q+t)O_5O_1O_{61} -q^{-1}t^{-\frac{1}{2}}(q-1) (q+t)O_1O_5O_{61} +q^{-\frac{1}{2}}O_{345}O_{23} -q^{\frac{1}{2}}O_{23}O_{345}\\
&+q^{-1}(q-1) (q+1)O_3O_{61} -q^{-\frac{3}{4}}t^{-\frac{1}{2}}(q-1) (q+t)O_5O_6 +q^{-\frac{1}{2}}t^{-\frac{1}{2}}(q-1) (q+t)O_{56}\\
=&q^{-1}t^{-\frac{1}{2}}(q-1) (q+t)\eta _{(5)(1)}O_{61} +q^{-\frac{1}{2}}\eta _{(345)(23)} -q^{-1}(q-1)\eta _{(61)(3)} +q^{-1}(q-1)\eta _{(45)(2)} -(q-1)\rho _{11},\\[0.3em]
d_1(\eta_{(345)(234)})=&q^{\frac{3}{4}}O_{345}O_1O_{234} -q^{-\frac{1}{4}}O_1O_{234}O_{345} +q^{-\frac{3}{4}}t^{-\frac{1}{2}}(q-1) (q+t)O_3O_1^2O_2 -q^{-\frac{1}{2}}t^{-\frac{1}{2}}(q-1) (q+t)O_3O_1O_{12}\\
&-q^{-1}t^{-\frac{1}{2}}(q-1) (q+t)O_1^2O_{23} -qO_{345}O_{56} +O_{56}O_{345} -q^{-\frac{3}{4}}(q-1) (q+1)O_5O_1O_2\\
&+q^{-\frac{3}{4}}t^{-\frac{1}{2}}(q-1) (q+t)O_1O_{123}+q^{-\frac{1}{2}}(q-1) (q+1)O_5O_{12} -q^{-\frac{3}{4}}t^{-\frac{1}{2}}(q-1) (q+t)O_3O_2\\ &+q^{-1}t^{-\frac{1}{2}}(q-1) (q+t)O_{23}\\
=&q^{-\frac{3}{4}}t^{-\frac{1}{2}}(q-1) (q+t)\eta _{(3)(1)}O_1O_2 +q^{\frac{3}{4}}\eta _{(345)(1)}O_{234} -q^{-\frac{1}{2}}t^{-\frac{1}{2}}(q-1) (q+t)\eta _{(3)(1)}O_{12}+q^{\frac{1}{4}}O_1\eta _{(345)(234)}\\
&-q^{-\frac{3}{4}}(q-1) (q+1)\eta _{(5)(1)}O_2 -q\eta _{(345)(56)} +q^{\frac{1}{2}}(q-1)\eta _{(34)(6)} -q^{\frac{1}{2}}(q-1)\eta _{(12)(5)}\\
&-q^{-\frac{1}{2}}t^{-\frac{1}{2}}(q-1) (q+t)\eta _{(3)(2)} -q^{\frac{1}{2}}(q-1)\rho _8.
\end{align*}

We omit the explicit action of $d_1$ on the q-Casimir relator $\rho_0$ which can be found at \cite{Arthamonov-GitHub-Flat}.

\section{Calculations for Theorem \ref{th:IsomorphismAq1t1CoordinateRingCharacterVariety}}
\label{sec:IdentiyCompositionPsiPhi}

In this section we present details of verification of identities (\ref{eq:IdentityCompositionPhiPsi}) used in the proof of Theorem \ref{th:IsomorphismAq1t1CoordinateRingCharacterVariety}. We start with a technical Lemma, which allows us to express the $SL(2,\mathbb C)$-invariant polynomials appearing on the right hand side of (\ref{eq:PsiActionOnGenerators}) in terms of generators (\ref{eq:ABL18Generators}).
\begin{lemma}
The following identities hold in $\mathcal O(\mathrm{Hom}(\pi_1(\Sigma_2),SL(2,\mathbb C)))$
\begin{align*}
\tau _{X_1Y_2X_2^{-1}Y_2^{-1}}=&\tau _{X_1X_2}-\tau _{X_1X_2} \tau _{Y_2}^2+\tau _{X_1} \tau _{X_2} \tau _{Y_2}^2+\tau _{Y_2} \tau _{X_1X_2Y_2}-\tau _{X_1} \tau _{Y_2} \tau _{X_2Y_2}-\tau _{X_2} \tau _{Y_2} \tau _{X_1Y_2}+\tau _{X_1Y_2} \tau _{X_2Y_2},\\
\tau _{Y_1X_1^{-1}}=&\tau _{X_1} \tau _{Y_1}-\tau _{X_1Y_1},\\
\tau _{X_1^{-1}Y_1^{-1}Y_2^{-1}}=&-\tau _{X_1Y_1Y_2}+\tau _{X_1} \tau _{Y_1Y_2}+\tau _{Y_1} \tau _{X_1Y_2}+\tau _{Y_2} \tau _{X_1Y_1}-\tau _{X_1} \tau _{Y_1} \tau _{Y_2},\\
\tau _{Y_2^{-1}Y_1^{-1}X_2}=&\tau _{Y_1X_2Y_2}-\tau _{Y_1} \tau _{X_2Y_2}-\tau _{Y_2} \tau _{Y_1X_2}+\tau _{X_2} \tau _{Y_1} \tau _{Y_2},\\
\tau _{X_1Y_2Y_2X_2^{-1}Y_2^{-1}}=&-\tau _{X_1X_2} \tau _{Y_2}^3+\tau _{X_1} \tau _{X_2} \tau _{Y_2}^3+\tau _{Y_2}^2 \tau _{X_1X_2Y_2}-\tau _{X_1} \tau _{Y_2}^2 \tau _{X_2Y_2}-\tau _{X_2} \tau _{Y_2}^2 \tau _{X_1Y_2}+2 \tau _{X_1X_2} \tau _{Y_2}\\
&\qquad+\tau _{Y_2} \tau _{X_1Y_2} \tau _{X_2Y_2}-\tau _{X_1} \tau _{X_2} \tau _{Y_2}-\tau _{X_1X_2Y_2}+\tau _{X_2} \tau _{X_1Y_2},\\
\tau _{X_1Y_1Y_2X_2^{-1}Y_2^{-1}}=&-\tau _{Y_1X_2}\tau _{X_1} \tau _{Y_2}^2 +\tau _{X_1} \tau _{X_2} \tau _{Y_1} \tau _{Y_2}^2+\tau _{Y_2} \tau _{X_1Y_2} \tau _{Y_1X_2}-\tau _{X_1X_2} \tau _{Y_1Y_2} \tau _{Y_2}+\tau _{X_1} \tau _{Y_2} \tau _{Y_1X_2Y_2}\\
&\qquad-\tau _{X_1} \tau _{Y_1} \tau _{Y_2} \tau _{X_2Y_2}-\tau _{X_2} \tau _{Y_1} \tau _{Y_2} \tau _{X_1Y_2}+\tau _{Y_1Y_2} \tau _{X_1X_2Y_2}-\tau _{X_1Y_1X_2}-\tau _{X_1Y_2} \tau _{Y_1X_2Y_2}\\
&\qquad+\tau _{X_1} \tau _{Y_1X_2}+\tau _{X_2} \tau _{X_1Y_1}+\tau _{X_1X_2} \tau _{Y_1}+\tau _{Y_1} \tau _{X_1Y_2} \tau _{X_2Y_2}-\tau _{X_1} \tau _{X_2} \tau _{Y_1},\\
\tau _{Y_1X_1^{-1}Y_1^{-1}Y_2^{-1}}=&\tau _{Y_1}^2 \tau _{X_1Y_2}-\tau _{X_1} \tau _{Y_2} \tau _{Y_1}^2-\tau _{Y_1} \tau _{X_1Y_1Y_2}+\tau _{X_1} \tau _{Y_1Y_2} \tau _{Y_1}+\tau _{Y_2} \tau _{Y_1} \tau _{X_1Y_1}-\tau _{X_1Y_2}-\tau _{Y_1Y_2} \tau _{X_1Y_1}+\tau _{X_1} \tau _{Y_2},\\
\tau _{X_1Y_1Y_2Y_2X_2^{-1}Y_2^{-1}}=&\frac{1}{2} \big(-2 \tau _{X_1} \tau _{Y_2}^3 \tau _{Y_1X_2}+2 \tau _{X_1} \tau _{X_2} \tau _{Y_1} \tau _{Y_2}^3+2 \tau _{Y_2}^2 \tau _{X_1Y_2} \tau _{Y_1X_2}-2 \tau _{X_1X_2} \tau _{Y_1Y_2} \tau _{Y_2}^2+2 \tau _{X_1} \tau _{Y_2}^2 \tau _{Y_1X_2Y_2}\\
&\qquad-2 \tau _{X_1} \tau _{Y_1} \tau _{Y_2}^2 \tau _{X_2Y_2}-2 \tau _{X_2} \tau _{Y_1} \tau _{Y_2}^2 \tau _{X_1Y_2}+2 \tau _{Y_1Y_2} \tau _{Y_2} \tau _{X_1X_2Y_2}-\tau _{Y_2} \tau _{X_1Y_1X_2}\\
&\qquad-2 \tau _{Y_2} \tau _{X_1Y_2} \tau _{Y_1X_2Y_2}+3 \tau _{X_1} \tau _{Y_2} \tau _{Y_1X_2}+\tau _{X_2} \tau _{Y_2} \tau _{X_1Y_1}+2 \tau _{X_1X_2} \tau _{Y_1} \tau _{Y_2}+2 \tau _{Y_1} \tau _{Y_2} \tau _{X_1Y_2} \tau _{X_2Y_2}\\
&\qquad-3 \tau _{X_1} \tau _{X_2} \tau _{Y_1} \tau _{Y_2}-\tau _{X_1Y_1} \tau _{X_2Y_2}-\tau _{X_1Y_2} \tau _{Y_1X_2}+\tau _{X_1X_2} \tau _{Y_1Y_2}-\tau _{X_1} \tau _{Y_1X_2Y_2}+\tau _{X_2} \tau _{X_1Y_1Y_2}\\
&\qquad-\tau _{Y_1} \tau _{X_1X_2Y_2}+\tau _{X_1} \tau _{Y_1} \tau _{X_2Y_2}+\tau _{X_2} \tau _{Y_1} \tau _{X_1Y_2}\big),\\
\tau _{X_2Y_1^{-1}}=&\tau _{X_2} \tau _{Y_1}-\tau _{Y_1X_2}.
\end{align*}
\label{lemm:FullRepresentationVarietyRelationsAux1}
\end{lemma}
\begin{proof}
We compute the Groebner basis in the coordinate ring $\mathcal O(\mathrm{Hom}(\pi_1(\Sigma_2),SL(2,\mathbb C)))$ of the full representation variety in degree reverse lexicographic order and reduce both sides of the equation modulo this basis. The source code of the program in Mathematica is available at \cite{Arthamonov-GitHub-Flat}.
\end{proof}

Combining $\Psi$-action on generators (\ref{eq:PsiActionOnGenerators}) with Lemma \ref{lemm:FullRepresentationVarietyRelationsAux1} we obtain formulas for the action of $\iota\circ\Psi$. Now we are ready to verify the first of the identities in (\ref{eq:IdentityCompositionPhiPsi}).
\begin{lemma}
The following composition of homomorphisms represents and identity map on $\mathcal A_{q=t=1}$.
\begin{equation*}
\Phi\circ\iota\circ\Psi=\mathrm{Id}_{\mathcal A_{q=t=1}}
\end{equation*}
\end{lemma}
\begin{proof}
Generators
\begin{equation*}
O_1,O_2,O_3,O_4,O_5,O_{12},O_{34},O_{45},O_{345}
\end{equation*}
are preserved verbatim, as for the rest we get
\begin{align*}
\Phi\circ\iota\circ\Psi(O_6)=&O_1 O_2 O_3 O_4 O_5 -O_3 O_4 O_5 O_{12} -O_1 O_4 O_5 O_{23} -O_1 O_2 O_5 O_{34} -O_1 O_2 O_3 O_{45} +O_5 O_{12} O_{34}\\
&\qquad +O_3 O_{12} O_{45} +O_1 O_{23} O_{45} +O_4 O_5 O_{123} +O_1 O_5 O_{234} +O_1 O_2 O_{345} -O_{45} O_{123}\\
&\qquad -O_{12} O_{345} -O_5 O_{56} -O_1 O_{61} +O_6\\
=&O_6+ \left(O_1 O_2-O_{1,2}\right)g_7^{(q=t=1)} +O_1g_{29}^{(q=t=1)} +O_2g_2^{(q=t=1)} -O_6g_{18}^{(q=t=1)} -g_{34}^{(q=t=1)}\\
\equiv&O_6\bmod \mathcal I_{q=t=1},\\[5pt]
\Phi\circ\iota\circ\Psi(O_{23})=&-O_1^2 O_2 O_3 +O_1 O_3 O_{12} +O_1^2 O_{23} +O_4 O_5 O_{61} -O_{45} O_{61} -O_1 O_{123} -O_4 O_{234} +O_2 O_3 +O_{23}\\
=&O_{23} -O_1g_9^{(q=t=1)} -g_{26}^{(q=t=1)}\\
\equiv&O_{23}\bmod \mathcal I_{q=t=1},\\[5pt]
\Phi\circ\iota\circ\Psi(O_{56})=&O_1 O_2 O_3 O_4 O_5^2 -O_3 O_4 O_5^2 O_{12} -O_1 O_4 O_5^2 O_{23} -O_1 O_2 O_5^2 O_{34} -O_1 O_2 O_3 O_5 O_{45} +O_5^2 O_{12} O_{34}\\
&\qquad +O_3 O_5 O_{12} O_{45} +O_1 O_5 O_{23} O_{45} +O_4 O_5^2 O_{123} +O_1 O_5^2 O_{234} +O_1 O_2 O_5 O_{345} -O_5 O_{45} O_{123}\\
&\qquad -O_5 O_{12} O_{345} -O_1 O_2 O_3 O_4 +O_3 O_4 O_{12} +O_1 O_4 O_{23} +O_1 O_2 O_{34} -O_5^2 O_{56} -O_1 O_5 O_{61}\\
&\qquad -O_{12} O_{34} -O_4 O_{123} -O_1 O_{234} +O_5 O_6 +O_{56}\\
=&O_{56} +\left(O_1 O_2 O_5-O_5 O_{12}\right)g_7^{(q=t=1)} -O_1g_{43}^{(q=t=1)} +O_1O_2g_{27}^{(q=t=1)} -O_5g_{34}^{(q=t=1)}\\
&\qquad +O_2O_5g_2^{(q=t=1)} -O_5O_6g_{18}^{(q=t=1)} -g_{28}^{(q=t=1)}\\
\equiv&O_{56}\bmod \mathcal I_{q=t=1},\\[5pt]
\Phi\circ\iota\circ\Psi(O_{61})=&-O_1 O_{23} O_{34} +O_1 O_3 O_{234} +O_{34} O_{123} +O_1 O_2 O_4 -O_4 O_{12} -O_3 O_{56} -O_{61}\\
=&O_{61} -O_1g_5^{(q=t=1)} +g_{15}^{(q=t=1)}\\
\equiv&O_{61}\bmod \mathcal I_{q=t=1},\\[5pt]
\Phi\circ\iota\circ\Psi(O_{123})=&-O_1^3 O_2 O_3 +O_1^2 O_3 O_{12} +O_1^3 O_{23} +O_1 O_4 O_5 O_{61} -O_1 O_{45} O_{61} -O_1^2 O_{123} -O_1 O_4 O_{234}\\
&\qquad +O_1 O_2 O_3 +O_{123}\\
=&O_{123} -O_1^2g_9^{(q=t=1)} -O_1g_{26}^{(q=t=1)}\\
\equiv&O_{123}\bmod \mathcal I_{q=t=1},\\[5pt]
\Phi\circ\iota\circ\Psi(O_{234})=&-O_1 O_5 O_{23} O_{34} +O_1 O_3 O_5 O_{234} +\frac{1}{2}O_1^2 O_2 O_{34} -\frac{1}{2}O_4^2 O_5 O_{61} +O_5 O_{34} O_{123} +\frac{1}{2}O_1 O_{23} O_{345}\\
&\qquad +O_1 O_2 O_4 O_5 -\frac{1}{2}O_1 O_{12} O_{34} +\frac{1}{2}O_4 O_{45} O_{61} -\frac{1}{2}O_1^2 O_{234} +\frac{1}{2}O_4^2 O_{234} -\frac{1}{2}O_{123} O_{345}\\
&\qquad -O_4 O_5 O_{12} -\frac{1}{2}O_1 O_2 O_{45} -O_3 O_5 O_{56} -\frac{1}{2}O_1 O_3 O_{61} +\frac{1}{2}O_{12} O_{45} -\frac{1}{2}O_4 O_{23} -\frac{1}{2}O_2 O_{34}\\
&\qquad +\frac{1}{2}O_1 O_{56} -\frac{1}{2}O_5 O_{61} +\frac{1}{2}O_3 O_6\\
=&O_{2,3,4} +\frac{1}{2}O_1g_{11}^{(q=t=1)} +O_5g_{15}^{(q=t=1)} -O_1O_5g_5^{(q=t=1)} -\frac{1}{2}g_{20}^{(q=t=1)} +\frac{1}{2}g_{39}^{(q=t=1)}\\
\equiv&O_{234}\bmod \mathcal I_{q=t=1}.
\end{align*}
\end{proof}

\begin{lemma}
The following combination of homomorphisms represents an identity map on the coordinate ring of the character variety
\begin{equation*}
\Psi\circ\Phi\circ\iota=\mathrm{Id}_{\mathcal O(\mathrm{Hom}(\pi_1(\Sigma_2),SL(2,\mathbb C)))^{SL(2,\mathbb C)}}.
\end{equation*}
\end{lemma}
\begin{proof}
Recall that invariant polynomials associated to (\ref{eq:ABL18Generators}) provide a complete set of generators of the $SL(2,\mathbb C)$-invariant subring. We will show that $\Psi\circ\Phi\circ\iota$ maps all elements (\ref{eq:ABL18Generators}) into the equivalent invariant polynomials modulo defining ideal of the coordinate ring of representation variety. The action on generators given by (\ref{eq:PsiActionOnGenerators}) and (\ref{eq:PhiActionOnGenerators}) implies that this property holds verbatim for
\begin{equation*}
\tau_{X_1},\tau_{Y_1},\tau_{X_2},\tau_{Y_2},\tau_{Y_1Y_2},\tau_{X_2Y_2}.
\end{equation*}
As for the rest of the generators we describe our calculation below.

Denote the 16 generators of $\mathcal O(\mathrm{Hom}(\pi_1(\Sigma_2)),SL(2,\mathbb C))$ by
\begin{equation*}
\left(X_i\right)_{jk},\left(Y_i\right)_{jk},\qquad 1\leq i,j,k\leq2.
\end{equation*}
and let $\mathcal I_{\mathrm{Hom}(\pi_1(\Sigma_2),SL(2,\mathbb C))}$ be defining ideal of the coordinate ring of representation variety. We get
\begin{align*}
\Psi\circ\Phi\circ\iota(\tau_{X_1Y_1})=&\Psi(O_1 O_2 -O_{12})\\
=&\tau _{X_1} \tau _{Y_1}-\tau _{Y_1X_1^{-1}}\\
=&\left(X_1\right)_{11} \left(Y_1\right)_{11}+\left(X_1\right)_{21} \left(Y_1\right)_{12}+\left(X_1\right)_{12} \left(Y_1\right)_{21}+\left(X_1\right)_{22} \left(Y_1\right)_{22}\\
=&\tau_{X_1Y_1},\\[7pt]
\Psi\circ\Phi\circ\iota(\tau _{X_1X_2})=&\Psi(O_1 O_2 O_{345} -O_{12} O_{345} -O_1 O_{61} +O_6)\\
=&\tau _{X_1Y_2X_2^{-1}Y_2^{-1}}-\tau _{Y_1} \tau _{X_1Y_1Y_2X_2^{-1}Y_2^{-1}}+\tau _{X_2Y_1^{-1}} \left(\tau _{X_1} \tau _{Y_1}-\tau _{Y_1X_1^{-1}}\right)\\
\equiv&\left(X_1\right)_{11} \left(X_2\right)_{11}+\left(X_1\right)_{21} \left(X_2\right)_{12}+\left(X_1\right)_{12} \left(X_2\right)_{21}+\left(X_1\right)_{22} \left(X_2\right)_{22}\bmod \mathcal I_{\mathrm{Hom}(\pi_1(\Sigma_2),SL(2,\mathbb C))}\\
\equiv&\tau_{X_1X_2}\bmod \mathcal I_{\mathrm{Hom}(\pi_1(\Sigma_2),SL(2,\mathbb C))},\\[7pt]
\Psi\circ\Phi\circ\iota(\tau _{X_1Y_2})=&\Psi(-O_1 O_2 O_3 +O_3 O_{12} +O_1 O_{23} +O_2 O_5 -O_{123})\\
=&\tau _{Y_1Y_2} \tau _{Y_1X_1^{-1}}-\tau _{Y_1X_1^{-1}Y_1^{-1}Y_2^{-1}}+\tau _{Y_1} \tau _{X_1^{-1}Y_1^{-1}Y_2^{-1}}-\tau _{X_1} \tau _{Y_1} \tau _{Y_1Y_2}+\tau _{X_1} \tau _{Y_2}\\
\equiv&\left(X_1\right)_{11} \left(Y_2\right)_{11}+\left(X_1\right)_{21} \left(Y_2\right)_{12}+\left(X_1\right)_{12} \left(Y_2\right)_{21}+\left(X_1\right)_{22} \left(Y_2\right)_{22}\bmod \mathcal I_{\mathrm{Hom}(\pi_1(\Sigma_2),SL(2,\mathbb C))}\\
\equiv&\tau _{X_1Y_2}\bmod \mathcal I_{\mathrm{Hom}(\pi_1(\Sigma_2),SL(2,\mathbb C))},\\[7pt]
\Psi\circ\Phi\circ\iota(\tau _{Y_1X_2})=&\Psi(O_1 O_4 -O_{345})\\
=&\tau _{X_2} \tau _{Y_1}-\tau _{X_2Y_1^{-1}}\\
=&\left(X_2\right)_{11} \left(Y_1\right)_{11}+\left(X_2\right)_{21} \left(Y_1\right)_{12}+\left(X_2\right)_{12} \left(Y_1\right)_{21}+\left(X_2\right)_{22} \left(Y_1\right)_{22}\\
=&\tau _{Y_1X_2},\\[7pt]
\Psi\circ\Phi\circ\iota(\tau _{X_1Y_1X_2})=&\Psi(O_1^2 O_2 O_{345} -O_1 O_{12} O_{345} -O_1^2 O_{61} -O_2 O_{345} +O_1 O_6 +O_{61})\\
=&\tau _{Y_1} \tau _{X_1Y_2X_2^{-1}Y_2^{-1}}-\left(\tau _{Y_1}^2-1\right) \tau _{X_1Y_1Y_2X_2^{-1}Y_2^{-1}}+\tau _{X_2Y_1^{-1}} \left(\tau _{X_1} \left(\tau _{Y_1}^2-1\right)-\tau _{Y_1} \tau _{Y_1X_1^{-1}}\right)\\
\equiv&\left(X_1\right)_{11} \left(X_2\right)_{11} \left(Y_1\right)_{11}+\left(X_1\right)_{21} \left(X_2\right)_{12} \left(Y_1\right)_{11}+\left(X_1\right)_{11} \left(X_2\right)_{21} \left(Y_1\right)_{12}+\left(X_1\right)_{21} \left(X_2\right)_{22} \left(Y_1\right)_{12}\\
&\quad+\left(X_1\right)_{12} \left(X_2\right)_{11} \left(Y_1\right)_{21}+\left(X_1\right)_{22} \left(X_2\right)_{12} \left(Y_1\right)_{21}+\left(X_1\right)_{12} \left(X_2\right)_{21} \left(Y_1\right)_{22}\\
&\quad+\left(X_1\right)_{22} \left(X_2\right)_{22} \left(Y_1\right)_{22}\bmod \mathcal I_{\mathrm{Hom}(\pi_1(\Sigma(2)),SL(2,\mathbb C))}\\
\equiv&\tau _{X_1Y_1X_2}\bmod \mathcal I_{\mathrm{Hom}(\pi_1(\Sigma_2),SL(2,\mathbb C))},\\[7pt]
\Psi\circ\Phi\circ\iota(\tau _{X_1Y_1Y_2})=&\Psi( -O_4 O_5 O_{61} +O_1 O_2 O_5 +O_{45} O_{61} +O_4 O_{234} -O_5 O_{12} -O_{23})\\
=&-\tau _{X_1^{-1}Y_1^{-1}Y_2^{-1}}+\tau _{X_2Y_2} \tau _{X_1Y_1Y_2X_2^{-1}Y_2^{-1}}+\tau _{X_2} \tau _{X_1Y_1Y_2Y_2X_2^{-1}Y_2^{-1}}-\tau _{Y_2} \tau _{Y_1X_1^{-1}}\\
&\quad-\tau _{X_2} \tau _{Y_2} \tau _{X_1Y_1Y_2X_2^{-1}Y_2^{-1}}+\tau _{X_1} \tau _{Y_1} \tau _{Y_2}\\
\equiv&\left(X_1\right)_{11} \left(Y_1\right)_{11} \left(Y_2\right)_{11}+\left(X_1\right)_{12} \left(Y_1\right)_{21} \left(Y_2\right)_{11}+\left(X_1\right)_{21} \left(Y_1\right)_{11} \left(Y_2\right)_{12}+\left(X_1\right)_{22} \left(Y_1\right)_{21} \left(Y_2\right)_{12}\\
&\quad+\left(X_1\right)_{11} \left(Y_1\right)_{12} \left(Y_2\right)_{21}+\left(X_1\right)_{12} \left(Y_1\right)_{22} \left(Y_2\right)_{21}+\left(X_1\right)_{21} \left(Y_1\right)_{12} \left(Y_2\right)_{22}\\
&\quad+\left(X_1\right)_{22} \left(Y_1\right)_{22} \left(Y_2\right)_{22}\bmod \mathcal I_{\mathrm{Hom}(\pi_1(\Sigma_2),SL(2,\mathbb C))}\\
\equiv&\tau _{X_1Y_1Y_2}\bmod \mathcal I_{\mathrm{Hom}(\pi_1(\Sigma_2),SL(2,\mathbb C))},\\[7pt]
\Psi\circ\Phi\circ\iota(\tau _{X_1X_2Y_2})=&\Psi(O_1 O_2 O_5 O_{345} -O_5 O_{12} O_{345} -O_1 O_2 O_{34} -O_1 O_5 O_{61} +O_{12} O_{34} +O_1 O_{234} +O_5 O_6 -O_{56})\\
=&-\tau _{X_1Y_2Y_2X_2^{-1}Y_2^{-1}}+\tau _{Y_1} \tau _{X_1Y_1Y_2Y_2X_2^{-1}Y_2^{-1}}-\tau _{X_1} \tau _{Y_1} \tau _{Y_2^{-1}Y_1^{-1}X_2}+\tau _{Y_2} \tau _{X_1Y_2X_2^{-1}Y_2^{-1}}\\
&\quad-\tau _{Y_1} \tau _{Y_2} \tau _{X_1Y_1Y_2X_2^{-1}Y_2^{-1}}+\tau _{X_1} \tau _{Y_1} \tau _{Y_2} \tau _{X_2Y_1^{-1}}+\tau _{Y_1X_1^{-1}} \left(\tau _{Y_2^{-1}Y_1^{-1}X_2}-\tau _{Y_2} \tau _{X_2Y_1^{-1}}\right)\\
\equiv&\left(X_1\right)_{11} \left(X_2\right)_{11} \left(Y_2\right)_{11}+\left(X_1\right)_{12} \left(X_2\right)_{21} \left(Y_2\right)_{11}+\left(X_1\right)_{21} \left(X_2\right)_{11} \left(Y_2\right)_{12}+\left(X_1\right)_{22} \left(X_2\right)_{21} \left(Y_2\right)_{12}\\
&\quad +\left(X_1\right)_{11} \left(X_2\right)_{12} \left(Y_2\right)_{21}+\left(X_1\right)_{12} \left(X_2\right)_{22} \left(Y_2\right)_{21}+\left(X_1\right)_{21} \left(X_2\right)_{12} \left(Y_2\right)_{22}\\
&\quad+\left(X_1\right)_{22} \left(X_2\right)_{22} \left(Y_2\right)_{22}\bmod \mathcal I_{\mathrm{Hom}(\pi_1(\Sigma_2),SL(2,\mathbb C))}\\
\equiv&\tau _{X_1X_2Y_2}\bmod \mathcal I_{\mathrm{Hom}(\pi_1(\Sigma_2),SL(2,\mathbb C))},\\[7pt]
\Psi\circ\Phi\circ\iota(\tau _{Y_1X_2Y_2})=&\Psi(-O_5 O_{345} +O_1 O_{45} +O_{34})\\
=&\tau _{Y_2^{-1}Y_1^{-1}X_2}-\tau _{Y_2} \tau _{X_2Y_1^{-1}}+\tau _{Y_1} \tau _{X_2Y_2}\\
=&\left(X_2\right)_{11} \left(Y_1\right)_{11} \left(Y_2\right)_{11}+\left(X_2\right)_{21} \left(Y_1\right)_{12} \left(Y_2\right)_{11}+\left(X_2\right)_{11} \left(Y_1\right)_{21} \left(Y_2\right)_{12}+\left(X_2\right)_{21} \left(Y_1\right)_{22} \left(Y_2\right)_{12}\\
&\quad+\left(X_2\right)_{12} \left(Y_1\right)_{11} \left(Y_2\right)_{21}+\left(X_2\right)_{22} \left(Y_1\right)_{12} \left(Y_2\right)_{21}+\left(X_2\right)_{12} \left(Y_1\right)_{21} \left(Y_2\right)_{22}+\left(X_2\right)_{22} \left(Y_1\right)_{22} \left(Y_2\right)_{22}\\
=&\tau _{Y_1X_2Y_2}
\end{align*}

\end{proof}

\section{$q$-Groebner basis}
\label{sec:qGroebnerBasis}

In this subsection we introduce 61 relations in $\mathcal A_{q,t}$ which present a deformation of the Groebner basis for defining ideal of $\mathcal A_{q=1,t}$. We assign weights to generators as in (\ref{eq:GeneratorWeights}) and present every relator as a combination of normally ordered monomials sorted according to weighted degree reverse lexicographic monomial order. In our formulas the leading monomial always comes first and we normalize the relator so that the coefficient in leading term is 1. This simplifies the comparison with commutative Groebner bases for defining ideals of $\mathcal A_{q=1,t}$ and $\mathcal A_{q=t=1}$.

Recall that by $F$ we denote a free $\mathbf k$-algebra with 15 generators. (\ref{eq:15Generators})
\begin{proposition}
The following 61 elements of $F$ belong to defining ideal of $\mathcal A_{q,t}$.
\begin{align*}
g_{1}=&O_{56}O_{61} -q^{\frac{1}{2}}O_6O_{234} -q^{-\frac{1}{2}}O_1O_5 +q^{-\frac{1}{2}}t^{-\frac{1}{2}}(q+t)O_3,\\
g_{2}=&O_{12}O_{61} -q^{-\frac{1}{2}}O_1O_{345} -q^{\frac{1}{2}}O_2O_6 +q^{-\frac{1}{2}}t^{-\frac{1}{2}}(q+t)O_4,\\
g_{3}=&O_{45}O_{56} -q^{\frac{1}{2}}O_5O_{123} -q^{-\frac{1}{2}}O_4O_6 +q^{-\frac{1}{2}}t^{-\frac{1}{2}}(q+t)O_2,\\
g_{4}=&O_{34}O_{45} -q^{\frac{1}{2}}O_4O_{345} -q^{-\frac{1}{2}}O_3O_5 +q^{-\frac{1}{2}}t^{-\frac{1}{2}}(q+t)O_1,\\
g_{5}=&O_{23}O_{34} -q^{\frac{1}{2}}O_3O_{234} -q^{-\frac{1}{2}}O_2O_4 +q^{-\frac{1}{2}}t^{-\frac{1}{2}}(q+t)O_6,\\
g_{6}=&O_{12}O_{23} -q^{\frac{1}{2}}O_2O_{123} -q^{-\frac{1}{2}}O_1O_3 +q^{-\frac{1}{2}}t^{-\frac{1}{2}}(q+t)O_5,\\
g_{7}=&O_3O_4O_5 -q^{-\frac{1}{2}}O_1O_2O_6 +q^{\frac{1}{4}}O_6O_{12} -q^{-\frac{1}{4}}O_5O_{34} -q^{\frac{1}{4}}O_3O_{45} +q^{-\frac{5}{4}}O_2O_{61} -q^{-\frac{3}{2}}(q-1)^2O_{345},\\
g_{8}=&O_2O_3O_4 -q^{-\frac{1}{2}}O_1O_5O_6 -q^{-\frac{1}{4}}O_4O_{23} -q^{\frac{1}{4}}O_2O_{34} +q^{-\frac{1}{4}}O_1O_{56} +q^{-\frac{3}{4}}O_5O_{61},\\
g_{9}=&O_1O_2O_3 -O_4O_5O_6 -q^{-\frac{1}{4}}O_3O_{12} -q^{\frac{1}{4}}O_1O_{23} +q^{-\frac{1}{4}}O_6O_{45} +q^{\frac{1}{4}}O_4O_{56},\\
g_{10}=&O_{56}O_{345} -q^{-\frac{1}{2}}O_5O_{12} -q^{\frac{1}{2}}O_6O_{34} +q^{-\frac{1}{4}}t^{-\frac{1}{2}}(q+t)O_2O_3 -t^{-\frac{1}{2}}(q+t)O_{23},\\
g_{11}=&O_{23}O_{345} -q^{-\frac{1}{2}}O_2O_{45} -q^{\frac{1}{2}}O_3O_{61} +q^{-\frac{1}{4}}t^{-\frac{1}{2}}(q+t)O_5O_6 -t^{-\frac{1}{2}}(q+t)O_{56},\\
g_{12}=&O_{45}O_{234} -q^{\frac{1}{2}}O_5O_{23} -q^{-\frac{1}{2}}O_4O_{61} +q^{-\frac{1}{4}}t^{-\frac{1}{2}}(q+t)O_1O_2 -t^{-\frac{1}{2}}(q+t)O_{12},\\
g_{13}=&O_{12}O_{234} -q^{-\frac{1}{2}}O_1O_{34} -q^{\frac{1}{2}}O_2O_{56} +q^{-\frac{1}{4}}t^{-\frac{1}{2}}(q+t)O_4O_5 -t^{-\frac{1}{2}}(q+t)O_{45},\\
g_{14}=&O_{61}O_{123} -q^{-\frac{1}{2}}O_6O_{23} -q^{\frac{1}{2}}O_1O_{45} +q^{-\frac{1}{4}}t^{-\frac{1}{2}}(q+t)O_3O_4 -t^{-\frac{1}{2}}(q+t)O_{34},\\
g_{15}=&O_{34}O_{123} -q^{\frac{1}{2}}O_4O_{12} -q^{-\frac{1}{2}}O_3O_{56} +q^{-\frac{3}{4}}t^{-\frac{1}{2}}(q+t)O_1O_6 -q^{-1}t^{-\frac{1}{2}}(q+t)O_{61},\\
g_{16}=&O_3O_4O_{34} -q^{-\frac{1}{2}}O_1O_6O_{61} -q^{\frac{1}{4}}O_{34}^2 +q^{-\frac{3}{4}}O_{61}^2 +q^{-\frac{3}{4}}O_1^2 -q^{\frac{1}{4}}O_3^2 -q^{-\frac{3}{4}}O_4^2 +q^{\frac{1}{4}}O_6^2,\\
g_{17}=&O_2O_3O_{23} -O_5O_6O_{56} -q^{\frac{1}{4}}O_{23}^2 +q^{\frac{1}{4}}O_{56}^2 -q^{\frac{1}{4}}O_2^2 -q^{-\frac{3}{4}}O_3^2 +q^{\frac{1}{4}}O_5^2 +q^{-\frac{3}{4}}O_6^2,\\
g_{18}=&O_1O_2O_{12} -O_4O_5O_{45} -q^{\frac{1}{4}}O_{12}^2 +q^{\frac{1}{4}}O_{45}^2 -q^{\frac{1}{4}}O_1^2 -q^{-\frac{3}{4}}O_2^2 +q^{\frac{1}{4}}O_4^2 +q^{-\frac{3}{4}}O_5^2,\\
g_{19}=&O_{234}O_{345} +t^{-\frac{1}{2}}(q+t)O_4O_5O_6 -q^{\frac{1}{2}}O_{34}O_{61} -q^{-\frac{1}{4}}t^{-\frac{1}{2}}(q+t)O_6O_{45} -q^{\frac{1}{4}}t^{-\frac{1}{2}}(q+t)O_4O_{56}\\
&-q^{-\frac{1}{2}}O_2O_5 +q^{-\frac{1}{2}}t^{-\frac{1}{2}}(q+t)O_{123},\\
g_{20}=&O_{123}O_{345} +q^{-\frac{1}{2}}t^{-\frac{1}{2}}(q+t)O_1O_5O_6 -q^{-\frac{1}{2}}O_{12}O_{45} -q^{-\frac{1}{4}}t^{-\frac{1}{2}}(q+t)O_1O_{56} -q^{-\frac{3}{4}}t^{-\frac{1}{2}}(q+t)O_5O_{61}\\
&-q^{\frac{1}{2}}O_3O_6 +q^{-\frac{1}{2}}t^{-\frac{1}{2}}(q+t)O_{234},\\
g_{21}=&O_{123}O_{234} +q^{-\frac{1}{2}}t^{-\frac{1}{2}}(q+t)O_1O_2O_6 -q^{\frac{1}{2}}O_{23}O_{56} -q^{\frac{1}{4}}t^{-\frac{1}{2}}(q+t)O_6O_{12} -q^{-\frac{5}{4}}t^{-\frac{1}{2}}(q+t)O_2O_{61}\\
&-q^{-\frac{1}{2}}O_1O_4 +q^{-\frac{3}{2}}t^{-\frac{1}{2}}\left(q^2-q+1\right) (q+t)O_{345},\\
g_{22}=&O_4O_5O_6^2 -q^{-\frac{3}{4}}O_6^2O_{45} -q^{\frac{3}{4}}O_4O_6O_{56} -q^{-\frac{3}{4}}O_2O_3O_{61} +q^{-\frac{1}{2}}O_{23}O_{61} +q^{-1}O_6O_{123} +q^{-\frac{3}{2}}O_3O_{345}\\
&+(q-2)O_4O_5 +q^{-\frac{7}{4}}(q-1)O_{45},\\
g_{23}=&O_1O_2O_6^2 -q^{\frac{5}{4}}O_6^2O_{12} -q^{\frac{3}{4}}O_3O_4O_{56} -q^{-\frac{5}{4}}O_2O_6O_{61} +qO_{34}O_{56} +qO_3O_{123}\\
&+q^{-\frac{3}{2}}(q^3-q+1)O_6O_{345} -q^{-1}(2 q-1)O_1O_2,\\
g_{24}=&O_1O_5^2O_6 -q^{\frac{1}{4}}O_2O_3O_{45} -q^{\frac{1}{4}}O_1O_5O_{56} -q^{-\frac{1}{4}}O_5^2O_{61} +q^{\frac{1}{2}}O_{23}O_{45} +O_5O_{234} +q^{\frac{1}{2}}O_2O_{345}\\
&-O_1O_6 -q^{-\frac{1}{4}}(q-1)O_{61},\\
g_{25}=&O_4^2O_5O_6 -q^{-\frac{1}{4}}O_1O_2O_{34} -q^{-\frac{1}{4}}O_4O_6O_{45} -q^{\frac{1}{4}}O_4^2O_{56} +O_{12}O_{34} +q^{-\frac{1}{2}}O_4O_{123} +O_1O_{234} -O_5O_6,\\
g_{26}=&O_1O_4O_5O_6 -q^{-\frac{1}{4}}O_1O_6O_{45} -q^{\frac{1}{4}}O_1O_4O_{56} -q^{-\frac{1}{4}}O_4O_5O_{61} +O_{45}O_{61} +q^{-\frac{1}{2}}O_1O_{123} +O_4O_{234} -O_2O_3,\\
g_{27}=&O_1O_2O_5O_6 -q^{\frac{3}{4}}O_5O_6O_{12} -q^{\frac{1}{4}}O_1O_2O_{56} -q^{-\frac{3}{4}}O_2O_5O_{61} +q^{\frac{1}{2}}O_{12}O_{56} +q^{-\frac{1}{2}}O_2O_{234} +q^{-1}(q^2-q+1)O_5O_{345}\\
&-O_3O_4 +q^{-\frac{3}{4}}(q-1)O_{34},\\
g_{28}=&O_1^2O_5O_6 -q^{\frac{1}{4}}O_3O_4O_{12} -q^{\frac{1}{4}}O_1^2O_{56} -q^{-\frac{1}{4}}O_1O_5O_{61} +qO_{12}O_{34} +q^{-\frac{1}{2}}O_4O_{123} +O_1O_{234}\\ &-O_5O_6 -q^{-\frac{3}{4}}(q-1)^2O_{56},\\
g_{29}=&O_1O_2^2O_6 -q^{\frac{5}{4}}O_2O_6O_{12} -q^{-\frac{1}{4}}O_4O_5O_{23} -q^{-\frac{5}{4}}O_2^2O_{61} +q^{\frac{1}{2}}O_{23}O_{45} +q^{-1}O_5O_{234} +q^{-\frac{3}{2}}(q^3-q^2+1)O_2O_{345}\\
&+(q-2)O_1O_6 -q^{-\frac{1}{4}}(q-1)O_{61},\\
g_{30}=&O_1O_5O_6O_{61} -q^{-\frac{1}{4}}O_5O_{61}^2 -q^{\frac{3}{4}}O_1O_6O_{234} -q^{-\frac{1}{4}}O_3O_4O_{345} +q^{\frac{1}{2}}O_{61}O_{234} +O_{34}O_{345} -q^{-\frac{1}{4}}O_1^2O_5 -q^{\frac{3}{4}}O_5O_6^2\\
&+q^{-1}O_4O_{45} +q^{\frac{3}{2}}O_6O_{56} +q^{-\frac{1}{4}}t^{-\frac{1}{2}}(q+t)O_1O_3 -q^{-\frac{5}{4}}(q-1)^2 (q+1)O_5,\\
g_{31}=&O_1O_2O_6O_{61} -q^{-\frac{3}{4}}O_2O_{61}^2 -q^{\frac{3}{4}}O_3O_4O_{234} -q^{-\frac{1}{4}}O_1O_6O_{345} +qO_{34}O_{234} +q^{-\frac{3}{2}}O_{61}O_{345} -q^{\frac{1}{4}}O_1^2O_2\\
&-q^{\frac{5}{4}}O_2O_6^2 +q^{\frac{1}{2}}O_1O_{12} +qO_3O_{23} +q^{\frac{1}{4}}t^{-\frac{1}{2}}(q+t)O_4O_6 -q^{-\frac{7}{4}}(q-1)^2O_2,\\
g_{32}=&O_4O_5O_6O_{56} -q^{\frac{1}{4}}O_4O_{56}^2 -q^{\frac{3}{4}}O_5O_6O_{123} -q^{-\frac{3}{4}}O_2O_3O_{234} +qO_{56}O_{123} +q^{-\frac{1}{2}}O_{23}O_{234} -q^{\frac{1}{4}}O_4O_5^2 -q^{-\frac{3}{4}}O_4O_6^2\\
&+q^{-\frac{3}{2}}O_3O_{34} +qO_5O_{45} +q^{-\frac{3}{4}}t^{-\frac{1}{2}}(q+t)O_2O_6 -q^{-\frac{7}{4}}(q-1)^2 (q+1)O_4,\\
g_{33}=&O_1O_5O_6O_{56} -q^{\frac{1}{4}}O_1O_{56}^2 -q^{\frac{3}{4}}O_2O_3O_{123} -q^{-\frac{3}{4}}O_5O_6O_{234} +qO_{23}O_{123} +q^{-\frac{1}{2}}O_{56}O_{234} -q^{\frac{1}{4}}O_1O_5^2 -q^{-\frac{3}{4}}O_1O_6^2\\
&+qO_2O_{12} +q^{-\frac{3}{2}}O_6O_{61} +q^{-\frac{3}{4}}t^{-\frac{1}{2}}(q+t)O_3O_5 -q^{-\frac{7}{4}}(q-1)^2 (q+1)O_1,\\
g_{34}=&O_4O_5O_6O_{45} -q^{-\frac{1}{4}}O_6O_{45}^2 -q^{-\frac{1}{4}}O_4O_5O_{123} -q^{\frac{1}{4}}O_1O_2O_{345} +q^{-1}O_{45}O_{123} +q^{\frac{1}{2}}O_{12}O_{345} -q^{\frac{3}{4}}O_4^2O_6 -q^{-\frac{1}{4}}O_5^2O_6\\
&+O_5O_{56} +q^{\frac{1}{2}}O_1O_{61} +q^{-\frac{1}{4}}t^{-\frac{1}{2}}(q+t)O_2O_4 -q^{-\frac{5}{4}}(q-1)^2O_6,\\
g_{35}=&O_1O_2O_6O_{45} -q^{\frac{3}{4}}O_6O_{12}O_{45} -q^{-\frac{1}{4}}O_2O_{45}O_{61} -q^{-\frac{1}{4}}O_1O_2O_{123} +O_{12}O_{123} +q^{-\frac{1}{2}}(q^2-q+1)O_{45}O_{345}\\ &-q^{\frac{3}{4}}O_3O_4^2+O_2O_{23} +q^{\frac{3}{2}}O_4O_{34} -2q^{-\frac{1}{4}}(q-1)^2O_3,\\
g_{36}=&O_5O_6^2O_{34} -O_2O_3^2O_{61} -q^{\frac{5}{4}}O_6O_{34}O_{56} +q^{\frac{3}{4}}O_3O_{23}O_{61} +q^{-\frac{5}{4}}O_3^2O_{345} -q^{-\frac{5}{4}}O_6^2O_{345} +q^{-\frac{3}{2}}(q^3-q^2+1)O_6O_{12}\\
&+(q-2)O_5O_{34} -q^{-\frac{3}{2}}O_3O_{45} -(q-2)O_2O_{61},\\
g_{37}=&O_1O_5O_6O_{34} -q^{\frac{3}{4}}O_1O_{34}O_{56} -q^{-\frac{1}{4}}O_5O_{34}O_{61} -q^{-\frac{3}{4}}O_1O_6O_{345} +q^{\frac{1}{2}}O_{34}O_{234} +q^{-1}O_{61}O_{345} -q^{\frac{3}{4}}O_2O_3^2\\
&+q^{-1}(q^2-q+1)O_1O_{12} +q^{\frac{3}{2}}O_3O_{23} -q^{-\frac{5}{4}}(q-1)^2 (q+1)O_2,\\
g_{38}=&O_1O_2O_6O_{34} -q^{\frac{3}{4}}O_6O_{12}O_{34} -q^{-\frac{3}{4}}O_2O_{34}O_{61} -q^{\frac{1}{4}}O_1O_6O_{234} +O_{61}O_{234} +q^{-\frac{3}{2}}(q^2-q+1)O_{34}O_{345} \\
&-q^{\frac{1}{4}}O_4^2O_5+q^{\frac{1}{2}}O_4O_{45} +qO_6O_{56} -q^{-\frac{7}{4}}(q-1)^2O_5,\\
g_{39}=&O_1^2O_2O_{34} -O_4^2O_5O_{61} -q^{\frac{1}{4}}O_1O_{12}O_{34} +q^{\frac{1}{4}}O_4O_{45}O_{61} -q^{\frac{1}{4}}O_1^2O_{234} +q^{\frac{1}{4}}O_4^2O_{234} -q^{\frac{1}{2}}O_4O_{23} -O_2O_{34}\\
& +q^{\frac{1}{2}}O_1O_{56} +O_5O_{61},\\
g_{40}=&O_1O_6^2O_{23} -qO_3^2O_4O_{56} +q^{\frac{5}{4}}O_3O_{34}O_{56} -q^{-\frac{5}{4}}O_6O_{23}O_{61} +q^{\frac{5}{4}}O_3^2O_{123} -q^{\frac{5}{4}}O_6^2O_{123} -q^{\frac{3}{2}}O_3O_{12}\\
&-q^{-1}(2 q-1)O_1O_{23} +q^{-\frac{3}{2}}(q^3-q+1)O_6O_{45} +qO_4O_{56},\\
g_{41}=&O_4O_5O_6O_{23} -q^{\frac{1}{4}}O_6O_{23}O_{45} -q^{\frac{1}{4}}O_4O_{23}O_{56} -q^{-\frac{3}{4}}O_5O_6O_{234} +O_{23}O_{123} +q^{-\frac{1}{2}}O_{56}O_{234} -q^{\frac{1}{4}}O_1O_2^2 \\
&+qO_2O_{12} +q^{-\frac{3}{2}}(q^2-q+1)O_6O_{61} -q^{-\frac{7}{4}}(q-1)^2 (q+1)O_1,\\
g_{42}=&O_1O_5O_6O_{23} -q^{\frac{1}{4}}O_1O_{23}O_{56} -q^{-\frac{3}{4}}O_5O_{23}O_{61} -q^{\frac{3}{4}}O_5O_6O_{123} +qO_{56}O_{123} +q^{-\frac{1}{2}}O_{23}O_{234} -q^{\frac{1}{4}}O_3^2O_4 \\
&+q^{\frac{1}{2}}O_3O_{34} +q^{-1}(q^2-q+1)O_5O_{45} -q^{-\frac{3}{4}}(q-1)^2O_4,\\
g_{43}=&O_4O_5^2O_{23} -O_1O_2^2O_{56} -q^{\frac{5}{4}}O_5O_{23}O_{45} +q^{\frac{3}{4}}O_2O_{12}O_{56} +q^{-\frac{5}{4}}O_2^2O_{234} -q^{-\frac{5}{4}}O_5^2O_{234} +(q-2)O_4O_{23} \\
&-q^{-\frac{3}{2}}O_2O_{34} -(q-2)O_1O_{56} +q^{-\frac{3}{2}}(q^3-q^2+1)O_5O_{61},\\
g_{44}=&O_5^2O_6O_{12} -O_2^2O_3O_{45} +q^{\frac{1}{4}}O_2O_{23}O_{45} -q^{-\frac{1}{4}}O_5O_{12}O_{56} +q^{\frac{1}{4}}O_2^2O_{345} -q^{\frac{1}{4}}O_5^2O_{345} -O_6O_{12} +q^{-\frac{1}{2}}O_5O_{34}\\
&+O_3O_{45} -q^{\frac{1}{2}}O_2O_{61},\\
g_{45}=&O_4O_5O_6O_{12} -q^{-\frac{1}{4}}O_6O_{12}O_{45} -q^{-\frac{1}{4}}O_4O_{12}O_{56} -q^{\frac{1}{4}}O_4O_5O_{345} +q^{-1}O_{12}O_{123} +q^{\frac{1}{2}}O_{45}O_{345} -q^{-\frac{1}{4}}O_2^2O_3,\\
&+O_2O_{23} +q^{-\frac{1}{2}}O_4O_{34} -q^{-\frac{5}{4}}(q-1)^2O_3\\
g_{46}=&O_3O_4^2O_{12} -q^{-1}O_1^2O_6O_{45} -q^{\frac{5}{4}}O_4O_{12}O_{34} +q^{-\frac{3}{4}}O_1O_{45}O_{61} +q^{-\frac{5}{4}}O_1^2O_{123} -q^{-\frac{5}{4}}O_4^2O_{123} +(q-2)O_3O_{12}\\
&-q^{-\frac{1}{2}}O_1O_{23} +q^{-1}O_6O_{45} +q^{-\frac{3}{2}}(q^3-q^2+1)O_4O_{56},\\
g_{47}=&O_1O_2O_6O_{345} -q^{\frac{3}{4}}O_6O_{12}O_{345} -q^{-\frac{3}{4}}O_2O_{61}O_{345} +q^{-1}(q^2-q+1)O_{345}^2 -q^{-\frac{1}{4}}O_4O_5O_{45} -q^{\frac{1}{4}}O_1O_6O_{61}\\
&+O_{45}^2 +O_{61}^2 +q^{-1}O_5^2 +qO_6^2 -q^{-1}t^{-1}(t+1) \left(q^2+t\right),\\
g_{48}=&O_1O_5O_6O_{234} -q^{\frac{1}{4}}O_1O_{56}O_{234} -q^{-\frac{1}{4}}O_5O_{61}O_{234} +O_{234}^2 -q^{\frac{3}{4}}O_5O_6O_{56} -q^{-\frac{3}{4}}O_1O_6O_{61} +qO_{56}^2 +q^{-1}O_{61}^2\\
&+q^{-1}O_1^2 -O_3^2 +qO_5^2 +O_6^2 -q^{-1}t^{-1}(t+1) \left(q^2+t\right),\\
g_{49}=&O_4O_5O_6O_{123} -q^{-\frac{1}{4}}O_6O_{45}O_{123} -q^{\frac{1}{4}}O_4O_{56}O_{123} +q^{-\frac{1}{2}}O_{123}^2 -q^{\frac{1}{4}}O_4O_5O_{45} -q^{-\frac{3}{4}}O_5O_6O_{56} +q^{\frac{1}{2}}O_{45}^2\\
&+q^{-\frac{1}{2}}O_{56}^2 -q^{-\frac{1}{2}}O_2^2 +q^{\frac{1}{2}}O_4^2 +q^{-\frac{1}{2}}O_5^2 +q^{-\frac{3}{2}}O_6^2 -q^{-\frac{3}{2}}t^{-1}(t+1) \left(q^2+t\right),\\
g_{50}=&O_1O_6O_{45}O_{61} -q^{\frac{1}{2}}O_3O_4^2O_{345} -q^{\frac{1}{4}}O_{45}O_{61}^2 +q^{\frac{5}{4}}O_4O_{34}O_{345} -q^{\frac{1}{4}}O_1O_6O_{23} -q^{-\frac{3}{4}}O_1^2O_{45} +q^{-\frac{3}{4}}O_4^2O_{45}\\
&-q^{\frac{1}{4}}O_6^2O_{45} +q^{-\frac{1}{2}}t^{-\frac{1}{2}}(q+t)O_1O_3O_4 +q^{\frac{1}{2}}O_{23}O_{61} +O_6O_{123} -q^{\frac{1}{2}}(q-2)O_3O_{345} -q^{-\frac{1}{4}}t^{-\frac{1}{2}}(q+t)O_1O_{34}\\
&-q^{-1}O_4O_5 +q^{-\frac{3}{4}}(q-1)O_{45},\\
g_{51}=&O_4O_5O_{45}O_{61} -O_1^2O_2O_{345} -q^{\frac{1}{4}}O_{45}^2O_{61} +q^{\frac{1}{4}}O_1O_{12}O_{345} -q^{\frac{1}{4}}O_4O_5O_{23} +q^{\frac{1}{4}}O_1^2O_{61} -q^{\frac{1}{4}}O_4^2O_{61} -q^{-\frac{3}{4}}O_5^2O_{61}\\ &+q^{-\frac{1}{2}}t^{-\frac{1}{2}}(q+t)O_1O_2O_4 +qO_{23}O_{45} +q^{-\frac{1}{2}}O_5O_{234} +O_2O_{345} -q^{-\frac{1}{4}}t^{-\frac{1}{2}}(q+t)O_4O_{12} -q^{\frac{1}{2}}O_1O_6 -q^{\frac{1}{4}}(q-1)O_{61},\\
g_{52}=&O_2O_3O_{45}O_{61} -q^{\frac{1}{4}}O_{23}O_{45}O_{61} -q^{-\frac{3}{4}}O_3O_{45}O_{345} -q^{-\frac{1}{4}}O_2O_{61}O_{345} -q^{\frac{1}{2}}O_5^2O_6^2 +q^{-\frac{1}{2}}O_{345}^2 +q^{\frac{5}{4}}O_5O_6O_{56}\\ &+q^{-\frac{3}{2}}O_{45}^2 +q^{\frac{1}{2}}O_{61}^2 -q^{-\frac{1}{2}}(q-1)O_2^2 -q^{\frac{1}{2}}(q-2)O_5^2 +q^{\frac{1}{2}}O_6^2 -q^{-\frac{3}{2}}t^{-1}(t+1) \left(q^2+t\right),\\
g_{53}=&O_1O_6O_{23}O_{61} -qO_3^2O_4O_{234} -q^{-\frac{3}{4}}O_{23}O_{61}^2 +q^{\frac{5}{4}}O_3O_{34}O_{234} -q^{\frac{1}{4}}O_1^2O_{23} +q^{\frac{5}{4}}O_3^2O_{23} -q^{\frac{5}{4}}O_6^2O_{23} -q^{-\frac{1}{4}}O_1O_6O_{45}\\ &+q^{\frac{1}{2}}t^{-\frac{1}{2}}(q+t)O_3O_4O_6 +q^{-1}O_{45}O_{61} +q^{\frac{1}{2}}O_1O_{123} +qO_4O_{234} -q^{\frac{3}{4}}t^{-\frac{1}{2}}(q+t)O_6O_{34}-q^{2}O_2O_3\\
&+q^{-\frac{3}{4}}(q-1)^2 (q+1)O_{23},\\
g_{54}=&O_4O_5O_{23}O_{61} -q^{\frac{3}{4}}O_{23}O_{45}O_{61} -q^{\frac{1}{4}}O_4O_{23}O_{234} -q^{-\frac{1}{4}}O_5O_{61}O_{234} -qO_1^2O_2^2 +O_{234}^2 +q^{-\frac{1}{4}}(q^2+q-1)O_4O_5O_{45}\\
&+q^{2}O_{12}^2 +qO_{23}^2 -(q^2+q-1)O_{45}^2 +q^{-1}(q^2-q+1)O_{61}^2 +2qO_1^2 +2qO_2^2 -(q^2+q-1)O_4^2\\
&-qO_5^2 -q^{-1}t^{-1}(t+1) \left(q^2+t\right),\\
g_{55}=&O_5O_6O_{34}O_{56} -O_2O_3^2O_{234} -q^{\frac{3}{4}}O_{34}O_{56}^2 +q^{\frac{3}{4}}O_3O_{23}O_{234} -q^{\frac{1}{4}}O_5O_6O_{12} +q^{-\frac{5}{4}}O_3^2O_{34} -q^{-\frac{1}{4}}O_5^2O_{34}\\
&-q^{-\frac{5}{4}}O_6^2O_{34} +q^{-1}t^{-\frac{1}{2}}(q+t)O_2O_3O_6 +qO_{12}O_{56} -(q-2)O_2O_{234} +q^{-\frac{1}{2}}O_5O_{345} -q^{-\frac{3}{4}}t^{-\frac{1}{2}}(q+t)O_6O_{23}\\
&-q^{-\frac{3}{2}}O_3O_4 +q^{-\frac{1}{4}}(q-1)O_{34},\\
g_{56}=&O_1O_2O_{34}O_{56} -q^{\frac{1}{4}}O_{12}O_{34}O_{56} -q^{-\frac{3}{4}}O_2O_{34}O_{234} -q^{-\frac{1}{4}}O_1O_{56}O_{234} -q^{\frac{1}{2}}O_4^2O_5^2 +q^{-\frac{1}{2}}O_{234}^2 +q^{\frac{5}{4}}O_4O_5O_{45}\\ &+q^{-\frac{3}{2}}O_{34}^2 +q^{\frac{1}{2}}O_{56}^2 -q^{-\frac{1}{2}}(q-1)O_1^2 -q^{\frac{1}{2}}(q-2)O_4^2 +q^{\frac{1}{2}}O_5^2 -q^{-\frac{3}{2}}t^{-1}(t+1) \left(q^2+t\right),\\
g_{57}=&O_5O_6O_{12}O_{56} -q^{\frac{1}{2}}O_2^2O_3O_{123} -q^{-\frac{1}{4}}O_{12}O_{56}^2 +q^{\frac{3}{4}}O_2O_{23}O_{123} +q^{\frac{3}{4}}O_2^2O_{12} -q^{\frac{3}{4}}O_5^2O_{12} -q^{-\frac{1}{4}}O_6^2O_{12}\\
&-q^{-\frac{1}{4}}O_5O_6O_{34} +t^{-\frac{1}{2}}(q+t)O_2O_3O_5 +q^{-\frac{1}{2}}O_{34}O_{56} +q^{\frac{1}{2}}O_3O_{123} +O_6O_{345} -q^{\frac{1}{4}}t^{-\frac{1}{2}}(q+t)O_5O_{23} -q^{\frac{3}{2}}O_1O_2\\ &+q^{-\frac{1}{4}}(q-1)^2O_{12},\\
g_{58}=&O_3O_4O_{12}O_{56} -q^{\frac{3}{4}}O_{12}O_{34}O_{56} -q^{\frac{1}{4}}O_3O_{12}O_{123} -q^{-\frac{1}{4}}O_4O_{56}O_{123} -q^{-1}O_1^2O_6^2 +O_{123}^2 +q^{-\frac{7}{4}}O_1O_6O_{61} +qO_{12}^2\\
&+q^{-1}(q^2-q+1)O_{56}^2 +q^{-2}(2 q-1)O_1^2 +q^{-1}(q-1)O_4^2 +q^{-1}O_6^2 -q^{-1}t^{-1}(t+1) \left(q^2+t\right),\\
g_{59}=&O_1O_6O_{23}O_{45} -q^{-\frac{1}{4}}O_{23}O_{45}O_{61} -q^{-\frac{1}{4}}O_1O_{23}O_{123} -q^{\frac{1}{4}}O_6O_{45}O_{123} -qO_3^2O_4^2 +O_{123}^2 +q^{\frac{5}{4}}O_1O_6O_{61} +O_{23}^2\\
&+q^{2}O_{34}^2 +q^{-1}(q^2-q+1)O_{45}^2 -qO_{61}^2 -qO_1^2 +2qO_3^2 +2qO_4^2 -(q^2+q-1)O_6^2 -q^{-1}t^{-1}(t+1) \left(q^2+t\right),\\
g_{60}=&O_4O_5O_{23}O_{45} -O_1O_2^2O_{123} -q^{\frac{3}{4}}O_{23}O_{45}^2 +q^{\frac{3}{4}}O_2O_{12}O_{123} +q^{-\frac{5}{4}}O_2^2O_{23} -q^{-\frac{1}{4}}O_4^2O_{23} -q^{-\frac{5}{4}}O_5^2O_{23} -q^{\frac{1}{4}}O_4O_5O_{61}\\ &+q^{-1}t^{-\frac{1}{2}}(q+t)O_1O_2O_5 +q^{\frac{3}{2}}O_{45}O_{61} -(q-2)O_1O_{123} +q^{-\frac{1}{2}}O_4O_{234} -q^{-\frac{3}{4}}t^{-\frac{1}{2}}(q+t)O_5O_{12} -q^{-\frac{3}{2}}O_2O_3\\
&-q^{-\frac{1}{4}}(q-1)^2O_{23},\\
g_{61}=&O_5O_6O_{12}O_{34} -q^{\frac{1}{4}}O_{12}O_{34}O_{56} -q^{-\frac{3}{4}}O_6O_{12}O_{345} -q^{-\frac{1}{4}}O_5O_{34}O_{345} -q^{\frac{1}{2}}O_2^2O_3^2 +q^{-\frac{1}{2}}O_{345}^2 +q^{\frac{5}{4}}O_5O_6O_{56}\\
&+q^{-\frac{3}{2}}(q^2-q+1)O_{12}^2 +q^{\frac{3}{2}}O_{23}^2 +q^{-\frac{1}{2}}O_{34}^2 -q^{\frac{3}{2}}O_{56}^2 +2q^{\frac{1}{2}}O_2^2 +2q^{\frac{1}{2}}O_3^2 -q^{-\frac{1}{2}}(q^2+q-1)O_5^2 -q^{\frac{1}{2}}O_6^2\\
&-q^{-\frac{3}{2}}t^{-1}(t+1) \left(q^2+t\right).
\end{align*}
\end{proposition}
\begin{proof}
Each of the elements above can be expressed as a two-sided linear combination of defining relations $c,\rho_i$ and $\eta_{I,J}$. Below we give an example of such calculation
\begin{equation}
\begin{aligned}
g_{9}=&O_1O_2O_3 -O_4O_5O_6 -q^{-\frac{1}{4}}O_3O_{12} -q^{\frac{1}{4}}O_1O_{23} +q^{-\frac{1}{4}}O_6O_{45} +q^{\frac{1}{4}}O_4O_{56}\\
=&\frac{t^{\frac{1}{2}}}{q+t}\rho _{14} -\frac{t^{\frac{1}{2}}}{q+t}\rho _{17} -q^{-\frac{1}{2}}O_2\eta _{(3)(1)} -q^{-\frac{1}{4}}\eta _{(2)(1)}O_3 +q^{-\frac{1}{2}}O_5\eta _{(6)(4)} +q^{-\frac{1}{4}}\eta _{(5)(4)}O_6\\
&-q^{-1}(q-1)\eta _{(45)(6)} +q^{-1}(q-1)\eta _{(12)(3)} -\frac{q^{-\frac{1}{2}}t^{\frac{1}{2}}}{q+t}\eta _{(5)(2)} -\frac{q^{\frac{1}{2}}t^{\frac{1}{2}}}{q+t}\eta _{(61)(34)}\\
=&\frac{t^{\frac{1}{2}}}{q+t}\Big(\rho_{14} -\rho_{17}\Big)+\mathbf O(q-1).
\end{aligned}
\label{eq:g9Decomposition}
\end{equation}
We omit a complete list of these routine calculations from the text and invite careful reader to examine our Mathematica program \cite{Arthamonov-GitHub-Flat} which generates all such decompositions.

\begin{remark}
An observation which is not used in our text, but which is still worth mentioning is that there is only a handful amount of factors that can appear in denominators of formulas like (\ref{eq:g9Decomposition}). The one we construct in our Mathematica program \cite{Arthamonov-GitHub-Flat} turns out to have the following Least Common Multiple of denominators
\begin{equation*}
\Lambda(q,t)=q^{\frac{3}{4}}t^{\frac{1}{2}}\Big(q^{\frac{1}{4}}+1\Big) \Big(q^{\frac{1}{2}}+1\Big) \Big(q^{\frac{1}{2}}-q^{\frac{1}{4}}+1\Big) \Big(q^{\frac{1}{2}}+q^{\frac{1}{4}}+1\Big) q^4 (q+1) \Big(q-q^{\frac{1}{2}}+1\Big) \Big(q^2-q+1\Big) (t+1)^2 (q+t)^2.
\end{equation*}
\end{remark}
\end{proof}

\bibliographystyle{alpha}
\bibliography{references}

\end{document}